\documentclass[11pt]{amsart}
\usepackage{enumerate,enumitem,amsmath,amssymb, amsthm, mathrsfs}
\usepackage{latexsym, amssymb,enumerate}
\usepackage{appendix}
\usepackage{color}
\usepackage{graphicx}
\usepackage{graphics}
\usepackage{tikz, tikz-cd}

\newcommand{\ga}{\alpha}
\newcommand{\gb}{\beta}
\newcommand{\gc}{\gamma}
\newcommand{\gd}{\delta}

\newcommand{\gra}{\nabla}

\newcommand{\bpf}{\begin{proof}}
\newcommand{\epf}{\end{proof}}
\newcommand{\beq}{\begin{equation}}
\newcommand{\eeq}{\end{equation}}
\newcommand{\beqn}{\begin{eqnarray*}}
\newcommand{\eeqn}{\end{eqnarray*}}

\newcommand\dist{\mathop{\rm dist}\nolimits}

\newcommand\tr{\mathop{\rm tr}\nolimits}

\newcommand\R{\mathbb{R}}

\DeclareMathOperator{\volume}{Vol}

\def\circledwedge{\setbox0=\hbox{$\bigcirc$}\relax \mathbin {\hbox
to0pt{\raise.5pt\hbox to\wd0{\hfil $\wedge$\hfil}\hss}\box0 }}

\def\H{{\mathbb H}}

\def\R{{\mathfrak R}}

\newtheorem{prop}{Proposition}[section]
\newtheorem{theo}[prop]{Theorem}
\newtheorem{lemm}[prop]{Lemma}
\newtheorem{coro}[prop]{Corollary}
\newtheorem{rema}[prop]{Remark}

\newtheorem{RK}{Remark}

\def\begeq{\begin{equation}}
\def\endeq{\end{equation}}
\def\p{\partial}

\def\R{\mathbb R}

\def\tr{{\rm tr}}

\def \ds{\displaystyle}

\def\S{\mathbb  {S}}

\def\odot{\setbox0=\hbox{$\bigcirc$}\relax \mathbin {\hbox to0pt{\raise.5pt\hbox to\wd0{\hfil $\wedge$\hfil}\hss}\box0 }}

\numberwithin{equation} {section}
\def\tilde{\widetilde}

\newtheorem{mylemm}{Lemma}
\newtheorem{mycoro}{Corollary}
\newtheorem{myrema}{Remark}
\numberwithin{mytheorem}{section}
\numberwithin{mylemm}{section}
\numberwithin{mycoro}{section}
\numberwithin{myrema}{section}

\makeatletter
\newcommand*{\greek}[1]{%
\ifcase#1\relax
\or I%
\or II%
\or III%
\else\@ctrerr \fi
}
\makeatother

\newtheorem{atheorem}{Theorem}

\begin{document}

\title["On the problem of filling by a Poincar\'e-Einstein metric in dimension 4" ]
{On the problem of filling by a Poincar\'e-Einstein metric in dimension 4 }

\author{Sun-Yung Alice Chang}
\address{Fine Hall, Washington Road, Department of Mathematics, Princeton University, Princeton, NJ 08544, USA}
\email{syachang@math.princeton.edu}

\author{Yuxin Ge}
\address{IMT,
Universit\'e Toulouse, \\118, route de Narbonne
31062 Toulouse, France}
\email{yge@math.univ-toulouse.fr}
\begin{abstract} 
Given a metric defined on a manifold of dimension three, we study the problem of finding a conformal filling by a Poincar\'e-Einstein 
metric on a manifold of dimension four.
We establish a compactness result for classes of conformally compact Einstein $4$-manifolds under conformally invariant conditions. 
A key step in the proof is a result of rigidity for hyperbolic space in the upper half space $\mathbb{R}^4_{+}$.

As an application, we also derive some existence results of conformal fill-in for metrics in the neighborhood of the canonical metric by the continuity method; when the conformal infinity is either $S^3$ or $S^1 \times S^2$.  
\end{abstract}

\thanks{Research of Chang is supported in part by Simon Foundation Travel Fund. Research of  Ge is  supported in part by the grant ANR-23-CE40-0010-02  of the French National Research Agency (ANR): Einstein constraints: past, present, and future (EINSTEIN-PPF). Part of the research was accomplished while both authors visited Di Giorgi Center in 2022, they wish to thank the center for the warm hospitality and generous support.}

\subjclass[2000]{}


\date{October 10, 2025}

 \maketitle
\setcounter{section}{-1}
\section{Introduction and statement of results}

\label{introduction}

Let $X^4$ be a smooth 4-manifold with boundary $\partial X$.
A smooth conformally compact metric $g^+$ on $X$ is a Riemannian metric such that $g = \rho^2 g^+$ extends smoothly to the 
closure $\overline{X}$ for a defining function $\rho$ of the boundary $\partial X$ in $X$. A defining function $\rho$ is a smooth 
nonnegative function on the closure $\overline{X}$ such that $\partial X = \{\rho=0\}$ and $d\rho\neq 0$ on $\partial X$. A 
conformally compact metric $g^+$ in $X$ is conformally compact Poincar\`e Einstein (which we denote by CCE)  if, in addition, 
$$
\operatorname{Ric}_{g^+} = - n g^+.
$$
One of the significant features of CCE manifolds $(X, \ g^+)$ is that the metric $g^+$ is ``canonically" associated with the conformal
structure $[\hat g]$ defined on the boundary $\partial X$, where $\hat g = g|_{T\partial X}$. $(\partial X, \ [\hat g])$ is called the conformal infinity of the conformal compact manifold $(X, \ g^+)$. There is great interest in the research fields of both Mathematics and Theoretical Physics to understand the correspondences between CCE manifolds $(X, \ g^+)$ and their conformal infinity $(\partial X, \ [\hat g])$, which becomes a central topic of research after the introduction of  AdS/CFT in Theoretical Physics by Maldacena \cite{Mald-1, Mald-2, Mald} and Witten \cite{Wi}.

There are many open questions in this field of research. One of the central ones is the problem of "existence" of conformal fill-in: Given a compact closed Riemannian manifold $(M^n, h)$, when does there exist some CCE manifold $X^{n+1}$ and metric $g^+$ such that $ h= \rho^2 g^+|_M $ for some defining function $\rho$  on $X^{n+1}$? We call such a metric $g^+$ a conformal CCE filling of $h$. Progress made on the problem includes the pioneering work of Graham-Lee \cite{GL}, where they showed that in the model example $(B^{n+1}, S^n, g^{H})$, $g^H $ being the hyperbolic metric; there exists a neighborhood $\mathfrak U$ of the canonical metric $h_c$ on $S^n$ so that any $h \in \mathfrak U$ 
allows a conformal filling $g^+$, which exists
in a neighborhood of $g^H$ on $B^{n+1}$. We remark that the existence result of Graham-Lee is a perturbation result, which was verified via the implicit function theorem.\\ 


On the other hand, there are also non-existence results (for example \cite{CGJQ,GH,GHS}).  In addition, there is also the famous non-uniqueness result of Hawking-Page \cite{HP}; where they showed that $(M= S^1 \times S^3,  h)$, where $h$ denotes a scale of the product metric, allows different CCE fillings by two topologically different spaces $X_i$'s, $i= 1, 2 $,  $X_1 =  S^1 \times B^3 $ and $X_2 = B^2 \times S^2$, and on $X_2$, there exist two CCE fill-in metrics (the Ads-Schwarzschild metrics) which are not diffeomorphic to each other.
For the general existence and uniqueness problems, we refer the reader to the introduction part of the recent article by Chang-Yang-Zhang \cite{CYZ} for a brief review.

In view of the non-uniqueness results, in this article we will first address the ``compactness'' aspect of the problem. Given a sequence
of CCE manifolds $(X^4, M^3, \{g_i^{+}\})$ with $M = \partial X$ and $\{ g_i \} = \{ \rho_i^2 g_i^{+} \} $ a sequence of compactified metrics, denote
${\hat g}_i = g_i |_M$,  assume that the sequence $\{{\hat g}_i\}$ forms a compact family of metrics in $M$, does it imply that some representative sequence ${\bar g_i} \in  [g_i]$
with $\{ {\bar g}_i |_M\} = \{{\hat g}_i\} $ also forms
a compact family of metrics in $X$?

We remark that the eventual goal of studying the compactness problem is to show the existence of conformal filling in for some classes of
Riemannian manifolds as conformal infinity. Indeed, as an application of our compactness result, we obtain some  existence results by the continuity method. 


One of the major difficulties in addressing the compactness issue is the existence of some ``non-local" curvature tensor term. To see this, 
first we recall (see \cite{G00})
that on a CCE manifold $(X^{n+1}, M^n, g^+)$, some special defining function $x$, which we call a geodesic defining function, exists, so that
so that $ | \nabla_{x^2 g^+} x | \equiv 1 $ in an asymptotic neighborhood $M \times [0, \epsilon)$ of $X$.  
The asymptotic behavior of the compactified metric $g = x^2 g^+$ in the special case when $n=3$ takes the form
$$  
 g:=x^2 g^{+} = h + g^{(2)} x^2 + g^{(3)} x^3 + g^{(4)} {x}^4 + \cdot \cdot \cdot \cdot 
$$
in an asymptotic neighborhood of $M \times (0, \epsilon)$. It turns out
$g^{(2)} = - \frac{1}{2}A_{\hat g} $, where $A_{\hat g} := \frac {1}{n-2} ( Ric_{\hat g} - \frac {1}{2(n-1)} R_{\hat g}) $ denotes the Schouten tensor, $Ric$ the Ricci tensor, and $R$ 
the scalar curvature respectively. Thus $g^{(2)}$ is  determined by $\hat g$ (we call such terms local terms), $Tr_{\hat g}  g^{(3)} =0 $, while 
$$  
g^{(3)} _{\alpha, \beta} = - \frac {1}{3} {\frac {\partial}{\partial n }{(Ric_{g})}_{\alpha, \beta} } $$
where $\alpha, \beta$ denotes the tangential coordinate on $M$, is a non-local term which is not determined by the boundary metric $\hat g$.   
We remark that $\hat g $, together with $g^{(3)}$, 
determine the asymptotic behavior of $g$ (\cite {Biquard, FG12}). From now on, we denote $S := -\frac{3}{2} g^{(3)}$; we remark that there is a study of some properties of $S$ in ( \cite[Lemma 2.1]{CG}).

Thus, if we view $S$ as third order Neumann data and $\hat g$ as the Dirichlet data defined on the conformal infinity, 
the compactness problem we are addressing amounts to show that in a CCE setting, 
when does the compactness of the Dirichlet data for a family of $\hat g_i$ 
imply the compactness of the corresponding family of the Neumann data $S_{g_i}$.  

Following the strategy in our earlier works (\cite{CG,CGQ}),
for a given $\hat g$ on conformal infinity $M$, we construct a parameter $s \in (\frac n2, n)$ of
compactification metrics $g_s$ with $g_s|_M = \hat g$ for each s, which we name as {\it adapted metrics};
among them for different choices of $s$, we have the scalar flat adapted metric, and
Fefferman-Graham metric. Properties of these metrics are discussed in detail in Section \ref{section2} below.

We now state our main theorem of the paper.

\begin{atheorem} 
\label{maintheorem}
Let $\{X,M=\p X, g_i^+\}$ be a family of $4$-dimensional oriented conformally compact Einstein metrics on X with conformal infinity $M$. Assume  
\begin{enumerate}

\item Suppose $\{{\hat g}_i\}$ is a sequence of metrics on M, and there exists some positive constant $C_0 > 0$ such that the Yamabe constant on the boundary is bounded below by $C_0$, i.e.
$$
Y(M,[{\hat g}_i]) \ge C_0. 
$$
Suppose also that the set 
$\{\hat g_i \}$ is compact in the $C^{k,\alpha}$ norm with $k\ge 6$ and has a positive scalar curvature.
\item There exists some positive constant $C_1 > 0$ such that
$$
\int |W_{g_i}|^2 \le C_1,
$$

\item  $H_2(X,\mathbb{Z})=0$ and $H^2(X,\mathbb{Z})=0$.
\end{enumerate}
Then the family of adapted scalar flat metrics $(X,g_i)$ is compact in the $C^{k,\alpha'}$ norm for any $\alpha'\in (0,\alpha)$ up to a diffeomorphism fixing the boundary.
\end{atheorem}

\vskip .1in

We now explain some differences between this version of the compactness result and the one in an earlier paper by Chang-Ge-Qing, 2020.

\begin{theo} (\cite[Theorem 1.1]{CGQ}) 
\label{maintheorem CGQ}
Suppose that $X$ is a smooth oriented 4-manifold with boundary $\partial X=\mathbb{S}^3$. Let $\{g_i^+\}$ be a set of conformally compact Einstein metrics on $X$. Assume the following conditions:

\begin{enumerate}
\item \label{Con:1}The set $\{\hat g_i\}$ of Yamabe metrics that represent the conformal infinities lies in a given set $\mathcal{C}$ 
of metrics that is of positive Yamabe type and compact in $C^{k+3}$ Cheeger-Gromov topology with $k \ge 2$.
\item \label{Con:2} For Fefferman-Graham's compactifications $\{g_i = \rho_i^2 g^+_i\}$ associated with the Yamabe representatives $\{\hat g_i\}$ 
on the boundary,
$$
\lim_{r\to 0}\sup_i\sup_{x\in \partial X}\oint_{B(x,r)} |S_i|[g_i] d\text{vol}[\hat g_i] =0
$$
\item \label{Con:3} $H_1(X, \mathbb{Z})= H_2(X, \mathbb{Z})=0$.
\end{enumerate}
Then, the set $\{g_i\}$ of Fefferman-Graham's compactifications (after diffeomorphisms that fix the boundary) 
is compact in the Cheeger-Gromov topology $C^{k+2,\alpha}$ for any $\alpha\in (0,1)$ .
\end{theo}


We remark that by a result of Wang-Zhou \cite{HW}, the compactness of the sequence of adapted scalar flat metrics is equivalent to the compactness of the sequence of Fefferman-Graham metrics with the same boundary metrics, see the statement of Theorem \ref{HWresult} below.

There are two major differences in the statements of the above two theorems. 
The first is that condition (2) in Theorem \ref{maintheorem} is a conformally invariant and scale-invariant condition on a manifold $X$ of dimension 4, that is, it depends only on the conformal class $[g]$ of $g$, thus is a natural condition for the problem. However,  condition (2) in Theorem \ref{maintheorem CGQ} is not
a conformally invariant condition and is hard to check. 
The second difference is that the topological condition (3) in the statement of Theorem \ref{maintheorem} can be checked to hold not only for the space 
$(X = B^4, M = S^3)$, but also for the space $(X= S^1 \times B^3, M = S^1 \times S^2)$. This allows us later to
derive existence results (e.g. Theorem \ref{maintheoremexistencebis1} below) for classes of metrics with conformal infinity both when $M=S^3$ and when $M = S^1 \times S^2$; while Theorem \ref{maintheorem CGQ} only applies when $M = S^3$.

The method of proof that we used to establish Theorem \ref{maintheorem} is via a contradiction argument.
A first step is to show that under conditions (1) and (2) of the theorem, the adapted scalar flat  metrics $\{ g_i \}$ have a uniform $C^{3} (X)$ bound. Assume the contrary, we then study the blow-up
Gromov-Hausdorff limit $g_{\infty}$ of some suitably re-scaled metrics of $g_i$, and verify
that $g_{\infty}$
satisfies all the conditions listed in the statement of Theorem \ref{aTheorem Liouville} below. The crucial step in the proof then is to establish Theorem \ref{aTheorem Liouville} which is a Liouville type rigidity result of the hyperbolic metric. 

\begin{atheorem}\label{aTheorem Liouville}
Let $(X_\infty,g_\infty^+)$ be a conformally $C^{2,\alpha}$ $4$-dimensional Poincar\'e Einstein metric with conformal infinity $(\R^3, dy^2)$. Denote $g_{\infty}$ the corresponding  adapted scalar flat metric. Assume $g_{\infty}= \rho^2 g_{\infty}^+$  satisfies the following conditions:
\begin{enumerate}
\item  $g_\infty$  has vanishing scalar curvature and bounded Riemann curvature with $|\nabla \rho|_{g_{\infty}} \le 1$.
\item There exists some $\varepsilon\in (0,\frac{3}{16})$ such that
$$
\lim_{a\to\infty} \frac{\inf_{y\in \partial B_{\R^3}(O,a)} d(y,O)}{a^{\frac{1}{1+\varepsilon}}}=+\infty
$$
\item There exists some positive constants $\delta\in (\frac{16\varepsilon}{3},1),C>0$ such that  for some fixed point $p$ we have 
$$|Ric[g](y)|\le C(1+d(y,p))^{-\delta}.$$
\item The interior and boundary injectivity radius of $g_\infty$ are positive.
\end{enumerate}
Then $g_{\infty}^+ $ is the hyperbolic metric, and the compactified metric $g_\infty$ is the flat metric on ${\mathbb R}^4_{+}$
\end{atheorem}
\vskip .2in

We note that under conditions (1) and (2) in the statement of Theorem \ref{maintheorem}, we actually show in Theorem \ref{curvaturedecay} that the limit metric $g_{\infty}$ satisfies a curvature decay condition 
$$|Ric[g](y)|\le C(1+d(y,p))^{-2 },$$
which is much stronger than the curvature decay condition (3) in Theorem \ref{aTheorem Liouville} 
above.
\vskip .2in

The proof of Theorem \ref{aTheorem Liouville} is the core of the paper. It is accomplished by establishing the conformal factor $\Psi$ between the metric compactified by the distance of the Einstein metric and the scalar flat adapted metric, which in the standard hyperbolic model is the identity function, see the definition of the function $\Psi$  in (\ref{scalar}),  when restricted to $\mathbb R^3$ is actually an extremal function of the Sobolev embedding of $H^1(\mathbb {R}^3) \hookrightarrow  L^6 (\mathbb {R}^3)$.
\vskip .2in

In the second part of the paper, we apply the compactness result and use the continuity method to derive some existence results. There are a number of existence results we obtain; here we will list the one with a clean statement as Theorem \ref{amaintheoremexistencebis1} below.


\begin{atheorem} 
\label{amaintheoremexistencebis1}
Let  $(X=B^4,M=\p X=\S^3)$ and $h$ be a metric on  $\S^3$ with the positive scalar curvature. Assume that $h$ is in $C^6$ and denote $h_c$ the canonical metric on $\S^3$.
Denote ${\bar C}_1:=(\|\hat{\nabla }\hat{C}\|_{3/2}+ \|\hat{\triangle }\hat{R}\|_{3/2})( \frac{2\sqrt{3}}{9} \|\hat{R}\|_{3/2}^{3/2}+ 16\pi^2)^{1/3} $.
There exist two positive dimension constants $a$ and $b$ such that if   
$$
\|h-h_c\|_{C^2}\le \min (a,\frac{b}{{\bar C}_1}),
$$
then we can find a CCE filling-in metric with the conformal infinity $[h]$ satisfying 
\beq
\label{eq10.1}
\frac{108}{5}\|W\|_2+ \frac{2}{\sqrt{3}}\|E\|_2\le \frac{Y_1}{3}.
\eeq
Moreover, such solution with the above bound is unique.
\end{atheorem}

We have another existence result Theorem \ref{maintheoremexistencebis1} when the boundary is either $\S^3$ or $\S^2\times \S^1$. We remark both existence results Theorem \ref{amaintheoremexistencebis1} and Theorem \ref{maintheoremexistencebis1}  are established by continuity method, not a perturbation argument at $h_c$.  The continuity hinges on some global curvature estimates of the metric on the entire $X$. \\

One of the crucial steps in the proof of Theorem \ref{amaintheoremexistencebis1} is to control the $L^3$ norm of $S$ tensor (condition (\ref{eq8.1bisbis}) in Lemma \ref{Lem8.3}) in terms of the metric $\hat g$  given on the boundary. We accomplish this by applying some PDE derived from a Kato inequality to the Weyl curvature of the Einstein manifold. This approach is inspired by earlier works of Bando-Kasue-Nagajima \cite{BKN} and Tian-Viaclovsky \cite{TV}.
\vskip .2in

There are also many interesting results in this topic by different groups of authors. (\cite{Anderson0, anderson1, anderson5, AKKLT, Biquard1, BiquardHerzlich} etc).

\vskip .2in

This paper is organized as follows. We break the writing into two parts.
Part I is for results related to compactness, and Part II is for results concerning uniqueness. In addition,
there are two Appendices, Appendix A and B, placed at the end of the paper. In Appendix A, we derive formulas of the asymptotic expansion up to the second order of the curvature tensor for some classes of metrics on 4-manifolds with boundary, which includes the compactification of CCE manifolds. We remark similar computations have been carried out before for metrics with total geodesic boundary (i.e., when the second fundamental form vanishes) in our earlier work \cite{CG}, but here due to our choice of the scalar flat adapted metric, we need to derive the asymptotic
formulas for metrics which are
umbilic (i.e., the second fundamental form is a nonzero constant) but not totally geodesic, thus the contribution of the mean curvatures terms and its derivatives come into the formulas.  Appendix B contains some integral formulas of the curvature tensors and its derivatives on 4-manifolds with boundary; again such formulas are known before for metrics without boundary or with totally geodesic boundary. The results in these appendices can be taken for granted in the first reading of the manuscript.

There are six sections in Part I. In Section \ref{SectionYamabeconstant}, we give quantitative estimates of the Yamabe constants for compactified metrics on CCE manifolds in terms of the Yamabe
constant of its boundary metric. This is a quantitative improvement of the result in \cite{Lee1}. In Section \ref{section2}, we introduce
the notion of adapted metrics, including as special cases the Fefferman-Graham adapted metrics and scalar flat adapted metrics, we discuss their properties and give justification 
(in Lemma \ref{boundedwyel}) of the reason why we make the choice of the scalar flat 
adapted metrics in the current article - mainly to get hold of the $L^2$ bound of the 
total curvature of the compactified metric under the assumption ((2) in Theorem \ref{maintheorem}) that its Weyl curvature is in $L^2$.
In Section \ref{sectionepsilonregularity}, we derive the $\varepsilon$-regularity result of the scalar flat adapted metrics, which is the key for us to gain regularity in our argument later. In Section \ref{section6}, we begin to run the method proof by contradication. That is, under assumption (later proved false) that the norm $C^3$ of the sequence of metrics $g_i$ (in Theorem \ref{maintheorem}) is not bounded, we study the limit of its rescaled metric $g_{\infty}$ and show the  decay property of curvature (in Theorem \ref{curvaturedecay}) of the metric. In Section \ref{Sect:Liouville}, we apply results in the previous section to establish the rigidity result Theorem \ref{aTheorem Liouville}. The arguments in this section are quite involved, with the complication partially due to 
the possible existence of cut-locus points. We first present an outline of the proof of Theorem \ref{aTheorem Liouville}
in Section \ref{section5.1}; then we provide the entire arguments of the proof 
in Section \ref{section5.2}. Finally in section \ref{section 7bis}, after establishing the lower estimates of the injectivity radius using an argument similar to that in \cite{CG,CGQ}, we apply the results in the previous sections to establish that there are no boundary blow-up points. We then apply condition (3) in the statement of Theorem \ref{maintheorem} to assert that there are also no interior blow up points; thus conclude the proof of Theorem \ref{maintheorem}.

There are 3 sections in Part II. In Section \ref{Section7}, we study the linearized operator of the Einstein equation, then apply the result in Lemma  \ref{localinvertible} by Wang \cite{FWang}
and independently by Gursky \cite{Gursky} to show that the operator is invertible under a $L^2$ pinching condition (\ref{weyl-pinching}) of the Weyl curvature. To establish the existence results, our strategy is to show that the condition (\ref{weyl-pinching}) once it is in place at a starting point continues to hold along a suitable path on the boundary metrics. Applying integral estimates of the curvature tensors, it turns out that we are able to do so once we are able to control the $L^3$ of the non-local tensor $S$, and we accomplish this in Section \ref{section8}. Finally, in Section \ref{section9}, we combine the arguments to establish a number of existence results, including Theorem \ref{amaintheoremexistencebis1}.

The authors have worked on this project for a long time. During this period, they have benefited from discussions with a long list of colleagues. They are grateful for their generosity and patience. In particular, they would like to thank Aaron Naber for a number of lengthy consultations, Olivier Biquard, Matt Gursky, Jie Qing, Paul Yang and Ruobing Zhang for many insightful comments.

\vskip .4in

\noindent \large PART I:  Compactness result.

\section{Study of the Yamabe invariants}
\label{SectionYamabeconstant}


We will start with a discussion of a second order conformal invariant called the Yamabe invariant. Recall on a $(M,\hat g)$ compact $n$-dimensional Riemannian manifold without boundary, the Yamabe invariant is defined as:
$$
Y(M,[\hat g]):=\inf_{\tilde g=u^{\frac{2n}{n-2}} \hat g\in [\hat g]}\frac{\ds\int_M R_{\tilde g}dvol_{\tilde g}}{(\ds\int_M dvol_{\tilde g})^{\frac{n-2}{n}}},
$$
where $R_{\tilde g}$ denotes the scalar curvature of the metric $\tilde g$.

We note that in terms of the conformal factor $u$ where $ \tilde g =  {u}^{\frac{4}{n-2}} \hat g$, $Y(M, [\hat g]) $ is equivalent to the constant of the Sobolev embedding of 
$ u \in W^{1,2} (M)$ to $ u \in L^{\frac{2n}{n-2} }(M)$.

We now recall the extension of the Yamabe invariant to compact manifolds with boundary.  On a compact $d$-dimensional Riemannian manifold $(X,g)$ with boundary, we denote $M= \p X$; M being of dimension $n = d-1$. When not in confusion, we also frequently denote $\hat g = g|_M$.  We consider the Yamabe energy functional
$$
Y(g):= \int_X R_{g}  dvol_{g}+2\int_{\p X}H_{g} dvol_{g|_M}, 
$$
where $R_{g}$ is the scalar curvature of the metric $g$ and $H_{g}$ denotes the mean curvature on the boundary $\p X$. 
\vskip .2in


\vskip .2in


As defined by Escobar \cite{Es1,  Es2}, the Yamabe constant of the first type is defined as
$$
Y_1(X,M,[g]):= \inf_{\tilde g\in [g]}\frac{Y(\tilde g)} { {vol(X,\tilde g)}^{(d-2)/d}},  
$$
and the Yamabe constant of the second type as 
$$
Y_2(X,M,[g]) := \inf_{\tilde g\in [g]}\frac{Y(\tilde g)} { {vol(\p X,\tilde g)}^{(d-2)/(d-1)}}.
$$

We note that the difference between these two conformal invariant constants is that the first one is defined by the normalization of the volume in the interior manifold $X$, while the second one is defined via the normalization of the volume on the boundary manifold $\p X =M$.  As before $Y_1$ and $Y_2$ each correspond to the (sharp) constants of the Sobolev and Sobolev trace embeddings of $W^{1,2}(X,M)$ into $L^{\frac{2d}{d-2}}(X)$ and into $L^{\frac{2n}{n-1}}(M)$ respectively; here on manifolds  we interpret the $W^{1,2}$ norm as the norm with respect to the conformal Laplacian $L_g = -\frac{4(d-1)}{d-2}\Delta_g +  R_g$  and with respect to the matching boundary operator,  $B = - \frac{2n}{n-1} \frac {\partial} { \partial  \nu}  + H_g u$, where $\nu$ denotes the inward normal. 

We now claim the following theorem, which is a quantitative estimate of the earlier result of J. Lee \cite{Lee95}, (see Theorem \ref {thm:Lee} quoted in Section \ref{section2} below), and is the main result of this section.

\begin{theo} 
\label{Yamabe}
Let  $\{X,M=\p X, g^+\}$ be a  $d$-dimensional oriented AHE manifold  of class $C^{3,\alpha}$ with boundary $\p X$ and $d\ge 4$.  Denote $g$ a compactified metric of $g^+$, assume  
$$
Y(M,[\hat g]) \ge C_0 >0 ;
$$
Then, 
\beq
\label{Ya-to-Y}
Y_1 (X, M,[g])\ge \left(\frac{4(d-1)}{d-2}\right)^{1/d}\left(\frac{d}{d-2}\right)^{(d-2)/d} (Y(M,[\hat g]))^{(d-1)/d}.
\eeq
\end{theo}

We remark that the inequality (\ref{Ya-to-Y}) does not hold in general compact manifold $(X, M, g)$ with boundary without under the special assumption that $g$ being a compactification metric of an asymptotic hyperbolic space.  One such example can be seen in the works of Gursky-Han \cite{GH} and Gursky-Han-Stolz \cite{GHS}, where they have shown among other examples, that there exists some 
metric $h$ on $S^{4k-1}$ where $k \geq 2$ with a positive Yamabe constant $Y(M, [h])$, such that any extension $g$ of $h$ on the ball $B^{4k}$ must have non-positive Yamabe constant $Y_1 (B^{4k}, S^{4k-1}, g)$. The
inequality of the type (\ref{Ya-to-Y}) cannot possibly be true.

\vskip .1in

To prove the theorem above, we first recall several earlier relevant results  due to Escobar, Gursky-Han and  Chen-Lai-Wang \cite{CLW, Es1, Es2,GH} (see also  Wang-Wang  \cite{WW}).

\begin{lemm} [Escobar \cite{Es1, Es2}]
\label{Yamabe0}
Let  $ (X, g )$ be a  smooth $d$-dimensional compact Riemannian manifold  with boundary $M:=\p X$ and $d\ge 4$. Assume  
$$ Y_1 (X, M,[g]) <Y_1 (\S^d_+,{\S^{d-1}},[g_c])$$
where $g_c$ is the standard metric on the sphere.  Then $Y_1(X, M,[g])$ is achieved by  a minimizing Escobar-Yamabe metric of type 1, that is, a metric with constant scalar curvature and free mean-curvature. 
Also, if  
$$ Y_2(X, M,[g]) <Y_2(\S^d_+,{\S^{d-1}},[g_c]),$$ then $Y_2(X, M,[g])$ is achieved by  a minimizing Escobar Yamabe metric of type 2, that is, a metric with  vanishing scalar curvature and constant mean-curvature.
\end{lemm}

\begin{lemm} [Gursky-Han \cite{GH}]
\label{Yamabe1}
Let  $\{X,M=\p X, g^+\}$ be     a  $d$-dimensional oriented AH manifold of class $C^{3,\alpha}$ with boundary $\p X$ and $d\ge 4$.  Assume $Y_1(X, M,[g])$ is achieved by  a minimizing Escobar Yamabe metric of type 1, 
 Then, 
\beq
\label{YaGH} 
\begin{array}{ll}
\ds Y_1 (X, M,[g]) I^2\ge  \frac{d}{d-2} Y(M,[g|_M])& \mbox{ when }d\ge 4, \\
\ds Y_1(X, M,[g]) I^2\ge 12\pi\chi(M) & \mbox{ when }d= 3.
\end{array}
\eeq
where  $I:=\ds\frac{vol(M)^{\frac{1}{d-1}}}{vol(\bar X)^{\frac{1}{d}}}$ is the  isoperimetric ratio for a minimizing Escobar Yamabe metric of type 1, that is, a metric with constant scalar curvature and zero mean curvature.
\end{lemm}

In the meanwhile, Chen-Lai-Wang \cite{CLW} has established the analogous estimates of
(\ref{Ya-to-Y})  for the $Y_2$ functional. 

\begin{lemm} [Chen-Lai-Wang \cite{CLW} (see also Wang-Wang \cite{WW})]
\label{Yamabe2}
Let  $\{X,M=\p X, g^+\}$ be a  $d$-dimensional oriented AH manifolds  of class $C^{3,\alpha}$ with boundary $\p X$ and $d\ge 4$. Assume   
Yamabe constant on the boundary is non-negative
$$
Y(M,[g|_M]) \ge 0.
$$
Then, 
\beq
\label{Yb-to-Y}
\begin{array}{ll}
\ds  \frac{d-2}{4(d-1)} Y_2(X, M,[g])^2\ge  Y(M,[g|_M]). & 
\end{array}
\eeq

\end{lemm}
\begin{rema} 
We remark that in \cite{CLW} the authors have also established the sharp constant in the inequality (\ref{Yb-to-Y}), with the equality holding only when the Einstein metric $g^+$ being the hyperbolic metric on the ball $X = B^d$.
\end{rema}
\vskip .1in

We now combine the results in Lemmas \ref{Yamabe1} and \ref {Yamabe2}  above to prove  Theorem \ref{Yamabe}.

\vskip .1in

\begin{proof}[Proof of Theorem \ref{Yamabe}] 
 We first observe that when $Y_1 (X, M,[g]) =Y_1 (\S^d_+,{\S^{d-1}},[g_s])$,  the inequality (\ref{Ya-to-Y})  holds trivially.
Thus we may suppose $Y_1(X, M,[g]) <Y_1 (\S^d_+,{\S^{d-1}},[g_s])$. By Lemma \ref{Yamabe0}, let $\tilde g$ be a minimizing Escobar Yamabe metric of type 1. By scaling arguments, we may suppose
$$
R_{\tilde g}\equiv 1\mbox{ in } \bar X
$$
and
$$
H_{\tilde g}\equiv 0\mbox{ on } M
$$
Hence $Y_1(X, M,[g])=(vol(\bar X, \tilde g))^{2/d}$. Applying Lemma \ref{Yamabe1}, we have
$$
{vol(M, \tilde g|_M)^{\frac{2}{d-1}}}= Y_1 (X, M,[g])\frac{vol(M, \tilde g|_M)^{\frac{2}{d-1}}}{vol(\bar X,  \tilde g)^{\frac{2}{d}}} \ge  \frac{d}{d-2} Y(M,[g|_M]).
$$
Now applying Lemma \ref{Yamabe2} , we have
\begin{align}
\frac{vol(\bar X,  \tilde g)}{vol(M, \tilde g|_M)^{\frac{d-2}{d-1}}}= \frac{Y(\tilde g)}{vol(M, \tilde g|_M)^{\frac{d-2}{d-1}}}\\
\ge  Y_2(X, M,[g])\ge \sqrt{ \frac{4(d-1)}{d-2}  Y(M,[g|_M])}.
\end{align}
Thus we have
$$
vol(\bar X,  \tilde g)^{2/(d-2)}\ge \frac{d}{d-2} (\frac{4(d-1)}{d-2} )^{1/(d-2)} ( Y(M,[g|_M]))^{(d-1)/(d-2)},
$$
and
\begin{align}
Y_1(X, M,[g])=(vol(\bar X, \tilde g))^{2/d}\\
\ge \left(\frac{4(d-1)}{d-2}\right)^{1/d}\left(\frac{d}{d-2}\right)^{(d-2)/d} (Y(M,[g|_M]))^{(d-1)/d},
\end{align}
which establishes the inequality (\ref{Ya-to-Y}) as we want.
\end{proof}
\begin{rema} 
The constant we obtained in Theorem \ref{Yamabe} is not sharp.
\end{rema}

\vskip .5in

\section{Adapted metric, Scalar flat adapted metric}
\label{section2}
\vskip .2in

\subsection{Adapted metric}

\vskip .1in

In this section, we will introduce a special class of metrics we named as the {\it {Adapted metrics}}  among the class of compactified metrics on a CCE manifold. To motivate the introduction of such a class of metrics,  we first very briefly describe the work of Graham-Zworski \cite{GZ} where they related the scattering operators on the conformal infinity of a CCE manifold to that of the class of GJMS operators.

 Let $(X,g^+)$ be a CCE manifold with conformal infinity $(M,[h])$.  We would like to find solutions to the Poisson equation: 
 \begin{equation}
\label{scattering1} -\Delta_{g_+} u-s(n-s)u=0\text{ in }X
 \end{equation}
in the class $u\in L^2(X)$.

To begin with, the fundamental work of Mazzeo-Melrose \cite{MM} showed that the equation \eqref{scattering1}  is solvable, except maybe at a finite number of $\lambda:= s (n-s)$ in $(0, \frac{n^2}{4}]$ which are point spectrum. In addition, all points $\lambda$ in  $[\frac{n^2}{4},\infty)$ are essential spectrum.

Thus, except for a finite number of $s$ with $ \Re(s) >\frac{n}{2}$, both $x^{n-s}$ and $x^s$ are asymptotic solutions of the equation (\ref{scattering1}), where ${x}$
denotes the geodesic defining function with respect to a given boundary metric $h$. In general, given $\lambda = s (n-s) $ not in the point spectrum and for every given data $f$ on $C^{\infty}(\partial X)$, there exists a solution $u$ of \eqref{scattering1} of the form
\begin{equation*}
\begin{cases}
u=x^{n-s}F+ x^s G\quad \text{if }s\not\in\tfrac{n}{2}+\tfrac{\mathbb N}{2},\\
u=x^{n-s}F+ x^s \, log x\,  G\quad \text{if }s=\tfrac{n}{2}+\tfrac{\ell}{2}, \ell\in\mathbb N,
\end{cases}
\end{equation*}
that satisfies
\begin{equation*}
x^{n-s} u|_{\partial X}=f.
\end{equation*}
Note that $F,G$ have the following asymptotic expansion as $ x \rightarrow 0$:
\begin{equation*}
\begin{split}
&F=f+f_2 x ^2+\sum_{l=2}^\infty f_{2l}x^{2l},\\
&G=k+k_2 x^2+\ldots.
\end{split}
\end{equation*}
 and $f,k\in L^\infty(\partial X)$.\\

 In \cite{GZ}, the scattering matrix the operator $\mathcal S $  defined as 
 \begin{equation}{\label{scattering}}
\mathcal S(s)f=G|_{\partial X},
 \end{equation}
 i.e.,
 \begin{equation*}
 \mathcal S(\tfrac{n}{2}+\gamma)f=G|_{\partial X},
 \end{equation*}
 The operator is meromorphic for    $Re(s)>\frac{n}{2}$   and it has simple poles at $s=\frac{n}{2}+\mathbb N$.
We call $f$ the Dirichlet data of equation \eqref{scattering1}, and $ k = G|_M$ the scattering data.

Denote $ s = \frac{n}{2} + \gamma$ for some $ \gamma \in (0, \frac{n}{2}) $, one of the main results in \cite{GZ} is to relate
the scattering operator $\mathcal S$ at $s$ to that of the  GJMS {\cite{GJMS}} operator $P_{ 2 \gamma} $ 
$$ P_{2 \gamma} f := \mathcal S(\frac{n}{2}+\gamma)f=G|_{\partial X}$$
with the conformal covariant property:
$$ (P_{2\gamma})_{h_{\phi}}  (\bullet ) =  e^{- \frac{n+2 \gamma }{2} \phi} (P_{2\gamma})_{h} (e ^{{\frac {n- 2 \gamma} {2}}  \phi}  \bullet ),$$ 
where $h_{\phi} = e^{ 2 \phi} h$ is defined on the conformal infinity $M^n$.

It turns out that when $ s = n+1$, the solution of equation \eqref{scattering1} was studied earlier in the work of J. Lee \cite{Lee95}.  In fact, in the special case when  Dirichlet data $ f \equiv 1$, Lee has chosen the function $\phi $, where  $ \phi = v^{ - \frac{n}{2} }$, where $v$ is the corresponding solution of the \eqref{scattering1}, as a test function to derive the following result concerning the spectrum of $\Delta_{g+}$, which plays a crucial role in this paper.

\begin{theo}
 \cite{Lee95}
\label{thm:Lee}
On the CCE setting, if $\hat R$ := scalar metric on $\partial X$, is positive; then
\begin{equation*}
\lambda_1(-\Delta_{g_+})\geq \tfrac{n^2}{4},
\end{equation*}
i.e., there is no point spectrum on $(0,\frac{n^2}{4})$.
\end{theo}
 
In the result of Lee above, we have $v >0$ on $X$ (see Lemma \ref{case:lemma} below for a more general statement), we define the compactified metric $ g_L$  as 
$$ r =v^{-1},\quad g_L =v^{-2}g_+,$$ which we named Lee's metric. 

An important consequence of Lee's result as pointed out by J. Qing \cite{CQY,ChangQingYang2006bis1} and \cite{Q}, is the observation that the properties derived by Lee for the solution $v$  also provide a positive sign of the scalar curvature of the metric $g_L$, namely:

\begin{theo}\label{Qing}  Under the assumption that the scalar curvature of the boundary metric ${\hat {g_L}}:= g_L|_M $ is positive on $\partial X$, we have $R_{g_L}>0$ on $X$.  
\end{theo}
 
Since then, intensive work has been done studying solutions of equation \eqref{scattering1} for different values of $s$, with applications to various  problems in conformal geometry, by different groups of authors. (\cite{casechang, Chang-Gonzalez, CYa, FG02,Gonzalez-Mazzeo+ 2others,Gonazalez-Qing,Sanghoon Lee, Li, R.Yang} etc.)

We also refer the readers to the more recent lecture notes \cite{Chang-Granada} by the first author of this article, which are the lecture notes of a mini-course she gave in the summer school at Granada Spain in 2023, for a survey of the study of the solutions of equation \eqref{scattering1} for different values of $s$ and their geometric applications.
    
To introduce the terminology of adapted metric, first defined in  \cite[section 6]{casechang}, we first cite a Lemma.

 \begin{lemm}\cite{casechang}\label{case:lemma}
If $\lambda_1(-\Delta_{g^+})>\frac{n^2}{4}-\gamma^2$, let $v_s$ be the solution to \eqref{scattering1} when $f\equiv 1$, which is of the form
\begin{equation*}
v_s=x^{n-s}(1+f_2 x^2+\ldots)+ x^sG,
\end{equation*}
then
\begin{equation*}
v_s>0 \text{ on } \, X.
\end{equation*}
 \end{lemm}

\vskip .1in

\noindent Definition: Denote $\rho_s=v_s^{\frac{1}{n-s}}$. The metric $g_s=(\rho_s)^2 g^+$ is called the family of {\emph{ adapted metric} } of $g^+$.
This class of metric was first introduced in \cite[section 6]{casechang} to study the positivity of GJMS operators.

We note that in the special case where $ s = n+1 $, the adapted metric $ g_{n+1}$ is J. Lee metric $g_L$.

We also remark that the above definition of adapted metrics, originally defined for $\Re(s)>\frac{n}{2}$  and $ s \not = n$, can be defined across $ s =n$ via a meromorphic extension as in the work of Fefferman-Graham \cite{FG02}, see the more precise description of this metric (which we call FG compactification)  in Section \ref{section 3.2} below. The metric has many interesting properties and has been applied \cite{FG02} to study the renormalized volume of CCE manifolds (see also  related work \cite{CQY, ChangQingYang2006bis1}).  

We also would like to mention the class of adapted metrics $g_s$ in general enjoys many good properties for all $ s>\frac{n}{2} + \frac{1}{2} $, among them, we cite the following result. 
From now on, we denote the boundary metric of a compactified metric $g$ by $\hat g := g|_M $, and the scalar curvature of $\hat g$ by $\hat R$
\begin{prop}\label{propCC}\cite[Proposition 6.4]{casechang}
Let  $(X^{n+1},M^n ,g^+)$ be a CCE manifold with $ {\hat R}_{\hat g_s} >0$ on $\partial X$. Then,  the metrics $g_s$ are geodesic all $s>\frac{n}{2} + \frac{1}{2} $  and 
\begin{equation*}
R_{g_s}>0,\quad\text{for} \quad s\in\left(\tfrac{n}{2}+\tfrac{1}{2}, n+1\right].
\end{equation*}
\end{prop}

We remark that in earlier works \cite{CG,CGQ}, the adapted metric we have chosen to study the compactness problem of CCE manifolds is the Fefferman-Graham metric, which among other things, is totally geodesic. But in the current paper, we instead shifted our choice to the one adapted metric corresponding to $ s = \frac{n}{2} + \frac{1}{2}.$, the adapted metric $g_{s = \frac n2 + \frac12} $ that we now call {\it{Scalar flat Adapted metric}}. This adapted metric is umbilic but not totally geodesic. In Section \ref{section 3.3} below, we will describe the main properties of this class of adapted scalar flat metric and justify why we make such a choice. \\  

In Section \ref{section 3.2} below, we will also briefly define and describe some properties of the Fefferman-Graham metric, which we will also use later in the proof of our main compactness result Theorem \ref{maintheorem} in Section \ref{Section7}, chiefly to help us to gain regularity of the metrics.

\vskip .2in

\subsection{Fefferman-Graham compactification}
\label{section 3.2}

\vskip .1in
 As we have mentioned before, the scattering matrix $S (s) $ has a simple pole at $s =n$, at this point the corresponding Poisson equation (\ref{scattering1}) has a trivial kernel $v_s \equiv 1$ when Dirichlet data $f \equiv 1$, so the compactified metric is not well defined at this value of $s$. Instead, Fefferman and Graham \cite{FG02}  considered a meromorphic extension of the solution of the Poisson equation when $s<n$ and considered the limit of the solution when $s\to n$. More precisely, for the solution $v_s$ of the Poisson equation (\ref{scattering1})

One defines
\begin{equation*}
w=-\left.\frac{d}{ds}\right|_{s=n}v_s
\end{equation*}
Then $w$  satisfies the equation
\begin{equation}\label{memoPoisson}
-\Delta_{g^+} w=n\quad\text{on }X^{n+1},
\end{equation}
and  which has the expansion
which holds when $n$ is odd,
\begin{equation*}
w=\log x+A+Bx^n.
\end{equation*}
The metric $ g_{FG} : = g_{s=n}=e^{2w}g_+$ is the Fefferman-Graham \cite{FG02} metric. We remark a key property as it turns out which $g_{FG}$ satisfies is:
$$( Q_{n+1})^{(n+1)}_{g_{FG}}\equiv 0,$$ 
when $ s =\frac{n}{2} + \frac{n}{2}$ for all n when n is odd, where $(Q_{n+1})^{(n+1)}$ denotes the (n+1)-th order Q-curvature in $X^{n+1}$ in conformal geometry \cite [Lemma 2.1]{CQY}. In the special case  where $n=3$, $(Q_4)^{(4)}$ is precisely Branson's \cite{Branson} $Q$-curvature in 4 dimensions, which  has the explicit expression
\begin{equation*}
6Q_4^{(4)}[g] =-\Delta_g R+R^2-3|Ric|^2=-\Delta_g R+\tfrac{1}{4}R^2-3|E|^2,
\end{equation*}
where $E$ is the traceless Ricci tensor.

In summary, when $n=3$ and $s=n=3$, Fefferman-Graham metric $g=g_{FG} $,
$\hat g = g|_M$, the metrics satisfy:
\beq
\label{Qcurvature}
\Delta_g R=\tfrac{1}{4}R^2-3|E|^2, \,\,\, 
and \,\,\, 
R = 3 \hat R.
\eeq

Since equation ($\ref{Qcurvature}$) is an elliptic equation w.r.t. the scalar curvature, another good property of $g_{FG}$ is that one gains the regularity of the metric once it is in $C^{l, \alpha} (X)$
when $l \geq 2$. We state this property as a corollary for later usage. 

\begin{coro} \label{gain of FG}
On a CCE manifold $(X^4, M, g^+) $, suppose the metric $g_{FG} \in 
C^{l, \alpha}(X)$, and ${\hat g}_{FG} \in C^{k, \alpha} (M)$ for some $ k >l >2 $, $0< \alpha < 1 $ then $g_{FG} \in C^{k, \alpha'} (X)$, and for all $0< \alpha' <\alpha.$
\end{coro}

\vskip .1in

\subsection {Scalar flat Adapted metric}
\label{section 3.3}

 In this subsection, we will make a special choice among the class of adapted metrics, a choice which we will use throughout the rest of this paper. This choice corresponds to the adapted metric $g_s$  where $ s= \frac{n}{2} + \frac{1}{2}$. We name the metric as the scalar flat adapted metric (due to  property (1) in Lemma \ref{freescalar} below.)

Let  $\{X,M=\p X, g^+\}$ be a  $4$-dimensional oriented conformally compact Einstein  manifold with boundary $\p X$ with the positive Yamabe invariant on conformal infinity. We consider a special case of the  Poisson equation (\ref{scattering1}) corresponding to the index $ s = \frac{n}{2} + \frac{1}{2} $ with the Dirichlet data  $ f \equiv 1$,  
\beq
\label{Poisson}
-\triangle_+ v  - \frac {n^2 -1}{4} v =0, 
\eeq
and define the compactification of $g^+$
$$
g= \rho^2 g_+ \,\,\,\,\, where \,\,\,  \rho = v^{ \frac{2}{n-1}} .
$$
We call such choice of $g$ a {scalar flat adapted metric}.

\begin{lemm} 
\label{freescalar}
Let  $\{X,M=\p X, g^+\}$ be a  $(n+1)$-dimensional oriented conformally compact Einstein  manifold of order $C^2$ with boundary $\p X$ with the positive Yamabe invariant on conformal infinity. Assume the scalar curvature of the restriction of the adapted metric $g$ on the boundary is positive.  Then  
\begin{enumerate}
\item The adapted metric $g$ has vanishing scalar curvature 
$$
R_{g}\equiv 0 \mbox{ in } \bar X
$$
\item The mean curvature $H$ on the boundary with respect to inwards normal vector is positive
$$
H_{g}> 0\mbox{ on }\p X. 
$$
\item There holds
$$
|\nabla \rho|^2_g\le 1  \mbox{ in }  X.
$$
\item We have
\beq
\label{eq2.2add}
H_g \ge \sqrt{\frac{n}{n-1}\min_M \hat{R}_{\hat g}}.
\eeq
Moreover, the equality \eqref{eq2.2add} holds if and only if $X$ is the hyperbolic space.
\end{enumerate}
\end{lemm}
\begin{proof}
We consider the conformal Laplace operator on $(n+1)$-dimensional manifold $X$, 
$$
L_g:=-\frac{4n}{n-1}\triangle_g+ R_{g}.
$$
Recall the conformal invariance property of $L$, 
$$
L_g(\rho^{-\frac{n-1}{2}}\varphi)=\rho^{-\frac{n+3}{2}} L_{g^+}\varphi.
$$
Hence $R_{g}=  L_g(1)= \rho^{-\frac{n+3}{2}} L_{g^+} \rho^{\frac{n-1}{2}}=\rho^{-\frac{n+3}{2}} \frac{4n}{n-1}(\triangle_+ +\frac{n^2-1}{4})v=0$. Thus, we have established (1). \\ 
Let $g_{n+1}=g_L$ be Lee metric as above. Then the scalar curvature $R[g_L]>0$ in $\bar X$ and the mean curvature $H_{g_L}=0$ on $\p X$. We write $g=u^{\frac{4}{n-1}}g_L$ and $u=1$ on the boundary $\p X$. 
Then $L_{g_L}u=u^{\frac{n+3}{n-1}}R_{g}=0$ and $L_{g_L}(u-1)= -L_{g_L}(1)=-R_{g_L}<0$. Applying  the strong maximum principle, we deduce  $u-1<0$ on $X$ . Applying the Hopf's Lemma, we have  $-\frac{\p u}{\p \nu }>0$ on $\p X$ for inward normal $\nu$, which yields 
$$H_{g}=u^{-\frac{n+1}{n-1}}( -\frac{2n}{n-1}\frac{\p u}{\p \nu }+H_{g_L}u)=-\frac{2n}{n-1}\frac{\p u}{\p \nu } >0\mbox{ on }\p X.
$$
Thus we have established (2).\\
To see (3),  following a similar argument in \cite{CLW}, we consider the quantity: 
\beq
\label{Pfunction}
P=\frac{1-|d \rho|_g^2}{\rho}.
\eeq
Direct calculations lead to $P|_{\p X}=\frac2{n} H_g>0$ on the boundary  
and
\beq
\label{laplacianP}
\triangle_g P=-\frac{2}{(n-1)^2} \rho |Ric-\frac1{n+1} Rg|_g^2\le 0 \mbox{ in } X.
\eeq
It follows from the Maximum Principle that $P\ge 0$ in $X$, that is, $|\nabla \rho|_g\le 1$. Thus we established assertion (3). \\
For the last statement, we know
$$
P=\frac{2}{n}H_g+(\frac1nH^2_g-\frac{1}{n-1}\hat{R}_h)\tilde r +o(\tilde r),
$$
where $\tilde r$ is the geodesic distance function to the boundary with respect to the adapted metric $g$. From the strong Maximum Principle that $P> \frac{2}{n}\min_M H_g$ in $X$ or $P\equiv \frac{2}{n}\min_M H_g$ in $X$. In the latter case, $X$ is the hyperbolic space. In the first case, it follows from the Hopf Lemma,
$$
\frac{\p P}{\p \nu}>0
$$
at the minimal point $y$ of $H$ on the boundary, that is, 
$$
\frac{1}{n}\min_M H_g^2=\frac{1}{n}H^2_g(y)>\frac{1}{n-1}\hat{R}_{\hat g}(y)\ge \frac{1}{n-1}\min_M\hat{R}_{\hat g}.
$$
Hence we have established the inequality (\ref{eq2.2add}). It is clear the equality for this inequality holds if and only if $H$, $\hat{R}_{\hat g}$ and $P$ are constants. Hence, the equality case related to $Y_2(X, M, [g]) $ and the boundary Yamabe constant $Y(M, [\hat g])$ in \cite[section4]{CLW} holds, that is, $g^+$ is hyperbolic space. We have thus finished the proof of the lemma.
\end{proof}

\vskip .1in
We now give the reason why we choose the scalar flat adapted metric as our preferred representative for the class of compactified metric. The main reason  being that for this choice of metric, we can derive a $L^2$ bound of the Ricci curvature together with the $L^3$  bound of the mean curvature. In the meanwhile, we also have the mean curvature $H$ of the metric is positive under the assumption that the scalar curvature of the boundary metric is positive. To see so, we will first recall the Gauss-Bonnet formula for 4 manifolds with boundary.

Next we  recall an non-local curvature tensors of order three--namely the $T$ curvature  defined  on the boundary of any compact four dimensional manifolds (see \cite{CQ1,CQY, ChangQingYang2006bis1}). 
On a four dimensional manifold $(X,\p X,g)$ with boundary, the
 $Q$-curvature on $X$  is given by
 \beq
Q=-\frac{1}{6}\triangle R-\frac{1}{2}|E|^2+ \frac{1}{24}R^2,
\eeq
 and the $T $-curvature on the boundary $\p X$ is given by
\beq
T=-\frac{1}{12}\frac{\p R}{\p \nu}+\frac{1}{6}H R- \frac{H}{3}R_{00}+\frac{1}{9}H^3- \frac{1}{3}tr(L^3)- \frac{1}{3}\hat{\triangle}H, 
\eeq 
where $\nu$ denotes the inwards normal vector (later on we denote this direction by the index $0$), $-\triangle $ is a positive operator in $X$, $\hat{\triangle}$ is the Laplace-Beltrami operator on the boundary $\p X$ and  $E$ denotes the traceless of Ricci. \\

We now recall the fact that $Q$-curvature  and $T $-curvature are related by the Chern-Gauss-Bonnet formula \cite{CQ1}:
\beq \label{GBC} 
\chi(X)=\frac1{32\pi^2}\int_X( |W|^2+4Q) dvol +\frac1{4\pi^2}\oint_{\p X} (\mathscr{L}  + T) d\sigma,
\eeq
where $\chi(X)$ is the Euler characteristic number of $X$, $\oint$ stands for the integration on the boundary and $\mathscr{L} d \sigma$ is a pointwise conformal
invariant on $\p X$. 

Since {$|W|^2 dvol$} is also a  pointwise conformal invariant term on X, as a consequence of (\ref{GBC}) in dimension 4, $\int_X Q+2 \oint_{\p X}T$ is an integral conformal invariant. We remark that in the cases when $g$ is the adapted metric for AHE $(X,g^{+})$  with umbilic boundary (e.g when $g$ is the scalar flat adapted metric), we have $\mathscr{L}=0$ on the boundary from the expression of $\mathscr{L}$.  Hence, Gauss-Bonnet-Chern formula in these cases can be written as
\beq 
\chi (X)=\frac{1}{32\pi^2}\int_X( |W|^2+4Q)+ \frac{1}{4\pi^2}\oint_{\p X} T.
\eeq

We now derive the key property of 
a scalar flat adapted metric metric under the additional assumption that the $L^2$ norm of the Weyl curvature is $L^2$, namely the total curvature of the metric is bounded in the sense of the formula (\ref{curvaturebounds})  below.

\begin{lemm} 
\label{boundedwyel}
Let  $\{X,M=\p X, g_i^+\}$ be a family of $4$-dimensional oriented conformally compact Einstein manifolds with boundary $\p X$. Assume $(X,g_i)$ is a sequence of scalar flat adapated metrics, with the boundary metrics $(M,[g_i|_M])$ being compact in $C^{k,\alpha}$ norm with $k\ge 3$, assume furthermore
   
\begin{enumerate}
\item there exists some positive constant $C_0 > 0$ such that the Yamabe constant on the boundary is bounded below by $C_2$, i.e.
$$
Y(M,[g_i|_M]) \ge C_0;
$$

\item there exists some positive constant $C_1 > 0$ such that
$$
\int_X |W_{g_i}|^2 dvol_{g_i} \le C_1.
$$
\end{enumerate}
Then, there exists a constant $C_2>0$ (depending on $C_0$, $C_1$ and the $C^2$ norm of the metrics $g_i|_M$) so that 
\beq
\label{curvaturebounds}
\int_X |Rm_{g_i}|^2 dvol_{g_i}+\oint_{\p X} H^3_{g_i} dvol_{g_i|_{\p X}}  \le C_2.
\eeq
\end{lemm}

\begin{proof}
We henceforth drop the indices $i$ for simplicity. Recall that on the boundary we have the Gauss-Codazzi equation:
\beq
\label{eq2.4bis}
R=\hat{R}+2Ric_{00}+\|L\|^2-H^2,
\eeq
where $\hat{R}$ is the scalar curvature of the boundary metric and $L$ is the second fundamental form.  Using the facts $R=0$ for the adapted metric $g$ and the boundary is umbilic, we have
$$
T=- \frac{H}{3}Ric_{00}+\frac{2}{27}H^3- \frac{1}{3}\hat{\triangle}H=\frac{\hat{R}H}{6}- \frac{H^3}{27}- \frac{1}{3}\hat{\triangle}H,
$$
which yields by Gauss-Bonnet-Chern formula that 
\beq
\label{GBS}
8\pi^2 \chi (X)-\frac14\int_X |W|^2 = \int_X Q + 2\oint_{\p X} T = \int_X -\frac12|Ric|^2+ \oint_{\p X} (\frac{\hat{R}H}{3}-\frac{2H^3}{27}),
\eeq
here we have used the fact that, by our assumption $M$ is compact without boundary, the term  $\ds\oint_{\p X}  \frac{1}{3}\hat{\triangle}H=0$. Applying the H\"older's inequality, we infer
$$
\left| \oint_{\p X} \frac{\hat{R}H}{3}\right|\le \frac 13 \|\hat{R}\|_{L^{3/2}}\|H\|_{L^{3}}.
$$
Applying Lemma \ref{freescalar}, we have $H>0$, thus $\|H\|^3_{L^3} = \oint_{\p X} H^3$. Thus we have
$$
\int_X |Ric|^2 +\frac{4}{27} \|H\|_{L^{3}(\p X)}^3-  C_0\|H\|_{L^{3}(\p X)}\le C, 
$$
by our assumptions that the metrics on the conformal infinity is a compact family in the $C^2$ topology and $L^2$ norm of the Weyl tensor is uniformly bounded. Thus,  we infer the bound of $L^2$ norm of the Ricci tensor and $L^3$ norm of the mean curvature on the boundary, and finish the proof of the Lemma.
\end{proof}

\vskip .4in

\section {$\varepsilon$-regularity properties of the scalar flat adapted metric} 
\label{sectionepsilonregularity}

The main goal of this section is to derive some curvature estimates--namely the $\varepsilon$-regularity property in Proposition \ref{epsilonregularity} below of the scalar flat adapted metric. 
This estimate will later allow us in a proof by contradiction argument, to gain regularity 
of the limit of the blow-up sequence of metrics from $C^3$ to $C^{3, \alpha}$ for some $ \alpha \in (0, 1)$ , which leads to another contradiction and thus establish our main compactness result Theorem \ref{maintheorem}. Details of this argument are given in Lemma \ref{Lem:curv-estimate} 
below.

\vskip .1in

To begin the derivation, first we need to carry out some lengthy curvature tensor computations, which we will summarize sometimes in the text, but mostly place them in Appendix \ref{A}, namely (Lemmas \ref{order0}, \ref{order1}, \ref{Weylorder1}, \ref{highorderexpansion} and Corollaries \ref{corocotten}, \ref{corocotten1}). In these lemmas we express the curvature tensors of the metric in $X$ in terms of the curvature tensors of its boundary metric on $M$ , the tangential derivatives of the boundary metric,  the mean curvature of the metric, together with the non-local curvature term $S$ which we have defined in the Introduction.

We remark that in the previous paper \cite{CG}, such computations have been carried out for the FG adapted metric $g_{FG}$. 
One major difference between $g_{FG} $ and the scalar flat adapted metric $g$, as  explained in the previous Section \ref{section2}, is that; $g_{FG}$ is {\it {totally geodesic}}, hence in particular has zero mean curvature, while the scalar flat adapted metric $g$ which we now use is only {\it umbilic}, with a non-vanishing (but positive) mean curvature term $H$. This complication forces us 
to redo all the tedious computations of curvature tensor in the earlier work \cite{CG} in order to pin down the added 
contributions from $H$ and its derivatives. We have placed these computations in Appendix A, we suggest that the interested readers take the computations in Appendix A for granted 
in the first round of reading and go directly to read the proof of Theorem \ref{curvaturedecay} in Section \ref{section6}.

\vskip .2in

We begin by introducing some notations.

\vskip .1in

Let $\hat{g}=g|_M$ be the boundary metric. Recall that for the scalar flat adapted metrics $g=\rho^2 g^+$ have vanishing scalar curvature. 
 Denote $A=\frac 12 Ric$ the Schouten tensor. Denote $\alpha,\beta\cdots,\in \{1,2,3\}$ the tangential indices and $0$ the unit normal vector on the boundary, and the letters  $i,j,k\cdots$ the full indices. $A=\frac12(Ric -\frac{R}{6}g)$ is the Schouten tensor in $X$ and $W$ is the Weyl tensor in $X$. $\hat{}\,$  is the corresponding operation related to the metric $ \hat{g}$.  Denote by $\hat{Ric},\hat{R}$ the Ricci curvature and the scalar curvature on the boundary $M = \p X$, and $\hat{A}=\hat{Ric} -\frac{\hat{R}}{4}\hat{g}$ the Schouten tensor on the boundary $M$. Recall the Cotton tensor in $X$ (resp. on $M$) is defined by $C_{\alpha\beta\gamma}=A_{\alpha\beta,\gamma}-A_{\alpha\gamma,\beta} $ (resp. $\hat{C}_{\alpha\beta\gamma}=\hat{A}_{\alpha\beta,\gamma}-\hat{A}_{\alpha\gamma,\beta} $).
 
Let $T_{i_1 \cdots i_k}$ be a tensor defined on $X$. Then the Ricci identity
\beq
\label{Ricci}
T_{i_1 \cdots i_k, j l}= T_{i_1 \cdots i_k, l j} - 
\sum_{s=1}^k  R_{m i_s l j} T_{i_1 \cdots i_{s-1} m i_{s+1} \cdots i_k}
\eeq
gives the formula for exchanging derivatives.  The curvature tensor is
  decomposed as
  $$R_{ijkl}= W_{ijkl} + A_{ik} g_{jl}+ A_{jl} g_{ik}- A_{il}g_{jk}- A_{jk} g_{il}.$$

One of the important property of the metric conformal to Einstein metric on 4-manifold is that it is Bach flat. Bach flat metrics are critical points of the 
$L^2$ norm of the Weyl curvature. Bach flat metrics are the metrics when the Bach tensor $B_{ij}$ vanishes, i.e.
\beq
\label{Bachflat1bis}
B_{ij} = \nabla^k\nabla^l W_{ikjl}+A^{kl}W_{ikjl}=0.
\eeq

 Bach-flat condition has been extensively studied in the literature, for example it was pointed out in Tian-Viaclovsky \cite{TV} that Bach flat condition can be re-written as a system of PDE for the Schouten tensor:
\beq
\label{Bachflat1}
\triangle A_{ij} +R_{ikjp}A^{pk}-R_{ip}{A_j}^p+A^{kl}W_{ikjl}=0; 
\eeq  
and equivalently  the Bach flat condition can be re-written as (see the discussion in  \cite[section2]{CG})
\beq
\label{Bachflat2}
\triangle W_{ijkl}+\nabla_l C_{kji}+\nabla_k C_{lij} +\nabla_iC_{jkl}+\nabla_jC_{ilk}=W*Rm+ g*W*A.
\eeq

\vskip .1in

We will now apply the formulas in Appendix \ref{A}, namely Lemmas \ref{order0}, \ref{order1}, \ref{Weylorder1}, \ref{highorderexpansion} and Corollaries \ref{corocotten}, \ref{corocotten1} 
to derive below the $\varepsilon$-regularity result in Proposition \ref{epsilonregularity} and Corollary \ref{epsilonregularity1} .\\

\vskip .1in

We first recall a general Euclidean-type Sobolev inequality. For simplicity, we denote by $dV_g$ (resp. $dV_h$) the volume element $dvol_g$ (resp. $dvol_h$). 
\begin{lemm}
\label{Sobolev}
Let $(X, g^+)$ be an AHE metric, and $g$ some compactified metric. Suppose that Assumption (1) in Theorem \ref{maintheorem} is satisfied.  There exists some positive constant $\varepsilon>0$ such that for any given $p\in M$  and for all $a>0$  if 
  $$
   \|Rm\|_{L^2(B(p,a))}+ \|H\|_{L^3(B(p,a)\cap M)}\le \varepsilon,
   $$
   then for all $f\in H^1(B(p,a))$ vanishing on $\p B(p,a)\cap X$, we have
\beq
\label{sobolev1}
\left\{\int_{B(p,a)\cap X} f^4 dV_g\right\}^{1/2}\le  C\int_{B(p,a) \cap X} |\nabla f|^2 dV_g;
\eeq
and
\beq
\label{sobolev2}
\left\{\oint_{B(p,a)\cap M} |f|^3 dV_h\right\}^{2/3}\le  C\int_{B(p,a) \cap X} |\nabla f|^2 dV_g.
\eeq
\end{lemm}
\begin{proof} In view of Lemma \ref{Yamabe2} and Theorem \ref{Yamabe}, suppose there exists some positive constant $C_0$ so that $Y(M, {g|_M}) \geq C_0$, then exists some other positive constant depending only on dimension of $M$ and $C_0>0$
so that both $Y_1(X, M, [g]) $
and $Y_2 (X, M, [g]) $ have positive lower bound, we denote this by 
$$
Y_1(X,M,[g]) \gtrapprox  C_0, \mbox{ and }Y_2(X,M,[g]) \gtrapprox  C_0.
$$
Hence there exists some constant C depending on $C_0$ such that for all $ f \in H^1(B(p,a)) $
$$
\begin{array}{ll}
&\ds\left\{\int_{B(p,a)\cap X} f^4 dV_g\right\}^{1/2}\\
\le& \ds  C(\int_{B(p,a)\cap X}( |\nabla f|^2 +\frac{R_g}{6}f^2)dV_g+\oint_{B(p,a)\cap M}\frac{H_g}{3}f^2 dV_h) \\
\le&\ds  C\int_{B(p,a)\cap X}|\nabla f|^2 dV_g +C\|\frac {R_g}{6}\|_{L^2(B(p,a))} \left\{\int_{B(p,a)\cap X} f^4 dV_g\right\}^{1/2}\\
&+\ds C\|\frac {H}{6}\|_{L^3(B(p,a))\cap M} \left\{\oint_{B(p,a)\cap M} |f|^3dV_h\right\}^{2/3}\\
\le&\ds  C\int_{B(p,a)\cap X}|\nabla f|^2 dV_g +C\varepsilon \left\{\int_{B(p,a)\cap X} f^4 dV_g\right\}^{1/2}\\
&+ \ds C\varepsilon \left\{\oint_{B(p,a)\cap M} |f|^3dV_h\right\}^{2/3},
\end{array}
$$
so that
 $$
\left\{\int_{B(p,a)\cap X} f^4 dV_g\right\}^{1/2}\le  C\int_{B(p,a) \cap X} |\nabla f|^2 dV_g
  +C\varepsilon \left\{\oint_{B(p,a)\cap M} |f|^3dV_h\right\}^{2/3} $$
provided $C\varepsilon<1/2$ is sufficiently small.  Similarly
$$
\left\{\oint_{B(p,a)\cap M} |f|^3dV_h\right\}^{2/3} \le  C\int_{B(p,a) \cap X} |\nabla f|^2 dV_g
  +C\varepsilon\left\{\int_{B(p,a)\cap X} f^4 dV_g\right\}^{1/2}. $$
As a consequence, we have
$$
\int_{B(p,a)\cap X}( |\nabla f|^2 +\frac{R_g}{6}f^2)dV_g+ \oint_{B(p,a)\cap M}\frac{H_g}{3}f^2 dV_h\le C  \int_{B(p,a)\cap X} |\nabla f|^2dV_g .
$$
Therefore, the desired results (\ref{sobolev1}) and (\ref{sobolev2}) follow.
\end{proof}

We now derive the $\varepsilon$-regularity property.
\begin{prop}
\label{epsilonregularity}
Let $(X, g^+)$ be an AHE metric, and $g$ be the scalar flat adapted metric satisfying the same conditions as in the statement of Theorem A.

Then for each $k$, there exist constants $1>\varepsilon>0$  and $C>0$ (depending on $k$, $\|\hat Rm\|_{C^{k+1}(M)}$,  Yamabe invariants $Y_1 \gtrapprox C_0>0$,  and $Y_2 \gtrapprox C_0 >0 $) such that if 
  $$
   \|Rm\|_{L^2(B(p,a))}+ \|H\|_{L^3(B(p,a)\cap M)}\le \varepsilon, 
   $$
   with $r<1$ and $p\in X$, 
   then 
    $$
\left\{\int_{B(p,a/2)} |\nabla^{k}A|^4 dV_g\right\}^{1/2}\le  \frac{C}{a^{2k+2}}\left(\int_{B(p,a)} |A|^2 dV_g+\oint_{B(p,a)\cap M} |S|  +1\right);
   $$
    $$
\int_{B(p,a/2)} |\nabla^{k+1} A|^2 dV_g\le  \frac{C}{a^{2k+2}}\left(\int_{B(p,a)} |A|^2 dV_g+\oint_{B(p,a)\cap M} |S|  +1\right);
   $$
    \beq
    \label{eqSobolev0}
\left\{\int_{B(p,a/2)} |\nabla^k Rm|^4 dV_g\right\}^{1/2}\le  \frac{C}{a^{2k+2}}\left(\int_{B(p,a)} |A|^2 dV_g+\oint_{B(p,a)\cap M} |S|  +1\right);
   \eeq
    $$
\int_{B(p,a/2)} |\nabla^{k+1} Rm|^2 dV_g\le  \frac{C}{a^{2k+2}}\left(\int_{B(p,a)} |A|^2 dV_g+\oint_{B(p,a)\cap M} |S|  +1\right).
   $$
   In particular, for $B(p,a)\subset X$, we have
   $$
\left\{\int_{B(p,a/2)} |\nabla^{k}A|^4 dV_g\right\}^{1/2}\le  \frac{C}{a^{2k+2}}\left(\int_{B(p,a)} |A|^2 dV_g\right);
   $$
    $$
\int_{B(p,a/2)} |\nabla^{k+1} A|^2 dV_g\le  \frac{C}{a^{2k+2}}\left(\int_{B(p,a)} |A|^2 dV_g\right);
   $$
   \beq
\label{W(k, 4)}
\left\{\int_{B(p,a/2)} |\nabla^k Rm|^4 dV_g\right\}^{1/2}\le  \frac{C}{a^{2k+2}}\left(\int_{B(p,a)} |A|^2 dV_g\right);
   \eeq
    \beq
    \label{W(k+1,2}
\int_{B(p,a/2)} |\nabla^{k+1} Rm|^2 dV_g\le  \frac{C}{a^{2k+2}}\left(\int_{B(p,a)} |A|^2 dV_g\right).
   \eeq
\end{prop}

\begin{proof}

We now begin the proof of the proposition by considering the case $k=0$ first. Let $\eta$ be some cut-off function such that $\eta=1$ on $B(p,\frac {3a}{4})$ and $\eta=0$ outsides $B(p,a)$ and $|\nabla \eta|\le C/a$. Taking the test tensor $\eta^2 A$ in (\ref{Bachflat1}), we have
\beq
\label{eq3.1}
\begin{array}{lllll}
\ds\int\eta^2|\nabla A|^2&=&\ds -\int \eta^2 \langle \triangle A,  A\rangle -2\int \eta \langle\nabla A,  \nabla\eta \otimes A\rangle-\oint \eta^2 \langle\nabla_0 A, A \rangle
\\
&\le& \ds C\int |\nabla A||A|\eta|\nabla \eta|+C\int \eta^2 |Rm| |A|^2 -\oint \eta^2 \langle\nabla_0 A, A \rangle.\\
\end{array}
\eeq
We now recall the expression of the boundary terms in appendix (\ref{bdy}) and apply the second Bianchi identity $ \hat{\nabla}^\alpha \hat{A}_{\alpha\beta} =\frac14 \hat{\nabla}_\beta\hat{R}. $  Applying (\ref{sobolev2}), we obtain 
$$
\begin{array}{lll}
\ds\left| \oint \eta^2\frac13 \hat{A}_{\alpha\beta} \hat{\nabla}_\alpha \hat{\nabla}_\beta H \right|&\ds\le \left| \oint \eta\frac23 \hat{A}_{\alpha\beta} \hat{\nabla}^\alpha \eta\hat{\nabla}^\beta H \right|+C \oint  \left|\eta^2 A\right|\\
&\ds\le \frac{C}{a}  \oint  \left|\eta A\right|\le Ca\| \eta A|_M\|_{L^3}\\
&\le Ca\|\nabla (\eta |A|)\|_{L^2}\le C\varepsilon_1\|  \nabla(\eta |A|)\|^2_{L^2}+ C\varepsilon^{-1}_1a^2.\\
\end{array}
$$
where $\varepsilon_1>0 $ is some small positive number to be fixed later. Here we use the fact $vol (B(p,a)\cap M)\le Ca^3$ since $\hat Rm$ is bounded in $C^0$ norm.  Similarly, 
$$
\begin{array}{lll}
\ds\left| \oint \eta^2  \frac{1}{12} \hat{R}\hat{\triangle }H\right|&\ds\le  \left| \oint \eta\frac{1}{6} \hat{R} \hat{\nabla}^\alpha \eta\hat{\nabla}_\alpha H \right|+ C \oint  \left|\eta^2 A\right|\le \frac{C}{a}  \oint  \left|\eta A\right|\\
&\le   C\varepsilon_1\|  \nabla(\eta |A|)\|^2_{L^2}+ C\varepsilon^{-1}_1a^2.
\end{array}
$$
and
$$
\left| -\oint \eta^2  \frac{1}{6} \langle\hat{\nabla } \hat{R},\hat{\nabla }H\rangle\right|\le   C\varepsilon_1\|  \nabla(\eta |A|)\|^2_{L^2}+ C\varepsilon^{-1}_1a^2.
$$
By the Cauchy-Schwarz inequality, we have also
$$
\begin{array}{lll}
&\ds
\left|- \frac{2}{27}  \oint \eta ^2H^2\hat{\triangle }H- \frac{4}{27}  \oint \eta ^2H|\hat{\nabla} H|^2\right|\\
=&\ds \left|\frac{4}{27}  \oint \eta H^2\langle \hat{\nabla} H, \hat{\nabla}\eta\rangle \right| \le   \ds \frac{C}{a} \oint H^2\eta|A| \\
&\le \ds \frac{C}{a}\|H\|^2_{L^3} \|\eta|A|\|_{L^3}\le  C\varepsilon_1\|  \nabla(\eta |A|)\|^2_{L^2}+ C\varepsilon^{-1}_1a^{-2}.
\end{array}
$$
since $\|H\|_{L^3}<1$ is bounded.  Gathering the above estimates, we obtain
\beq
\oint \eta^2 (\langle\nabla_0 A, A \rangle-\langle S,\hat{A}\rangle)\ge -C\varepsilon_1\|  \nabla(\eta |A|)\|^2_{L^2}- C\varepsilon^{-1}_1a^{-2}.
\eeq
since it follows from Corollary \ref{corocotten1} in the appendix  that on $M$
$$
\frac H3|\hat{A}_{\alpha\beta}|^2-\frac{4}{27} H^3\hat{R}+ \frac{4}{81} H^5+ \frac{5}{48} H\hat{R}^2+\frac{14}{27}H|\hat{\nabla }H|^2\ge 0.
$$
On the other hand, since by our assumption on $(M, \hat g)$, $\hat A$ is bounded in $C^0$ norm, we obtain that there exists some constant $C$ which depends on $C_0$, the dimension (n=3) and on the $||{\hat R_m}||_{C^1 (M)} $ norm so that 
\beq
\label{k=0case}
\oint \eta^2 |S||\hat A|\le C \oint \eta^2 |S|.
\eeq
Now we apply the Cauchy-Schwarz inequality and (\ref{sobolev1})
\beq
\label{eq3.4}
\begin{array}{lllll}
 &\ds C\int |\nabla A||A|\eta|\nabla \eta|+C\int \eta^2 |Rm| |A|^2\\
 \le & \ds\frac 18 \int_X |\eta \nabla  A|^2 +C'(\int_X  |\nabla \eta|^2 |  A|^2+\|\eta A\|_{L^4}^2 \|Rm\|_{L^2(B(p,a))})\\
 \le & \ds\frac 18 \int_X |\eta \nabla  A|^2 +C'(\int_X  |\nabla \eta|^2 |  A|^2+\|\nabla(\eta |A|)\|_{L^4}^2 \|Rm\|_{L^2(B(p,a))}).
\end{array} 
\eeq
Together with (\ref{eq3.1})-(\ref{eq3.4}), we infer
$$
\begin{array}{lllll}
&\ds \int\eta^2|\nabla A|^2\\
\le&\ds C(1+\varepsilon^{-1}_1a^{-2}+ \int  (\eta^2+|\nabla \eta|^2)| A|^2+  \|\nabla (\eta |A|)\|_{L^2}^2 (\|Rm\|_{L^2(B(p,a))}+\varepsilon_1))\\
&\ds+ C \oint \eta^2 |S|+ \frac 12 \int\eta^2|\nabla A|^2\\
\le&\ds (\frac 12+2C(\|Rm\|_{L^2(B(p,a))}+\varepsilon_1)) \int\eta^2|\nabla A|^2+ C \oint \eta^2 |S|\\
&\ds+C(1+\varepsilon^{-1}_1a^{-2}+ \int  (\eta^2+|\nabla \eta|^2)| A|^2+ 2 \||\nabla \eta| |A|\|_{L^2}^2 \|Rm\|_{L^2(B(p,a))}).
\end{array} 
$$
Therefore, when $4C(\varepsilon_1+\varepsilon)<\frac14$,  we get
$$
\int\eta^2|\nabla A|^2 dV_g\le  \frac{C}{a^2}\left(\int_{B(p,a)} |A|^2 dV_g+\oint_{B(p,a)\cap M} |S|+1\right).
$$
Here we use $|\nabla |A||\le |\nabla A|$. Again from the Sobolev inequality (\ref{sobolev1}), we deduce
$$
\left\{\int\eta^2| A|^4 dV_g\right\}^{1/2}\le  \frac{C}{a^2}(\int_{B(p,r)} |A|^2 dV_g+\oint_{B(p,a)\cap M} |S|+1).
$$
Recall from the Bach flat equation for the Weyl tensor (\ref{Bachflat2}) and the Ricci identity  (\ref{Ricci}) 
$$
\triangle W_{ijkl} =2(\nabla_i\nabla_k A_{jl}-\nabla_i\nabla_l A_{jk}-\nabla_j\nabla_k A_{il}+\nabla_j\nabla_l A_{ik})  + A*A+A*W+W*W.
$$
As before, we take $\eta^2 W$ as the test tensor to the above equation. We remark on the boundary $W=0$. Thus
$$
\begin{array}{lllll}
&&\ds\int\eta^2|\nabla W|^2\\
&\le& \ds C\int (|\nabla W|+|\nabla A|)|W|\eta|\nabla \eta|+C\int \eta^2 |Rm| (|W|^2+|A|^2)\\
&&\ds+2|\int \eta^2( \nabla_k A_{jl} \nabla_i W_{ijkl}  -\nabla_l A_{jk}  \nabla_i W_{ijkl}-\nabla_k A_{il}\nabla_j W_{ijkl}+\nabla_l A_{ik} \nabla_jW_{ijkl})|\\
&\le&\ds  C(\int (|\nabla W|+|\nabla A|)|W|\eta|\nabla \eta|+\int \eta^2 |Rm| (|W|^2+|A|^2)+\int \eta^2 |\nabla W| |\nabla A| )\\
&\le&\ds  \frac12 \int\eta^2|\nabla W|^2+C(\int | W|^2|\nabla \eta|^2+(\|\eta W\|_{L^4}^2+\|\eta A\|_{L^4}^2) \|Rm\|_{L^2(B(p,a))}\\
&&\ds+\int \eta^2 |\nabla A|^2). 
\end{array}
$$
With the similar arguments as above, we infer
$$
\int\eta^2|\nabla W|^2 \le  \frac{C}{a^2}\left(\int_{B(p,a)}|Rm|^2 dV_g+\oint_{B(p,a)\cap M} |S|+1\right).
$$
Again from the Sobolev inequality  (\ref{sobolev1}), we get the desired inequality.
$$
\left \{\int(\eta|W|)^4\right\}^{\frac12}\le  \frac{C}{a^2}\left(\int_{B(p,a)}|Rm|^2 dV_g+\oint_{B(p,a)\cap M} |S|+1\right).
$$
Now use the relation
$$
 |\nabla^{(l)}Rm|^2= |\nabla^{(l)}W|^2+|\nabla^{(l)}A\circledwedge g|^2.
$$
Therefore, we obtain the corresponding inequalities for $Rm$. Thus we have finished the proof of the proposition for the case $k=0$
.\\ 

We now prove the proposition for the high $k\ge 1$ by induction. For each k, let $\eta_k$ be some cut-off function such that $\eta_k=1$ on $B(p,\frac {a}{2}+\frac {a}{2^{k+2}})$ and $\eta_k=0$ outside $B(p,\frac {a}{2}+\frac {a}{2^{k+1}})$ and $|\nabla \eta_k|\le C/a$. But for simplicity of the notation, we denote all such cut functions as $\eta$ and skip the index k. 
First we treat the estimates for the Schouten tensor $A$. From the Bach flat equation (\ref{Bachflat1}), we obtain
$$
\triangle \nabla^{(k)} A+\sum_{l=0}^{k}\nabla^{(l)}Rm*\nabla^{(k-l)}A=0.
$$
As above, we take $\eta^2\nabla^{(k)} A$ as the test tensor, integrate the equality
of the Bach equation, we obtain
\beq
\label{highorderschouten}
\begin{array}{lllll}
&&\ds\int\eta^2| \nabla^{(k+1)} A|^2\\
&=&\ds -\int \eta^2 \langle \triangle  \nabla^{(k)} A,  \nabla^{(k)} A\rangle -2\int \eta \langle \nabla^{(k+1)} A,  \nabla\eta \otimes  \nabla^{(k)}A\rangle\\
&&\ds-\oint \eta^2 \langle\nabla_0  \nabla^{(k)}A,  \nabla^{(k)}A \rangle.
\\
&\le& \ds C\int | \nabla^{(k+1)} A|| \nabla^{(k)}A|\eta|\nabla \eta|+ C \sum_{l=0}^{k}\left|\int \eta^2\nabla^{(l)}Rm*\nabla^{(k-l)}A*\nabla^{(k)}A\right|\\
&&\ds+C\left|\oint \eta^2 \langle\nabla_0 \nabla^{(k)}  A, \nabla^{(k)}  A \rangle\right|.
\end{array}
\eeq
We need just to consider the boundary term
$$
\langle \nabla_0\nabla^{(k)}A, \nabla^{(k)}A\rangle.
$$
Our basic observation is that in all these products, one is an odd term and another one an even term, where odd and even is defined as in the  proof of the Lemma \ref{highorderexpansion}; we also deduce
from this lemma that on the boundary $M$
$$
\langle \nabla_0\nabla^{(k)}_{odd}A, \nabla^{(k)}_{old}A\rangle=O(|\nabla^{(k)}A|(1+ |H \nabla^{(k)}A|+\sum_{l=0}^{k-1}|\nabla^{(l)}Rm||P_{k-1-l}(H,\nabla A)|));
$$
and
$$
\begin{array}{lllll}
&\ds\oint \eta^2\langle \nabla_0\nabla^{(k)}_{even}A, \nabla^{(k)}_{even}A\rangle\\
=&\ds\oint \eta^2\langle \sum c_\alpha {\nabla}_\alpha \nabla^{(k)}_{odd}A, \nabla^{(k)}_{even}A\rangle+O(\sum_l\oint \eta^2|\nabla^{(k)}A||\nabla^{(l)}Rm||\nabla^{(k-1-l)}A|)\\
&\ds+O(\oint \eta^2H |\nabla^{(k)}A|^2),
\end{array}
$$
where $c_\alpha$ is some  constant. Apply integration by parts, we infer
$$
\begin{array}{lllll}
&\ds\oint \sum c_\alpha\eta^2\langle \nabla_\alpha \nabla^{(k)}_{odd}A, \nabla^{(k)}_{even}A\rangle\\
=&\ds-2\oint \eta  \sum c_\alpha \nabla_\alpha \eta \langle  \nabla^{(k)}_{odd}A, \nabla^{(k)}_{even}A\rangle-\oint \eta^2\langle  \nabla^{(k)}_{odd}A,   \sum c_\alpha\nabla_\alpha\nabla^{(k)}_{even}A\rangle\\
&\ds+O(\oint \eta^2H |\nabla^{(k)}A|^2).
\end{array}
$$
Thus, we could apply the estimate from Lemma  \ref{highorderexpansion} in the appendix that
$$
\begin{array}{lllll}
&\ds\oint \sum c_\alpha\eta^2\langle \nabla_\alpha \nabla^{(k)}_{odd}A, \nabla^{(k)}_{even}A\rangle\\
=&\ds
O(\oint \frac{\eta}{a}|\nabla^{(k)}A|(1+ |H \nabla^{(k-1)}A|+\sum_{l=0}^{k-2}|\nabla^{(l)}Rm||P_{k-2-l}(H,\nabla A)|))
\\
&\ds +O(\oint \eta^2( | \nabla^{(k)}A|(1+ |H \nabla^{(k)}A|+\sum_{l=0}^{k-1}|\nabla^{(l)}Rm||P_{k-1-l}(H,\nabla A)|))\\
&\ds+O(\oint \eta^2H |\nabla^{(k)}A|^2).
\end{array}
$$
Our basic observation is that
$$
\begin{array}{lllll}
&\ds\oint \eta^2( | \nabla^{(k)}A|(1+ |H \nabla^{(k)}A|+\sum_{l=0}^{k-1}|\nabla^{(l)}Rm||P_{k-1-l}(H,\nabla A)|)) \\
\le &\ds Ca^2  \| \eta\nabla^{(k)}A\|_{L^3}+ C  \| \eta\nabla^{(k)}A\|_{L^3}^2 \|H\|_{L^3} \\
&\ds +\sum_l\| \eta\nabla^{(k)}A\|_{L^3} \|\nabla^{(l)}Rm\|_{L^3(B(p,r_l)\cap  M)} \|\eta |P_{k-1-l}(H,\nabla A)|\|_{L^3}.
\end{array}
$$
and
$$
\oint \eta^2H |\nabla^{(k)}A|^2\le  C  \| \eta\nabla^{(k)}A\|_{L^3}^2 \|H\|_{L^3} .
$$
and
$$
\begin{array}{lllll}
&\ds\oint \frac{\eta}{a}|\nabla^{(k)}A|(1+ |H \nabla^{(k-1)}A|+\sum_{l=0}^{k-2}|\nabla^{(l)}Rm||P_{k-2-l}(H,\nabla A)|)\\
\le&\ds  Cr  \| \eta\nabla^{(k)}A\|_{L^3}+ Cr^{-1}  \| \eta\nabla^{(k)}A\|_{L^3} \| \nabla^{(k-1)}A\|_{L^3(B(p,r_l)\cap  M)}\|H\|_{L^3}\\
&\ds+\sum_l \frac1r \| \eta\nabla^{(k)}A\|_{L^3} \|\nabla^{(l)}Rm\|_{L^3(B(p,a_l)\cap  M)} \| P_{k-2-l}(H,\nabla A)\|_{L^3}.
\end{array}
$$
where $a_l=a/2+a/2^{l+2}$. Applying the Sobolev trace inequality (\ref{sobolev2}), it follows from the induction that we have for any $l< k$
$$
 \|\nabla^{(l)}Rm\|_{L^3(B(p,a_l)\cap M)}^2\le   \frac{C}{a^{2l+2}}\left(\int_{B(p,a)} |Rm|^2 dV_g+\oint_{B(p,a)\cap M} |S|+1\right).
$$
$$
\begin{array}{lllll}
 &&\|\eta\nabla^{(k)}A\|_{L^3(M)}^2\\
 &\le&  C  \|\eta\nabla^{(k+1)}A\|_{L^2(X)}^2+ C\|\nabla \eta\otimes\nabla^{(k)}A\|_{L^2(X)}^2 \\
 &\le&  \ds C  \|\eta\nabla^{(k+1)}A\|_{L^2(X)}^2+   \frac{C}{a^{2k+2}}(\int_{B(p,a)} |Rm|^2 dV_g+\oint_{B(p,a)\cap M} |S|+1).
 \end{array}
 $$
We claim for all $0\le l\le k-1$ 
$$
 \| P_{k-1-l}(H,\nabla A)\|_{L^3(B(x,a_k)\cap M)}^2 \le  \frac{C}{a^{2(k-l)}}\left(\int_{B(p,a)} |Rm|^2 dV_g+\oint_{B(p,a)\cap M} |S|+1\right).
$$
In fact, by the induction, for all $0\le j\le k-1$, we have
$$
 \|\nabla^{(j)} A\|_{L^{3(k-j)}(B(x,a_k)\cap M)}^2\le  \frac{C}{a^{2(k-l)}}\left(\int_{B(p,a)} |Rm|^2 dV_g+\oint_{B(p,a)\cap M} |S|+1\right).
$$
On the other hand, we have
$$
\begin{array}{lllll}
 \|H^2\|_{L^{3k}(B(x,a_k)\cap M)}^2&\le& C(\|A\|_{L^{3k}(B(x,a_k)\cap M)}^2+\|\hat R\|_{L^{3k}(B(x,a_k)\cap M)}^2) \\
&\le &\ds \frac{C}{a^{2k}}\left(\int_{B(p,a)} |Rm|^2 dV_g+\oint_{B(p,a)\cap M} |S|+1\right).
 \end{array}
$$
Combining these inequalities, the desired result yields.  Similarly, we have
 for all $0\le l\le k-2$ 
$$
 \| P_{k-2-l}(H,\nabla A)\|_{L^3(B(x,a_k)\cap M)}^2 \le  \frac{C}{a^{2(k-l-1)}}\left(\int_{B(p,a)} |Rm|^2 dV_g+\oint_{B(p,a)\cap M} |S|+1\right).
$$
Thanks of the Cauchy-Schwarz inequality
$$
\begin{array}{lllll}
&\ds  \| \eta\nabla^{(k)}A\|_{L^3} \|\nabla^{(l)}Rm\|_{L^3} \|\eta |P_{k-1-l}(H,\nabla A)|\|_{L^3}\\
\le&\ds \frac1\gamma  \| \eta\nabla^{(k)}A\|_{L^3}^2+ \frac \gamma 4 ( \|\nabla^{(l)}Rm\|_{L^3}\|\eta |P_{k-1-l}(H,\nabla A)|\|_{L^3})^2\\
\le&\ds \frac1\gamma  \|\eta\nabla^{(k+1)}A\|_{L^2(X)}^2+   \frac{C}{a^{2k+2}}\left(\int_{B(p,a)} |Rm|^2 dV_g+\oint_{B(p,a)\cap M} |S|+1\right);
\end{array}
$$
and
$$
\begin{array}{lllll}
a \| \eta\nabla^{(k)}A\|_{L^3}&\le& Ca^2+ \frac1\gamma  \|\eta\nabla^{(k+1)}A\|_{L^2(X)}^2\\
&&\ds +  \frac{C}{a^{2k+2}}\left(\int_{B(p,a)} |Rm|^2 dV_g+\oint_{B(p,a)\cap M} |S|+1\right);
\end{array}
$$
and
$$
\begin{array}{lllll}
\| \eta\nabla^{(k)}A\|_{L^3}^2 \|H\|_{L^3} \le &C   \|\eta\nabla^{(k+1)}A\|_{L^2(X)}^2\|H\|_{L^3}\\
&\ds +  \frac{C}{a^{2k+2}}\left(\int_{B(p,a)} |Rm|^2 dV_g+\oint_{B(p,a)\cap M} |S|+1\right);
\end{array}
$$
and
$$
\begin{array}{lllll}
&a^{-1}  \| \eta\nabla^{(k)}A\|_{L^3} \| \eta\nabla^{(k-1)}A\|_{L^3}\\
\le & \ds\frac1\gamma  \|\eta\nabla^{(k+1)}A\|_{L^2(X)}^2 +  \frac{C}{a^{2k+2}}\left(\int_{B(p,a)} |Rm|^2 dV_g+\oint_{B(p,a)\cap M} |S|+1\right);
\end{array}
$$
and also
$$
\begin{array}{lllll}
&\ds \frac 1a \| \eta\nabla^{(k)}A\|_{L^3} \|\nabla^{(l)}Rm\|_{L^3} \|P_{k-2-l}(H,\nabla A)|\|_{L^3}\\
\le&\ds \frac1\gamma  \|\eta\nabla^{(k+1)}A\|_{L^2(X)}^2+   \frac{C}{a^{2k+2}}\left(\int_{B(p,a)} |Rm|^2 dV_g+\oint_{B(p,a)\cap M} |S|+1\right).
\end{array}
$$
On the other hand, we have
$$
\begin{array}{lllll}
\ds 
&&\ds\int | \nabla^{(k+1)} A|| \nabla^{(k)}A|\eta|\nabla \eta|\\
&\le &\ds \frac1\gamma  \|\eta\nabla^{(k+1)}A\|_{L^2(X)}^2 + C\int | \nabla^{(k)}A|^2|\nabla \eta|^2\\
&\le&\ds \frac1\gamma  \|\eta\nabla^{(k+1)}A\|_{L^2(X)}^2+   \frac{C}{a^{2k+2}}\left(\int_{B(p,a)} |Rm|^2 dV_g+\oint_{B(p,a)\cap M} |S|+1\right);
\end{array}
$$
and
$$
\begin{array}{lllll}
\ds 
&&\ds\sum_{l=1}^{k-1}\int \eta^2|\nabla^{(l)}Rm||\nabla^{(k-l)}A|| \nabla^{(k)}A|\\
&\le &\ds \sum_{l=1}^{k-1}\|\nabla^{(l)}Rm\|_{L^4(B(p,a_l))} \|\eta \nabla^{(k-l)}A\|_{L^2}\|\eta \nabla^{(k)}A\|_{L^4}\\
&\le&\ds \frac1\gamma  \|\eta\nabla^{(k+1)}A\|_{L^2(X)}^2+   \frac{C}{a^{2k+2}}\left(\int_{B(p,a)} |Rm|^2 dV_g+\oint_{B(p,a)\cap M} |S|+1\right).
\end{array}
$$
Here the constant $C$ depends also on the $\gamma$. It remains to treat $\ds\int \eta^2\nabla^{(k)}Rm*A* \nabla^{(k)}A$ and $\ds \int \eta^2Rm*\nabla^{(k)}A* \nabla^{(k)}A$.  For the term $\int \eta^2\nabla^{(k)}Rm*A* \nabla^{(k)}A$,  using the Sobolev inequality (\ref{sobolev1}), H\"older's inequality and Cauchy-Schwarz inequality
$$
\begin{array}{lllll}
&&\ds\int \eta^2|\nabla^{(k)}Rm* A* \nabla^{(k)}A)|\\
&\le &\ds \|\eta \nabla^{(k)}Rm\|_{L^2}\|\eta \nabla^{(k)}A\|_{L^4} \| A\|_{L^4(B(p,a_1))}\\
&\le&\ds \frac1\gamma  \|\eta\nabla^{(k)}A\|_{L^4(X)}^2+C\|\eta  \nabla^{(k)}Rm\|_{L^2}^2 \| A\|_{L^4(B(p,a_1))}^2\ \\
&\le&\ds   \frac C\gamma  \|\eta\nabla^{(k+1)}A\|_{L^2(X)}^2+\frac{C}{a^{2k+2}}\left(\int_{B(p,a)} |Rm|^2 dV_g+\oint_{B(p,a)\cap M} |S|+1\right).
\end{array}
$$
Similarly, we have
$$
\begin{array}{lllll}
&&\ds 
\int \eta^2|Rm*\nabla^{(k)}A* \nabla^{(k)}A|\\
&\le&\ds  \|Rm\|_{L^4(B(p,a_1))}\|\eta \nabla^{(k)}A\|_{L^4}\|\eta \nabla^{(k)}A\|_{L^2}\\
&\le &\ds \frac 1\gamma \|\eta \nabla^{(k)}A\|_{L^4}^2+ C   \|Rm\|_{L^4(B(p,a_1))}^2 \|\eta \nabla^{(k)}A\|_{L^2}^2\\
&\le & \ds \frac C\gamma  \|\eta\nabla^{(k+1)}A\|_{L^2(X)}^2+   \frac{C}{a^{2k+2}}\left(\int_{B(p,a)} |Rm|^2 dV_g+\oint_{B(p,a)\cap M} |S|+1\right).
\end{array}
$$
Gathering all these estimates together, we deduce
\begin{align*}
 &\|\eta\nabla^{(k+1)}A\|_{L^2(X)}^2 \\
 \le&\ds \frac12  \|\eta\nabla^{(k+1)}A\|_{L^2(X)}^2+   \frac{C}{a^{2k+2}}\left(\int_{B(p,a)} |Rm|^2 dV_g+\oint_{B(p,a)\cap M} |S|+1\right)
\end{align*}
provided $\gamma$ is a sufficiently large constant. Therefore
$$
 \|\eta\nabla^{(k+1)}A\|_{L^2(X)}^2 \le  \frac{C}{a^{2k+2}}\left(\int_{B(p,a)} |Rm|^2 dV_g+\oint_{B(p,a)\cap M} |S|+1\right).
$$
By the Sobolev inequality (\ref{Sobolev}),  we get 
\begin{align*}
\|\eta\nabla^{(k)}A\|_{L^4(X)}^2 \le \ds \frac{C}{a^{2k+2}}\left(\int_{B(p,a)} |Rm|^2 dV_g+\oint_{B(p,a)\cap M} |S|+1\right).     
\end{align*}
It is similar for the Weyl tensor. From the Bach flat equation (\ref{Bachflat2}), we have
\beq
\label{Bachflat3}
\begin{array}{lllll}
\ds 
&&\triangle \nabla^{(k)} W_{ijml} \\
&=&\ds 2 ( \nabla_i\nabla_m \nabla^{(k)} A_{jl}-\nabla_i\nabla_l \nabla^{(k)} A_{jm}-\nabla_j\nabla_m\nabla^{(k)} A_{il}+\nabla_j\nabla_l \nabla^{(k)}A_{im})  \\
&&\ds + \sum_{l=0}^{k} (\nabla^{(l)}W*\nabla^{(k-l)}W+\nabla^{(l)}W*\nabla^{(k-l)}A+\nabla^{(l)}A*\nabla^{(k-l)}A).
\end{array}
\eeq
As before, we take $\eta^2\nabla^{(k)} W$ as test tensor and integrate the equality. Thus, we have
\beq
\label{highorderWeyl}
\begin{array}{lllll}
&&\ds\int\eta^2| \nabla^{(k+1)} W|^2\\
 &\le& \ds C(\int (| \nabla^{(k+1)} A|+| \nabla^{(k+1)} W|)| \nabla^{(k)}W|\eta|\nabla \eta|+\int \eta^2| \nabla^{(k+1)} A|| \nabla^{(k+1)} W|) \\
&&\ds + C \sum_{l=0}^{k}\left[\left|\int \eta^2\nabla^{(l)}W*\nabla^{(k-l)}A*\nabla^{(k)}W\right|+\left|\int \eta^2\nabla^{(l)}W*\nabla^{(k-l)}A*\nabla^{(k)}A\right|\right.\\
&&\ds +\left|\int \eta^2\nabla^{(l)}A*\nabla^{(k-l)}A*\nabla^{(k)}A\right|+\left|\int \eta^2\nabla^{(l)}W*\nabla^{(k-l)}W*\nabla^{(k)}A\right|\\
&&\ds+\left.\left|\int \eta^2\nabla^{(l)}W*\nabla^{(k-l)}W*\nabla^{(k)}W\right|+\left|\int \eta^2\nabla^{(l)}A*\nabla^{(k-l)}A*\nabla^{(k)}W\right| \right]\\
&&\ds+C\left|\oint \eta^2 \langle\nabla_0 \nabla^{(k)}  W, \nabla^{(k)}  W \rangle\right|+C\left|\oint \eta^2 \langle\nabla_0 \nabla^{(k)}  A, \nabla^{(k)}  W \rangle\right|\\
&&\ds+C\left|\oint \eta^2 \langle \nabla^{(k)}  A, \nabla_0 \nabla^{(k)}  W \rangle\right|.
\end{array}
\eeq
Here we use the above two Bach flat equations (\ref{Bachflat1}) and (\ref{Bachflat3}) and Ricci identity (\ref{Ricci}). With the similar arguments, we can bound the boundary terms as above
$$
\begin{array}{lllll}
&\ds C\left|\oint \eta^2 \langle\nabla_0 \nabla^{(k)}  W, \nabla^{(k)}  W \rangle\right|\\
&\ds+C\left|\oint \eta^2 \langle\nabla_0 \nabla^{(k)}  A, \nabla^{(k)}  W \rangle\right|+C\left|\oint \eta^2 \langle \nabla^{(k)}  A, \nabla_0 \nabla^{(k)}  W \rangle\right|\\
\le &\ds    \frac{C}{a^{2k+2}}\left(\int_{B(p,a)} |Rm|^2 dV_g+\oint_{B(p,a)\cap M} |S|+1\right)+\frac 14 \int\eta^2| \nabla^{(k+1)} W|^2.
\end{array}
$$
And also from the induction and results for $A$ and H\"older's and Cauchy-Schwarz inequalities
$$
\begin{array}{lllll}
& \ds C\int (| \nabla^{(k+1)} A|+| \nabla^{(k+1)} W|)| \nabla^{(k)}W|\eta|\nabla \eta|\\
\le &\ds\frac 14 \int\eta^2| \nabla^{(k+1)} W|^2 +C \int\eta^2| \nabla^{(k+1)} A|^2+ C \int|\nabla \eta|^2 |\nabla^{(k)} W|^2\\
\le &\ds    \frac{C}{a^{2k+2}}\left(\int_{B(p,a)} |Rm|^2 dV_g+\oint_{B(p,a)\cap M} |S|+1\right)+\frac 14 \int\eta^2| \nabla^{(k+1)} W|^2.
\end{array}
$$
and
$$
\begin{array}{lllll}
&\ds C \sum_{l=0}^{k}\left[\left|\int \eta^2\nabla^{(l)}W*\nabla^{(k-l)}A*\nabla^{(k)}W\right|+\left|\int \eta^2\nabla^{(l)}W*\nabla^{(k-l)}A*\nabla^{(k)}A\right|\right.\\
&\ds +\left|\int \eta^2\nabla^{(l)}A*\nabla^{(k-l)}A*\nabla^{(k)}A\right|+\left|\int \eta^2\nabla^{(l)}W*\nabla^{(k-l)}W*\nabla^{(k)}A\right|\\
&\ds+\left.\left|\int \eta^2\nabla^{(l)}W*\nabla^{(k-l)}W*\nabla^{(k)}W\right|+\left|\int \eta^2\nabla^{(l)}A*\nabla^{(k-l)}A*\nabla^{(k)}W\right| \right]\\
\le &\ds C  \sum_{l=1}^{k-1}[\| \nabla^{(l)}Rm\|_{L^4}\|\eta \nabla^{(k-l)}Rm\|_{L^4}\|\eta \nabla^{(k)}Rm \|_{L^2}]\\
&\ds + C\| Rm\|_{L^4} \|\eta \nabla^{(k)}Rm \|_{L^2}(\|\eta \nabla^{(k)}W \|_{L^4}+\|\eta \nabla^{(k)}A \|_{L^4})\\
\le &\ds    \frac{C}{r^{2k+2}}\left(\int_{B(p,r)} |Rm|^2 dV_g+\oint_{B(p,r)\cap M} |S|+1\right)+\frac 14 \int\eta^2| \nabla^{(k+1)} W|^2.
\end{array}
$$
Finally, we infer
$$
 \int\eta^2| \nabla^{(k+1)} W|^2\le    \frac{C}{a^{2k+2}}\left(\int_{B(p,a)} |Rm|^2 dV_g+\oint_{B(p,a)\cap M} |S|+1\right);
$$
which implies from the Sobolev inequality (\ref{Sobolev})
$$
\|\eta\nabla^{(k)}W\|_{L^4(X)}^2 \le  \frac{C}{a^{2k+2}}\left(\int_{B(p,a)} |Rm|^2 dV_g+\oint_{B(p,a)\cap M} |S|+1\right).
$$
We remark by tracing through the above steps  but skipping some details  that the constant $C$ so obtained in the k stage inequalities above depends on $C_0$, $k$ and $\|\hat Rm\|_{C^{k+1}(M)}$.
We have thus finished the proof of Proposition \ref{epsilonregularity}.
\end{proof}

\begin{coro}
\label{epsilonregularity1}
Assume $(X, M, g+)$ a CCE manifold, $g$ a scalar flat adapted metric satisfies assumptions as in Theorem A. For each $ k \geq 2$, we have the estimates for $L^\infty$ norm, that is, there exist constants $\varepsilon\in (0,1)$ and $C$ (depending on $k$, $C_0$ and $\|\hat Rm\|_{C^{k+1}(M)}$) such that if for any $a<1$
  $$
   \|Rm\|_{L^2(B(p,a))} +  \|H\|_{L^3(B(p,a)\cap M)} \le \varepsilon,
   $$
   then 
\beq
\label{Infinitybound}
 \sup_{B(p,a/2)} |\nabla^{k-2} Rm| \le  \frac{C}{a^{k}}\left(\int_{B(p,a)}|Rm|^2 dV_g+\oint_{B(p,a)\cap M} |S|+ 1\right)^{\frac12}.
\eeq  
In particular, when $B(p,a)\subset X$, we have
\beq
\label{Infinitybound1}
 \sup_{B(p,a/2)} |\nabla^{k-2} Rm| \le  \frac{C}{a^{k}}\left(\int_{B(p,a)}|Rm|^2 dV_g \right)^{\frac12}.
\eeq  
\end{coro}

\begin{proof} 
The proof follows from the estimates (\ref{eqSobolev0}), (\ref{W(k, 4)}) and the Sobolov embedding result that functions in $W^{2, 4}$ on $X$ of dimension 4 are bounded.
\end{proof}
\begin{RK} Applying the Sobolev embedding theorem, we have $W^{1,8}\subset C^{0,\frac {1}{2}}$. Thus in the statement of Corollary \ref{epsilonregularity1}, we actually get $Rm\in C^{k-2,\frac 12}$.
\end{RK}

\begin{RK} Later in our proof of Theorem A in sections \ref{section6}, \ref{Sect:Liouville} and \ref{section 7bis}, our starting point is to establish some uniform bound of $\|Rm \|_{C^{1, \alpha}(X)} $ norm with $\alpha\in (0,\frac12)$; which by Corollary \ref{epsilonregularity1}, which will follow from the uniform bound of
$\|\hat Rm \|_{C^{4}(M)} $ norm and $\| Rm \|_{C^{1}(X)} $ norm. 
Equivalently, in order to achieve a uniform $C^{3, \alpha}(X)$ norm of $g$, one needs to assume the boundary metric $\hat g$ having a uniform $C^{6}(M)$ norm. This justifies the assumption we made in the statement of Theorem A that the sequence of metrics ${\hat g_i}$ has a uniform $C^{\bar k, \alpha}(M)$ norm for some $\bar k \geq 6$.
\end{RK}



\vskip .2in
\section{Curvature decay of the limiting metric}
\label{section6}

Our plan to establish the main Theorem \ref{maintheorem} is via a contradiction argument. To that aim, we consider a sequence of CCE metrics ${g_i^+}$  on a fixed 4-manifold X ,  and a sequence of corresponding adapted scalar flat metrics ${g_i}$ with boundary metrics $ {h_i =g_i|_M}$ that satisfy the assumptions in the statement of Theorem \ref{maintheorem}. 
To show that the corresponding sequence of scalar flat adapted metrics $ {g_i} $ forms a compact family in $X$, we will first establish that the curvature of the family of such metrics is $C^3$ bounded by a contradiction argument. This was the same strategy used in the proof of the corresponding compactness theorem in our earlier work, \cite[Theorem1.1]{CG} which we had established a similar result but under different and stronger and {\it not}  conformally invariant assumptions.
 
Assume the contrary, i.e. if the $C^3$ norm of a sequence of such metrics is not bounded,  denote $K_i^2 = sup_{X} (|Rm_{g_i}| + |\nabla Rm_{g_i}|^{2/3})  = ( |Rm_{g_i}| + |\nabla Rm_{g_i}|^{2/3})(p_i)  $  for some point ${p_i} \in X$  and assume $K_i $ tends to infinity as i tends to infinity and denote $ \bar{g_i}:= : K_{i }^2 g_{i}$.

We now state the main properties of rescaled metric $\bar g_i$.

\vskip .1in   

\begin{lemm} \label{limiting-metric} Under the same assumptions of Theorem \ref{maintheorem} and assume the $ C^3$ norm of the curvature of $g_i$ is unbounded. Then 
\begin{enumerate} 
\item $\ds\int_{X_{\infty} } |Rm_{g_{\infty}}|^2  <\infty$;
\item Denote $H_{\infty}$ the mean curvature  $H_{g_\infty} \ge 0$ and $\ds\int_{\p X} H_{\infty}^3<\infty$;
\item The Sobolev inequalities (\ref{sobolev1}) and  (\ref{sobolev2}) hold for all $f\in H^1 (X_\infty)$ with compact support;
\item The conformal infinity of $X_{\infty}$ is  $\mathbb{R}^3$ equipped with the flat metric.
\item $g_{\infty} $ is complete non-compact, Bach flat metric with vanishing scalar curvature and umbilic on the boundary $\mathbb R^3.$

\item $ (X,\bar g_i) $ convergence uniformly in $C^{k-3}$ norm for some $k\ge 6$  in the Gromov-Hausdroff sense to a limiting space $(X_{\infty}, g_{\infty})$ satisfying $\| Rm_{g_{\infty}} \|_{C^3}\le 1$.

and
\item The interior injectivity radius $i(g_{\infty})$ and the boundary injectivity radius $i_\p(g_{\infty})$ are bounded below by some positive constant $C>0$.
\end{enumerate} 
\end{lemm}

\begin{proof} 
Estimates in (1), (2) and (3) follow directly from the corresponding properties of $g_i$ and hence $\bar g_i$ as they are scale invariant. (4) follows by our assumption that $g_i|_M$ is a compact family in $C^{k, \alpha}$ for some $ k \geq 6$. (5) is a direct consequence of (4) and the fact that  $g_{\infty}$ is the limit of $\bar g_i$. 
(6) follows from the gain of regularity of $\bar g_i$ to $C^{k-3}$ norm for some $k\ge 6$ as in Corollary \ref{epsilonregularity1}. The property (7) follows from the fact that the limiting metric $g_\infty$ is a complete manifold with a bounded curvature tensor and without collapse, that is, the volume of any geodesic ball with the radius equal to $1$ is bounded below. By a similar blow-up analysis as in  (\cite[Lemmas 3.1 and 3.3]{CGQ}), we obtain (7). We leave the details to the interested readers. 
Thus, we established the lemma.

\end{proof}


The main result of this section is the curvature decay property (\ref{decay2}) of the limiting metric $g_{\infty}$, together with its delicate distance decay property (\ref{conformalfactorestimate1})  stated in Theorem \ref{curvaturedecay} below.

We will apply the Bach flat equation (\ref{Bachflat1}) to prove the curvature decay of the limiting metric. The proof is parallel to the proof of the $\varepsilon$-regularity result in Proposition \ref{epsilonregularity} in the previous section. The difference is that for the rescaled limiting metric, we now will work on a region outside of a geodesic ball with sufficiently large radius, and verify that once the integral of the curvature on the region is sufficiently small, the curvature decays point-wisely in the region; while the argument in Proposition \ref{epsilonregularity} is a local argument which only holds locally on balls with sufficiently small radius. Also, we note that the scaling invariance of the metrics allows us to drop the dependence of the non-local term $S$ in the estimates in Proposition \ref{epsilonregularity}. Another key fact used in the proof is  that under Assumption (1) in Theorem \ref{maintheorem}, the conformal infinity of the limiting metric restricted on the boundary is, in fact, the flat metric on  ${{\R}^3}$. 

\begin{theo}
\label{curvaturedecay}
Let $(X_{\infty}, g_{\infty})$ be a $C^5$ complete non-compact metric  satisfying  properties (1)-(6) in Lemma \ref{limiting-metric}.   
Then for any fixed point $O \in \partial X_{\infty}$, and any $ y\in X_{\infty}$, we have
\beq \label{decay2}
Ric_g(y)=o(d_{g} (y, O)^{-2})
\eeq
Moreover, 
\beq \label{conformalfactorestimate}
\lim_{d_g (y,O)\to\infty}|\nabla \rho(y)|=1
\eeq
Furthermore, we have for all $y\in \R^3$
$$
d_g(y, O)\le |y|,
$$
where $g = g_{\infty}$  and $|y|$ denotes the Euclidean distance of $y$ to the  point $O$. 
And for any small $\varepsilon>0$, we have 
\beq 
\label{conformalfactorestimate1}
\lim_{a\to\infty} \frac{\inf_{y\in \partial B_{\R^3}(O,a)} d_{g} (y, O)}{a^{\frac{1}{1+\varepsilon}}}=+\infty
\eeq
where  $B_{\R^3}(O,a)$ is the Euclidean ball of center $O$ and radius equals to $a$ on  ${{\R}^3}$.
\end{theo}

The delicate estimate (4.3) in the lemma 
is crucial for us to derive the key estimates (5.32) and (5.36)  later in the proof of rigidity Theorem B.

To prove the theorem above, we first derive two lemmas. For the rest of this section, to simplify the notation, we often drop the sub-index infinity, i.e. we denote $g_{\infty}$ by $g$, $X_{\infty}$ by $X$ and $M=\p X$ etc.\\

We start with a version of the $\varepsilon$ regularity result in this setting.
\begin{lemm}
\label{regularity-at-infty}
There exist some constants  $\varepsilon>0$ and $C>0$ (depending on $k$, the constants appearing in the Sobolev inequalities (\ref{sobolev1}) and  (\ref{sobolev2}) )  such that if 
  \beq
   \|Rm\|_{L^2(B(p,a))}+ \|H\|_{L^3(B(p,a)\cap M)}\le \varepsilon,
   \eeq
  for any $p\in X$, $a>1$ sufficiently large and $k\le 2$, we have
    
    \beq
    \label{k+1order2decay}
\int_{B(p,a/2)} |\nabla^{k+1} Rm|^2 dV_g\le  \frac{C}{a^{2k+2}}\left(\int_{B(p,a)} |A|^2 dV_g+\oint_{B(p,a)\cap M}H^3\right);
   \eeq
   \beq 
    \label{korder4decay}
\left\{\int_{B(p,a/2)} |\nabla^k Rm|^4 dV_g\right\}^{1/2}\le  \frac{C}{a^{2k+2}}\left(\int_{B(p,a)} |A|^2 dV_g+\oint_{B(p,a)\cap M}H^3\right).
   \eeq 
 \end{lemm}

 \begin{proof}
 As stated above, the proof of this lemma is in spirit analogous to the proof of the $\varepsilon$ regularity result in Proposition \ref{epsilonregularity}. We will indicate below just the difference and leave out the details. First, we notice that, under Assumption (1) of Theorem \ref{maintheorem}, we have  $\hat{A}=0$ on the boundary of the rescaled limiting metric.
 
 When  $ k=0$, to prove (\ref{k+1order2decay}), first we notice the estimate (\ref {eq3.1}) holds. Moreover, we have from (\ref{bdy}) that
$$
\langle \nabla_0 A, A\rangle=\frac{4}{81}H^5+\frac{10}{27}H|\hat{\nabla }H|^2- \frac{2}{27} H^2\hat{\triangle }H, 
$$
when the conformal infinity is flat. Thus by applying the Cauchy-Schwarz inequality and the fact $H>0$, we have
$$
\begin{array}{lll}
&&\ds\oint \eta^2\langle \nabla_0 A, A\rangle\\
&=&\ds \oint \eta^2( \frac{4}{81}H^5+\frac{10}{27}H|\hat{\nabla }H|^2)- \oint \eta^2  \frac{2}{27} H^2\hat{\triangle }H\\
&=& \ds \oint \eta^2( \frac{4}{81}H^5+\frac{14}{27}H|\hat{\nabla }H|^2) + \oint \eta  \frac{4}{27} H^2\langle \hat{\nabla }H,  \hat{\nabla }\eta\rangle\\
&\ge&  \ds \oint \eta^2( \frac{4}{81}H^5+\frac{14}{27}H|\hat{\nabla }H|^2) - \oint C \eta  \frac{4}{27a} H^2 |\hat{\nabla }H |\\
&\ge&  \ds \oint \eta^2( \frac{4}{81}H^5+\frac{14}{27}H|\hat{\nabla }H|^2) - \oint C \eta  \frac{4}{27a} H^2 |\hat{\nabla }H |\\
&\ge&  \ds \oint \eta^2( \frac{4}{81}H^5+\frac{14}{27}H|\hat{\nabla }H|^2) -\oint \eta^2 \frac{4}{27}H|\hat{\nabla }H|^2-\frac{C^2}{27a^2} \oint_{B(p,a)\cap M} H^3\\
&\ge&  \ds -\frac{C^2}{27a^2} \oint_{B(p,a)\cap M} H^3.
\end{array}
$$
When $k=1$ and $k=2$, recall again the boundary $h$ is the flat metric, thus we can apply the Lemmas \ref{order0}, \ref{order1}, \ref{Weylorder1}, \ref{Weylorder2} and \ref{highorderexpansion} in the appendix, and obtain  on the boundary $M$ that
$$
\nabla^{(1)}_{even} A= H*A,  \nabla^{(1)}_{odd} A= L ( S,\hat \nabla^{(2)} H ), \nabla^{(1)}_{even} W =0, \nabla^{(1)}_{odd} W =L( S);
$$
and
$$
\begin{array}{lll}
\nabla^{(2)}_{even} A =H*\nabla_{odd} A+ Rm*A;\\
\nabla^{(2)}_{even} W = H*\nabla_{odd} W+Rm*Rm;\\
\end{array}
$$
and
$$
\begin{array}{lll}
\nabla^{(3)}_{even} A =H*\nabla^{(2)}_{odd} A+\nabla Rm*A+ Rm*\nabla A; \\
\nabla^{(3)}_{even} W = H*\nabla^{(2)}_{odd} W+ \nabla Rm*Rm.\\
\end{array}
$$
With the similar arguments as in the proof of Proposition \ref{epsilonregularity}, we can establish the desired estimates (\ref{korder4decay}) and (\ref{k+1order2decay}).
Thus we finish all the estimates in Lemma \ref{regularity-at-infty}.

\end{proof}
\vskip .2in

\begin{lemm}
\label{decay-at-infty}
Under the same assumptions as in Lemma \ref{regularity-at-infty}, 
there exists a constant $C_1>0$  (depending on the constants appearing in the Sobolev inequalities (\ref{sobolev1}) and (\ref{sobolev2})) such that
\beq 
\label{pointwisedecay}
 \sup_{B(p,a/2)} |Rm| \le  \frac{C_1}{a^{2}}\left(\int_{B(p,a)} |A|^2 dV_g+\oint_{B(p,a)\cap M}H^3\right)^{\frac12}.
\eeq
\end{lemm}
\begin{proof}
First we recall Kato's inequality: for any tensor $T$, we have
$$
|\nabla |T||\le |\nabla T|
$$ 
We will show first $\nabla Rm\in L^{8}(B(p,a/2))$. Let $\eta$ be some cut-off function such that $supp(\eta)\subset B(p,3a/4)$ and $\eta\equiv 1$ on $B(p,a/2)$ and $|\nabla \eta|\le C/a$. From  \cite[Lemma3.3]{CG}
$$
 \|\eta |\nabla Rm|\|_{L^8}^2\le C \|\eta |\nabla Rm|\|_{L^4}\|\nabla (\eta |\nabla Rm|)\|_{L^4}.
$$
By Lemma \ref{regularity-at-infty} , we have
$$
\|\eta |\nabla^{(2)} Rm|\|_{L^4}\le \frac{C}{a^{3}}\left(\int_{B(p,a)} |Rm|^2 dV_g+\oint_{B(p,a)\cap M} H^3\right)^{\frac12}.
$$
$$
\begin{array}{lll}
\|\nabla (\eta |\nabla Rm|)\|_{L^4}&\le&\ds C(\frac1a \| \nabla Rm\|_{L^4(B(p,3a/4))}+\| \eta \nabla^{(2)} Rm\|_{L^4})\\
&\le&\ds  \frac{C}{a^{3}}\left(\int_{B(p,a)} |Rm|^2 dV_g+\oint_{B(p,a)\cap M} H^3\right)^{\frac12},
\end{array}
$$
which  yields
$$
\|\eta |\nabla Rm|\|_{L^8}^2\le \frac{C}{a^{5}}\left(\int_{B(p,a)} |Rm|^2 dV_g+\oint_{B(p,a)\cap M} H^3\right).
$$
By the same argument, we have
$$
\|\eta Rm\|_{L^8}^2\le \frac{C}{a^{3}}\left(\int_{B(p,a)} |Rm|^2 dV_g+\oint_{B(p,a)\cap M} H^3\right).
$$
Again from \cite[Lemma3.3]{CG}
$$
\begin{array}{lll}
\|\eta  Rm\|_{L^\infty}^2&\le&\ds C \|\eta | \nabla Rm|\|_{L^8}\|\eta  Rm\|_{L^8} \\
&\le &\ds\frac{C}{a^{4}}\left(\int_{B(p,a)} |Rm|^2 dV_g+\oint_{B(p,a)\cap M} H^3\right).
\end{array}
$$
This gives the desired estimate (\ref{pointwisedecay}).
\end{proof} 
\vskip .1in




\vskip .2in

We now start to prove the $\varepsilon$ regularity result for the conformal factor for the adapted limiting metric with free scalar curvature and Euclidean boundary. Recall that $P=\frac{1-|d \rho|_g^2}{\rho}$. By the conformal change, $P=-\frac{1}{2}\triangle \rho$ and (\ref{laplacianP}) we have
$$
\triangle_g P=-\frac{1}{2} \rho |Ric-\frac1{4} Rg|_g^2 = -\frac{1}{2} \rho |Ric|_g^2\mbox{  in } X.
$$


\begin{lemm}
\label{regularity-at-inftyforconformalfactor}
There exist some constants  $\varepsilon>0$ and $C>0$ (depending on  the constants appearing in the Sobolev inequalities (\ref{sobolev1}) and  (\ref{sobolev2}) )  such that if 
  \beq
   \|Rm\|_{L^2(B(p,a))}+ \|H\|_{L^3(B(p,a)\cap M)}\le \varepsilon,
   \eeq
  for any $p\in X$ and $a>1$ sufficiently large, we have
 \beq 
\label{pointwisedecaybis}
 \sup_{B(p,a/2)} |P| \le  \frac{C_1(a+\rho(p))}{a^2}\left(\int_{B(p,a)} |A|^2 dV_g+\oint_{B(p,a)\cap M}H^3\right)^{\frac12}.
\eeq
 \end{lemm}
 \begin{proof} 
 We recall the facts $|\nabla \rho|\le 1$ and $P=\frac{2}{3}H$ and $\p_\nu P=\frac{1}{3}H^2$ on the boundary $\R^3$. Thus $ \rho(y)\le a+\rho(p)$ for all $y\in B(p,a)$. On the other hand, by the conformal change
\beq
\label{expressionRicci}
 Ric= -2\frac{1}{\rho}(\nabla^2 \rho-\frac{1}{4}\triangle \rho g)= -2\frac{1}{\rho}(\nabla^2 \rho+\frac{1}{2}P  g)
 \eeq
On the other hand, we  differentiate the equation (\ref{laplacianP}) to obtain higher order estimates. For example, when we take the first derivative
 $$
 \triangle \nabla P+ Ric*\nabla P=-\frac{1}{2}(\nabla \rho |Ric-\frac1{4} Rg|_g^2+  \rho \nabla|Ric-\frac1{4} Rg|_g^2)
 $$
 Applying the same arguments as in the proof of Lemma \ref{decay-at-infty}, we obtain the desired result (\ref{pointwisedecaybis}).
 \end{proof}

 We now return to the proof of Theorem \ref{curvaturedecay}

\begin{proof}[Proof of Theorem \ref{curvaturedecay}]


Fix $O\in \p X$ and choose $a_1>0$ such that 
$$
\int_{X\setminus B(O,a_1)} |A|^2 dV_g+\oint_{(X\setminus B(O,a_1))\cap M}H^3\le \varepsilon.
$$
Then we have for all  $a>a_1$
$$
 \sup_{B(O,2a)^c} |Rm| \le  \frac{C_1}{a^{2}}\left( \int_{X\setminus B(O,a)} |A|^2 dV_g+\oint_{(X\setminus B(O,a))\cap M}H^3\right)^{\frac12}.
$$
Indeed, let $p\in B(O,2a)^c$, we have $B(p,a)\subset  X\setminus B(O,a)$. By  Lemma \ref{decay-at-infty},
$$
\begin{array}{lll}
|Rm(p)|&\le \sup_{B(p,a/2)} |Rm| \le   \ds\frac{C_1}{a^{2}}\ds\left( \int_{ B(p,a)} |A|^2 dV_g+\oint_{( B(p,a))\cap M}H^3\right)^{\frac12}\\
&\ds\le  \frac{C_1}{a^{2}}\left( \int_{X\setminus B(O,a)} |A|^2 dV_g+\oint_{(X\setminus B(O,r))\cap M}H^3\right)^{\frac12}. 
\end{array}
$$

Hence we have established (\ref{decay2}).\\


It is clear that $d(y,O)\le |y|$ when $y\in\R^3$ and $O$ is the origin on $\R^3$. To establish the last two statements of Theorem \ref{curvaturedecay}, we divide the proof into 4 steps.\\

{\sl Step 1. (\ref{conformalfactorestimate}) holds}.

Applying Lemma \ref{regularity-at-inftyforconformalfactor}, we have 
$$
|P(y)|\le o(d(y,O)^{-1})
$$
We remark 
if $O\in\R^3$, then $\rho(y)/d(y,O)\le 1$ for all $y\in X$ since $|\nabla \rho|\le 1$. Thus, we obtain (\ref{conformalfactorestimate}).\\

Let $\varepsilon_1<\min (\frac{1}{6}, \frac{\varepsilon}{3})$ be a small positive number and $D>0$ a sufficiently large  number such that $|\nabla \rho|\ge 1-\varepsilon_1\frac{\rho(y)}{d(O,y)}$ and $|\nabla^{(2)}\rho(y)|\le \varepsilon_1\frac{\rho(y)}{d(O,y)^2}$  for all $y\in X\setminus B(O, D)$. For the second inequality, we recall (\ref{expressionRicci}). By the decay of Ricci tensor and $P$ function, the desired claim follows. We consider the unit vector field $\frac{\nabla \rho}{|\nabla \rho|}$. Let $\varphi (s)$ be the flow of the field, that is $\partial_s \varphi (s)= \frac{\nabla \rho}{|\nabla \rho|}(\varphi (s))$ with the initial condition on the boundary $\varphi (0,\cdot)\in \R^3$.\\

{\sl Step 2. We claim there exists a large number $D_1>D$ such that $\varphi ([0,\infty)\times (\R^3\setminus B_{\R^3}(0,D_1)))$ is a smooth embedding into $X\setminus B(O,D)$}.\\

For this purpose, it is clear that
$$
\sup_{y\in \bar B(O, D)} \rho(y)\le D
$$
since $|\nabla \rho|\le 1$. We can choose $D_1>0$ such that 
$$
\bar B(O, 4D)\cap \R^3\subset B_{\R^3}(O,D_1)
$$
For all $(s, z)\in [0,D]\times \R^3\setminus B_{\R^3}(0,D_1)$, we have $\rho(\varphi(s,z))=s$ and the length of curve $\varphi([0,s]\times\{z\})$ is smaller than $\frac{s}{1-\varepsilon_1}\le 2D$. Thus $d(\varphi(s,z), O)\ge d(\varphi(0,z), O)-d(\varphi(s,z), \varphi(0,z))\ge D_1-2D\ge 2D$. On the other hand, $\rho(\varphi(\cdot,y))$ is increasing along the flow. Hence, the flow exists for all time and $\rho(\varphi ([2D,\infty)\times (\R^3\setminus B_{\R^3}(0,D_1))))\subset [2D, +\infty)$, which implies
$$
\varphi ([2D,\infty)\times (\R^3\setminus B_{\R^3}(0,D_1)))\subset X\setminus B(O,D)
$$
Thus, the claim is proved since the vector field is unit one on $X\setminus B(O,D)$.

We take the coordinates $[0,\infty)\times (\R^3\setminus B_{\R^3}(0,D_1))$ for $\varphi ([0,\infty)\times (\R^3\setminus B_{\R^3}(0,D_1)))$ and the metric on this set can be written as
$$
g=\frac{d\rho^2}{|\nabla \rho|^2}+ g_\rho
$$
where $g_\rho$ is a metric on $\varphi (\{\rho\}\times (\R^3\setminus B_{\R^3}(0,D_1)))$, or equivalently a family metric on $\R^3\setminus B_{\R^3}(0,D_1)$. \\

{\sl Step 3. For all $s\ge 0$, we have $(s+1)^{-6\varepsilon_1}g_{\R^3}\le g_s\le  (s+1)^{6\varepsilon_1}g_{\R^3}
$.}

Let $\psi :[0,1]\to \R^3\setminus B_{\R^3}(0,D_1)$ be a smooth curve on the boundary. We consider 
$$
\Psi(s,v)=\varphi(s, \psi(v))
$$
and
$$
l(s)=\int_0^1 \langle \partial_v \Psi(s,v), \partial_v \Psi(s,v)\rangle
$$
Thus, a direct calculation leads to
\begin{align*}
l'(s)&=  2  \int_0^1\langle \nabla_s \nabla_v \Psi(s,v), \nabla_v \Psi(s,v)\rangle= 2 \int_0^1 \langle \nabla_v \nabla_s \Psi(s,v), \nabla_v \Psi(s,v)\rangle\\
&=2  \int_0^1 \langle \nabla_v \frac{\nabla \rho}{|\nabla \rho|}(\Psi(s,v)), \nabla_v \Psi(s,v)\rangle\\
&=2  \int_0^1 \frac{1}{|\nabla \rho|} \nabla^{(2)}\rho( \nabla_v \Psi(s,v), \nabla_v \Psi(s,v))\\
&=  \int_0^1 -\frac{s}{|\nabla \rho|} Ric( \nabla_v \Psi(s,v), \nabla_v \Psi(s,v))+ \frac{P}{|\nabla \rho|} |\nabla_v \Psi(s,v)|^2
\end{align*}
since $\rho(\Psi(s,v))=s$. Hence, when $s\ge 0$, we have
$$
|l'(s)|\le \frac{6\varepsilon_1}{s+1}l(s)
$$
Therefore, for all $s\ge 0$
$$
l(0)^2 (s+1)^{-6\varepsilon_1}  \le l(s)^2\le l(0)^2  (s+1)^{6\varepsilon_1} 
$$
Finally, the claim follows.\\

{\sl Step 4.  There holds (\ref{conformalfactorestimate1}).}\\

Given $y\in \partial B_{\R^3}(O,a)$ for sufficiently large $a>0$, let $\sigma:[0,1]\to X$ be a smooth curve joining $y=\sigma(0)$ to $O=\sigma(1)$. Let $s_{\max}\in [0,1)$ such that 
$\sigma([0, s_{\max}]) \subset \varphi ([0,\infty)\times (\R^3\setminus B_{\R^3}(0,D_1)))$ and $  \sigma(s_{\max})\in \partial \varphi ([0,\infty)\times (\R^3\setminus B_{\R^3}(0,D_1)))$. We use the coordinates $[0,\infty)\times (\R^3\setminus B_{\R^3}(0,D_1))$ and write $\sigma(t)=(\rho(t), \tilde y(t))$ for $t\in [0, s_{\max}]$. If $\max \rho(\sigma[0,1])\ge a^{\frac{1}{1+3\varepsilon_1}}$, it follows from the fact $|\nabla \rho|\le 1$ that the length of the curve $\sigma$ is bigger than $a^{\frac{1}{1+3\varepsilon_1}}$. Otherwise, we consider the curve 
$\sigma([0, s_{\max}])\subset \varphi([0, a^{\frac{1}{1+3\varepsilon_1}}]\times (\R^3\setminus B_{\R^3}(0,D_1)))$. It is clear that on $\varphi([0, a^{\frac{1}{1+3\varepsilon_1}}]\times (\R^3\setminus B_{\R^3}(0,D_1)))$, we have
$$
g\ge (a^{\frac{1}{1+3\varepsilon_1}}+1)^{-6\varepsilon_1}g_{\R^3}
$$
Therefore, the length of $\sigma([0, s_{\max}])$ is bigger than
$(a^{\frac{1}{1+3\varepsilon_1}}+1)^{-{\varepsilon}}(a-D_1)$. Hence
$$
d(O,y)\ge \min\{  (a^{\frac{1}{1+3\varepsilon_1}}+1)^{-{3\varepsilon_1}}(a-D_1), a^{\frac{1}{1+3\varepsilon_1}}\}\ge \frac{1}{2}a^{\frac{1}{1+3\varepsilon_1}},
$$
provided $a$ is sufficiently large. Therefore, the desired result (\ref{conformalfactorestimate1}) follows. Finally, we prove Theorem \ref{curvaturedecay}.

\end{proof}

\section{A rigidity result for hyperbolic space }\label{Sect:Liouville}


Our goal in this section is to prove the following Liouville type rigidity result for the hyperbolic metric on the upper half Euclidean space of dimension 4.

\addtocounter{atheorem}{-2}
\begin{atheorem}
Let $(X_\infty,g_\infty^+)$ be a $C^{2,\alpha}$ $4$-dimensional conformally Poincare Einstein metric with conformal infinity $(\R^3, dy^2)$. Denote $g_{\infty}$ the corresponding  adapted scalar flat metric. Assume $g_{\infty}= \rho^2 g_{\infty}^+$  satisfies the following conditions:
\begin{enumerate}
\item  $g_\infty$  has vanishing scalar curvature and bounded Riemann curvature with $|\nabla \rho|_{g_{\infty}} \le 1$.
\item There exists some $\varepsilon\in (0,\frac{3}{16})$ such that
$$
\lim_{a\to\infty} \frac{\inf_{y\in \partial B_{\R^3}(O,a)} d(y,O)}{a^{\frac{1}{1+\varepsilon}}}=+\infty
$$
\item There exists some positive constants $\delta\in (\frac{16\varepsilon}{3},1),C>0$ such that  for some fixed point $p$ we have 
$$|Ric[g](y)|\le C(1+d(y,p))^{-\delta}.$$
\item The interior and boundary injectivity radius of $g_\infty$ are positive.
\end{enumerate}
Then  $g_{\infty}^+ $ is the hyperbolic metric, and the compactified metric $g_\infty$ is the flat metric on ${\mathbb R}^4_{+}$
\end{atheorem}
\vskip .2in
\begin{rema}
We remark that for the sequence of adapted metrics which satisfy the conditions in the statement of Theorem \ref{maintheorem}, the blow-up limiting metric $g_{\infty}$ in our contradiction argument in the previous sections satisfies all conditions listed in Theorem \ref{aTheorem Liouville} above by the results in Theorem \ref{curvaturedecay}. In fact, condition (3) in the theorem is satisfied for the much stronger condition $ \delta = 2$ as in the estimate (\ref{decay2}) in Theorem \ref{curvaturedecay}. Very recently, S. Lee and F. Wang \cite{LeeWang} posted a preprint on a rigidity result when the norm of Weyl tensor in some weighted space is sufficiently small.
\end{rema}
Our proof was partially motivated by the strategy used in the earlier
works of Shi-Tian \cite{Shi-Tian} , Dutta-Javaheri \cite{Dutta}  and Li-Qing-Shi.\cite{LQS} in the sense that we will establish that $g^+ :=g_{\infty}^+$ is the hyperbolic metric on the upper half space via Bishop volume estimates.
\vskip .2in

\subsection {An outline of the proof of Theorem \ref{aTheorem Liouville}}
\label{section5.1}
\vskip .1in

The proof we will present in this subsection is technically complicated mainly due to the presence of possible "cut locus" points along the geodesics of the metric. To help with the reading, we will first present an outline of the proof assuming that there are no cut locus points. We will then fill in the details handling the cut locus points and provide a complete proof.

Fix  a point $ p \in X_{\infty}$, $B(t)= B_{g^+}(p,t)$, $\Gamma_t = \partial B(t)$, to prove that
$g^+ = g_{\mathbb H}$, by the Bishop-Gromov theorem it suffices to establish that

\begin{equation}\label{star}       \frac {\volume_{g^+}(\Gamma_t)}{\volume_{g_{\mathbb H}}(\Gamma_t)} \rightarrow 1, \quad\text{as}\quad
 t \rightarrow \infty.
\end{equation} 

We now outline a proof establishing  \eqref{star} under the additional assumption that there are  no locus points along the path joining $p$ to any points in $\Gamma_t$.
 
 In this case, we denote 
 $\tilde g_t = 4 e^{-2t} g^+$, then
$$  \volume_{g^+}(\Gamma_t)= \volume_{\tilde g_t}(\Gamma_t) {(\frac{e^t}{2})}^{3},$$
and
$$\volume_{g_{\mathbb  H}}(\Gamma_t)= \volume_{{\tilde {g_{\mathbb H}}}_t} (\Gamma_t)\, {(\sinh t)}^{3}.$$
Thus \eqref{star} is equivalent to show
\begin{equation}\label{two-star}
 \volume_{\tilde g_t} (\Gamma_t) \rightarrow |\volume(\mathbb S^3)|\quad\text{as}\quad t \rightarrow \infty.
 \end{equation}

We now estimate \eqref{two-star} by a conformal change in coordinate. 

We drop the index $\infty$ and fix a point $p\in X$. Denote $g=g_\infty = \rho^2 g^+$, $r(y)=-\log\frac{ \rho(y)}{2}$, and recall $t(y)=\dist_{g_+}(y,p)$.  Denote also $ u (y) = r(y) - t(y)$. We set $X_\delta:=[0,\delta]\times \p X$ where the first coordinate is  $\rho(y)$ (we could use the flow of $\frac{\nabla \rho}{|\nabla\rho|}$ to define a global coordinate in a neighborhood of $\p X$, see below). We know $|\nabla \rho|=1$ on the boundary and we have a uniform $C^{2,\alpha}$ estimate for $\rho$. Without loss of generality, we assume that $|\nabla \rho|\ge 1/2$ on $X_1$. Hence, we could use the flow of $\frac{\nabla \rho}{|\nabla\rho|}$ to define a global coordinate $[0,1]\times \R^3$ of $X_1$.

We now recall that in the standard model of the hyperbolic ball, $$ r (y) = t (y) = 2\frac{1-|y|}{1+|y|}.$$ This motivates us to  consider the metrics  $\bar g_r =4e^{-2r}g^+$ and
$\tilde g_t =4e^{-2t}g^+,$
$\Gamma_t:=\{y\,:\, t(y)=t\}$ and $\Sigma_r:=\{y\,:\, r(y)=r\}$ and to consider the difference of behavior of them; with the main idea  to estimate
$\volume_{\tilde g_t}(\Gamma_t)$ by
$\volume_{\bar g_r} (\Sigma_r)$.\\

To do so, we write
$$\tilde g_t = {\Psi}^4 \bar g_r,$$
where $\Psi=e^{\frac{1}{2}(r-t)} = e^{\frac {u}{2}}$. Note that
$\Psi$ satisfies the scalar curvature equation
\beq
\label{scalar}
\frac{1}{8} R_{\tilde g} \Psi^5 = - \Delta_{g} \Psi + \frac{1}{8}R_{g} \Psi.
\eeq
Recall in our case, $R_{g}=0$. We denote $\hat B(p,a)=[0,1]\times B_{\R^3}(\hat{p},a)$ be a cylinder in an adapted coordinate generated by flow for the vector field $\frac{\nabla \rho}{|\nabla \rho|}$ where  $B_{\R^3}(\hat{p},a)$ is an Euclidean ball of center $\hat{p}$ and radius equal to $a$ on the boundary.\\

We now make the following claims.


\noindent {\bf (a):}  We first claim $| \nabla  u|_{ g}=  O(1)$ and apply the decay of the Ricci curvature in Theorem \ref{curvaturedecay} to establish 
$$\ds\oint_{\partial (\hat{B}(p,a)\cap \Sigma_r)}  \Psi\frac{\partial \Psi}{\partial \nu}= o(1)\quad\text{as} \quad r,a\to\infty.$$
Detailed proof of assertion (a) is presented in Step 4 in Lemma \ref{Lemma2.8} below.
\vskip .1in

\noindent {\bf (b):} We then show that:
$$\ds  8\int_{{\mathbb{R}^{3}}\cap \hat{B}(p,a)} |\nabla \Psi|^2 \,dvol_{g_{\mathbb{R}^{3}}} \le  \int_{\Sigma_r\cap\hat{B}(p,a)}R_{\tilde g_r} \, dvol_{\tilde g_r}+ o_r(1),$$
and
$$
\int_{{\mathbb{R}^{3}}\cap \hat{B}(p,a)}\Psi^{6} dvol_{g_{\mathbb{R}^{3}}}= \volume_{\tilde g_r}(\Sigma_r\cap \hat{B}(p,a))+o_r(1).$$

Detailed proof of  assertion (b) is presented in Step 6, Lemma \ref{Lemma 2.10}
\vskip .1in

\noindent {\bf (c):} We then show that 
\begin{equation*}
\ds\int_{\Sigma_r\cap \hat{B}(p,a)}R_{\tilde g_r} dvol_{\tilde g_r}
\le\ds 6 vol_{\tilde g_r}(\Sigma_r\cap \hat{B}(p,a))+o_r(1);
\end{equation*}
\begin{equation*}
\ds\lim_{r\to \infty}vol_{\tilde g_r}(\Sigma_r\cap \hat{B}(p,a))
\le\ds \lim_{t\to \infty}\volume_{\tilde g_t}(\Gamma_t)\le \volume(\mathbb S^{3});
\end{equation*}
Detailed proof of assertion (c) is presented in Step 8 Lemma \ref{Lemma6.13}.
\vskip .1in

\noindent {\bf (d):}
On the other hand, we will apply the Sobolev embedding theorem and obtain
$$
\frac{3}{4} |\volume(\mathbb S^3)|^{\frac{2}{3}}
\leq \frac{\displaystyle \int_{\mathbb{R}^3}| \nabla \Psi|^2 \,dvol_{g_{\mathbb{R}^{3}}}}
{\displaystyle \left(\int_{\mathbb{R}^{3}}\Psi^6\, dvol_{g_{\mathbb{R}^{3}}}\right)^{\frac{1}{3} }}.
$$
Detailed proof of assertion (d) is presented in Step 10, Lemma \ref{Lemma6.15}.
\vskip .1in

Combining  assertions (a) to (d) above, we have established  that
$\Psi|_{\mathbb{R}^{3}}$ is an extremal function for the Sobolev embedding $H^1({\mathbb{R}^{3}})\hookrightarrow L^6({\mathbb{R}^{3}})$ and
$$
\ds\lim_{r\to \infty}\volume_{\tilde g_r}(\Sigma_r\cap {B}(p,a))
= \ds \lim_{t\to \infty}\volume_{\tilde g_t}(\Gamma_t) =  \volume(\mathbb S^{3}).
$$
Apply the Bishop's comparison theorem, we then conclude that  $g_+$ is the hyperbolic metric with its Weyl curvature $ W \equiv 0$.

From there, we use the PDE of the adapted metrics to establish that $ g$ is flat.

This contradicts our set up that, via the $\varepsilon$ argument that $C^1 $ norm of the curvature of $ g$ at a point $p$ is one; and completes the outline of the proof of Theorem \ref{aTheorem Liouville}.

\vskip .2in

\subsection {Proof of Theorem \ref{aTheorem Liouville}}
\label{section5.2}

 As mentioned before, our analysis becomes delicate mainly due to the presence of cut-locus points, whose behavior we will now study based on a careful analysis of the Riccati equation. We now present in this subsection a rigorous proof of the Theorem \ref{aTheorem Liouville}.\\

Given a point $p\in X\setminus X_\kappa$ for some small $\kappa<1$, we denote the set $C_p$ where minimizing geodesics from $p$ cease to be unique or minimal and we call $C_p$ the set of cut-locus of a point $p$. Without confusion, we call the set cut loci sometimes, in particular, when the base points vary.\\

To start the argument, we define and consider the function
\beq
\phi(y)= g^+(\nabla_+ r(y), \nabla_+ t(y)).
\eeq
We remark geometrically $\phi$ measures up to a scale cosine of the angle between the vectors $\nabla_+ r$ and $\nabla_+ t$.

Note that in general, although $\rho$ is a $C^{2, \alpha}$ function on X, the function $t (y) = dist_{g^+}(y, p)$ may only be
Lipschitz at points $y$ which are in $C_p$; thus t is not differentiable in general. In the steps 2, 3 below we will first do the formal computation at points $y\in X\setminus C_p$ where the function $t$ is 
$C^2$; we will then justify the computation of those points $y\in C_p$ at the end of step 3 and before step 4 in the argument below. 
 
Let $\gamma: s\in I\subset \R\to (X,g_+)$ be a minimizing geodesic starting from $p=\gamma(0)$ with a unit speed.  We denote by $\phi(s)=\phi\circ\gamma(s)$ (or more generally $f(s)=f\circ\gamma(s)$ for functions $f:X\to\R$) when there is no confusion.

\vskip .2in

\noindent{\underline{\it Step 1.}}  \\ 

We claim at a point where 
$r$ is $C^2$, we have:
\beq
\label{eq2.4}
\phi'(s)=\nabla_+^2 r(\nabla_+ t, \nabla_+ t).
\eeq
To see (\ref{eq2.4}), we notice
\begin{align}
\phi'(s) = \nabla_{\nabla_+ t}  g^+(\nabla_+ r, \nabla_+ t)&=  g^+(\nabla_{\nabla_+ t} \nabla_+ r, \nabla_+ t)+ g^+(\nabla_+ r, \nabla_{\nabla_+ t}\nabla_+ t)\\
&= g^+(\nabla_{\nabla_+ t} \nabla_+ r, \nabla_+ t)=\nabla_+^2 r(\nabla_+ t, \nabla_+ t).    
\end{align}


Recall the set $X_{\delta}$ for small $\delta$ denotes the set $ X_{\delta}:= [\rho \leq \delta ] \times \partial X$.

\begin{lemm}\label{Lemma2.2}

For all  $ y \in X_\kappa\setminus C_p$ with small $\kappa>0$, $r = r(y)$  and for all bounded vectors (for $g^+$) $Z,V\perp \nabla_+ r$,  we have

\beq
\label{eq6.6}
\nabla_+^2 r(Z,V) =\frac12(|\nabla_+ r|^2+1)g^+(Z,V)+ O(e^{-2r})= g^+(Z,V)+ O(e^{-r}),
\eeq
\beq
\nabla_+^2 r(\nabla_+ r, Z)=  O(e^{-2r}),
\eeq
\beq
\label{eq6.8}
\nabla_+^2 r(\nabla_+ r, \nabla_+ r)=\frac12(1-|\nabla_+ r|^2)  |\nabla_+ r|^2=O(e^{-r}),
\eeq
   and
\beq
\label{eq6.9}
|\nabla_+ r|_{g^+}^2=1+O(e^{-r}).
\eeq
\end{lemm}

To prove Lemma \ref{Lemma2.2}, we first recall some curvature formulas under conformal changes, which apply in particular for all the compactified metrics in CCE settings.

\begin{lemm}
\label{shoutenofg}
On a $(X^4, M, g^+)$ CCE setting, $g= \rho^2 g^+$, $\rho = e^{2r}$, then
\beq
\label{eq2.22}
R_g= \frac{1}{4}e^{2r }(-12+ 6\triangle_+ r-6|  \nabla_+ r|^2_{g^+}).
\eeq
\beq
\label{eq2.21}
Ric_g =-3g^++2(\nabla_+^2r+ \nabla_+ r\otimes  \nabla_+ r)+(\triangle_+ r -2|  \nabla_+ r|^2_{g^+})g^+.
\eeq
Combining the two formulas above, we can rewrite the Schouten tensor as
\beq 
\label{shouten}
A_g=(-\frac12-\frac12|  \nabla_+ r|^2_{g^+})g^++ (\nabla_+^2 r+ \nabla_+ r\otimes  \nabla_+ r).
\eeq

\end{lemm}

We now establish Lemma \ref{Lemma2.2}.
\begin{proof}
Recall $g = g_{\infty}$ satisfies condition (1) in Theorem \ref{aTheorem Liouville}, hence 
we
 have $|A|_g=O(1), |Z|_g=O(\rho), |V|_g=O(\rho)$ and $|\nabla_+ r|_g=O(\rho)$ in $X_\kappa$. Thus the claims in the Lemma follow from the expression of the Schouten tensor (\ref{shouten}) in Lemma \ref{shoutenofg}.
 
\end{proof}

\vskip .1in

We now derive some further estimates of the function $\phi$. First we notice that 
\beq
\label{eq2.9}
\nabla_+ t=g^+(\frac{\nabla_+ r}{|\nabla_+ r|_{g^+}}, \nabla_+ t)\frac{\nabla_+ r}{|\nabla_+ r|_{g^+}}+\sqrt{1-\frac{\phi(s)^2}{|\nabla_+ r|_{g^+}^2}}\nu= I( r,t) \,+ \, II(r,t),
\eeq
where 
$$ I = I(r, t) =
\phi(s)\frac{\nabla_+ r}{|\nabla_+ r|_{g^+}^2},
$$
and 
$$ II = II (r,t) = \sqrt{1-\frac{\phi(s)^2}{|\nabla_+ r|_{g^+}^2}}\nu;
$$
where $\nu $ is an unit vector $\perp \nabla_+ r$. Thus applying (\ref{eq2.4}), the properties in Lemma \ref{Lemma2.2} and the fact $|\nabla_+r|_{g^+}=|\nabla \rho|_{ g}\le 1$,
we have
\beq
\label{eq2.10}
\phi'(s)=O(e^{-r})\phi(s)+(1-\phi(s)^2)(1+O(e^{-r})).
\eeq
Hence for some small number $0<\rho_0=e^{-r_0}<\kappa<1$ we have
\beq
\label{eq2.11}
\phi'(s)\ge -\frac18|\phi(s)|+\frac12 (1-\phi(s)^2), 
\eeq
provided the geodesic $\gamma$ starting from $p$ satisfies $\gamma(s)\in  X_{\rho_0}$. Without loss of generality, we assume $p$ is outside  $X_{\rho_0}$.\\

In the hyperbolic model case, the geodesic curve $\gamma$ intersect the conformal infinity at a perpendicular angle.In the following, we are aiming to describe similar behavior of general CCE metric $g^+$ near the conformal infinity.

Thus, once $\gamma$ gets into  $X_{\rho_0}$ at first time $s_0$, we have at the point
$x_0 = \gamma(s_0)$ the angle between the vector $\gamma'(s) = \nabla_+(t)$ and 
$\nabla_+ r = - \frac{1}{\rho} \nabla_+ \rho(\gamma(s)) $ at the point $s =s_0$ is less than $\frac{\pi}{2} $, and 
$ \phi (s_0) = g^+ ( \nabla_+ r, \nabla_+ t) \geq 0$. Furthermore for all $s \geq  s_0$, on either $\phi(s) \ge \frac{1}{2}$ or $\phi (s)< \frac{1}{2}$ while we can see
from (\ref{eq2.11}) that $\phi' > 0$ in a neighborhood of $s$ provided $\phi (s)< \frac{1}{2}$. Thus we conclude that for all $s\ge s_0$,  $\phi(s)\ge 0$. 
This in turn implies $\rho(\gamma (s))$ is decreasing in $s$ when $ s \geq s_0$ and $\phi (s) \in  X_{\rho_0}$ for all $s\ge s_0$. That is, once $\gamma$ gets into $X_{\rho_0}$, it never leaves.\\

\vskip .2in

\begin{figure}[h]
\begin{tikzpicture}[scale=2]
\draw[thick] (-4, -1) to (4, -1);

\draw[<->] (2, -1) to (2, 1);

\draw[thick] (-4, 1) to (4, 1);



\draw[thick] (-2,-1) to [out = 90, in = 200] (1,2);

\node at (1.45, 2) {$p = \gamma(0)$};

\node at (-0.97, 1) {$\bullet$};


\node at (-0.5, 0.8) {$\gamma(s_0)=\gamma(s_1)$};




\node at (2.2, 0) {$X_{\rho_0}$};

\node at (3.5, -1.2) {$M = \partial X = \mathbb{R}^3$};

\node at (3.8, 1.2) {$\Sigma_{r_0}$};

\end{tikzpicture}
\caption{ 
 \begin{equation*}
 \text{Case 1:} \quad 
 \begin{cases}
  \gamma \mbox{ is a geodesic}\\
  \phi=\langle \nabla^+ r,  \nabla^+ t\rangle\\
 	r_0 = - \log \frac{\rho_0}{2}
	\\
\phi(s_0) \geq \frac{1}{2}
 \end{cases}
 \end{equation*} }
\end{figure}

\begin{figure}[h]
\begin{tikzpicture}[scale=2]
\draw[thick] (-4, -1) to (4, -1);

\draw[<->] (2, -1) to (2, 1);

\draw[thick] (-4, 1) to (4, 1);

\draw[thick] (-4, -0.2) to (1, -0.2);

\draw[<->] (0.5, -0.2) to (0.5, -1);

\draw[thick] (-2,-1) to [out = 90, in = 200] (1,2);

\node at (1.45, 2) {$p = \gamma(0)$};

\node at (-0.97, 1) {$\bullet$};

\node at (-1.87, -0.2) {$\bullet$};

\node at (-0.8, 0.8) {$\gamma(s_0)$};

\node at (-1.6, -0.4) {$\gamma(s_1)$};

\node at (0.9, -0.03) {$\Sigma_{r_1}$};

\node at (0.7, -0.6) {$X_{\rho_1}$};

\node at (2.2, 0) {$X_{\rho_0}$};

\node at (3.5, -1.2) {$M = \partial X = \mathbb{R}^3$};

\node at (3.8, 1.2) {$\Sigma_{r_0}$};

\end{tikzpicture}
\caption{ 
 \begin{equation*}
 \text{Case 2:} \quad 
 \begin{cases}
 	r_0 = - \log \frac{\rho_0}{2}
	\\
 	r_1 = - \log \frac{\rho_1}{2}
\\
0 \leq \phi(s_0) < \frac{1}{2}
\\
\phi(s_1) = \frac{1}{2}
 \end{cases}
 \end{equation*} }
\end{figure}

\noindent {\underline {\it Step 2.}}
\vskip .1in

Apply similar arguments if we denote $s_1$ as the first time $s_1>s_0$, 
such that $\gamma(s_1)\in \p X_{\rho_1}$, where $\rho_1 \le  {\rho_0}$ (in the interior boundary) and $\phi(s_1)= 1/2$, otherwise we set $s_1=s_0$ and $\rho_1=\rho_0$ when $\phi(s_0)\ge 1/2$. Then 
\beq
\phi(s)\ge 1/2\; \,\,\, \forall s\ge s_1;
\eeq
and any $\gamma(s)\in   X_{\rho_1}$ when $s \ge s_1.$ \\

Thus, we conclude that if $\gamma(s)$ is a minimizing geodesic $g^+$ from $p$ to $y$ such that $\gamma(0)=p$. 
For all $s\ge s_1$, we have $\phi(s)\ge 1/2$ so that $r(\gamma(s))$ is an increasing function when $s\ge s_1$. We denote $y_1=\gamma(s_1)$. \\

To summarize, from the above argument we conclude that we have established the following lemma.  
\begin{lemm}
\label{Lemma 2.3}
In fact, $\phi(s)$ is bounded and $\forall s\ge s_1$ we have
\beq
\label{eq2.13}
\frac12\le \phi(s)\le |\nabla_+ r|_{g^+}\le 1,
\eeq
where $s_1$ is defined at the beginning of step (2) above.
\end{lemm}
\vskip .1in

We choose an adapted coordinate $(\rho,z_1, z_2,z_3)\in [0,\delta]\times \R^{3}$ in $X_\delta$, where $(z_1, z_2,z_3)$ are extensions of the Euclidean coordinate in $\R^{3}$ by flow for the vector field $\frac{\nabla \rho}{|\nabla \rho|}$.

\begin{lemm}
\label{Lemma 2.4}
Assume $s\ge s_1$, 

\noindent (1) $\forall s\ge s_1$, we have $z(\gamma(s))- z(\gamma(s_1))$ is uniformly bounded for all $y\in
  X_{\rho_1}$ where $z=(z_1, z_2,z_3)$. \\
Moreover, we also have 

\noindent (2) Then $u(\gamma(s))-u(\gamma(s_1))$ is uniformly bounded for all $y\in
  X_{\rho_1}$.

\end{lemm}
\vskip .2in

\begin{proof}  We rewrite the adapted scalar flat metric as:
$$ 
g_{ij}= \delta_{ij}+O(\rho),
$$
where $\delta_{ij}$ is the Kronecker symbol.  We first consider $\rho(\gamma(s))$. By Lemma \ref{Lemma 2.3}, it is clear that for $s\ge s_1$
$$
-\rho(\gamma(s)) \le \frac{d}{ds}\rho(\gamma(s))=- \rho(\gamma(s))  \phi(s)\le -\frac 12\rho(\gamma(s)), 
$$ 
so that 
$$- (s-s_1)\le \log\rho(\gamma(s))-\log\rho(\gamma(s_1))\le -\frac12 (s-s_1).$$
Similarly, for all $\alpha=1,2,3$, we have $$\frac{d}{ds}z_\alpha(\gamma(s))= \langle  \nabla z_\alpha, \gamma'(s)\rangle_{g}=O(|\gamma'(s)|_{g}| \nabla z_\alpha|_{g})=O((\rho(\gamma(s))))= O(e^{-1/2(s-s_1)}), $$
which implies
$$
|z_\alpha(\gamma(s))- z_\alpha(\gamma(s_1))|\le C.
$$
We have thus established part (1)  of Lemma \ref{Lemma 2.4}.\\

We now choose $\rho_2$ and $\rho_3$, so that $0 < \rho_2<\rho_3<\rho_1<1$, and denote $\rho_i = 2e^{-r_i}$ for $ 1 \leq i \leq 3$. For any $y\in \Sigma_{r_3}$ and  $\tilde y\in \Sigma_{r_2}$, let 
$\gamma_1(s)$ be some curve connecting $y$ and $\tilde y$. We can estimate its length for the metric $g^+$
$$
\begin{array}{ll}
l(\gamma_1)&\ds=\int \sqrt{g^+(\gamma_1'(s), \gamma_1'(s))}ds=\int\frac{1}{\rho(\gamma_1(s))}\sqrt{g(\gamma_1'(s), \gamma_1'(s))}ds\\
&\ds\ge \int_{\rho_2}^{ \rho_3}\frac{(1+O(\rho))}{\rho}{d\rho}=\log\frac{ \rho_3} {\rho_2}+O( \rho_3).
\end{array}
$$
Hence $d_{g^+}(y,\tilde y)\ge  \log\frac{ \rho_3} {\rho_2}+O( 1)$. For points $y$ in the interior boundary of $\Sigma_{r_3}$, we denote $y=(\rho_3, z_1,z_2,z_3)$. Let $\gamma_2(s)=(\rho_3-s(\rho_3-\rho_2), z_1,z_2,z_3)$ for $s\in [0,1]$. We compute the length of $\gamma_2$ in the metric $g^+$
$$
l(\gamma_2)=\int^1_0 \sqrt{g^+(\gamma_2'(s), \gamma_2'(s))}ds=\int_{\rho_2}^{ \rho_3}\frac{(1+O(\rho))d\rho}{\rho}=\log\frac{ \rho_3} {\rho_2}+O(1).
$$
This infers
$$
d_{g^+}(y,\Sigma_{r_2})\le \log\frac{ \rho_3} {\rho_2}+O(1).
$$
Gathering the above estimates, we obtain 
$$
d_{g^+}(y,\Sigma_{r_2})= \log\frac{ \rho_3} {\rho_2}+O(1)
$$
for all such points $y$.

Now given points $y \in X_{\rho_1}$ and a fixed point $\tilde y = \Gamma(t_1)$, consider the minimizing geodesic $\gamma$ from $\gamma(s)=y$ to $\gamma(s_1)=\tilde y$. We write $y=(\rho (y), z)$ and $\tilde y=(\rho(\tilde y) , \tilde z)$. From the above analysis, we have 
$$
d_{g^+}(y,\tilde y)\ge  \log\frac{ \rho(y)} {\rho(\tilde y)}+O( 1).
$$
On the other hand, we consider another curve $\tilde\gamma$ from $y$ to $\tilde y$ is the union of the segment from $y$ to $(\rho(y), \tilde z)$ and the segment from $(\rho(y), \tilde z)$ to $\tilde y$. We know $\tilde z- z$ is bounded.  The length of the first segment for the metric $g^+$ is $O(1)$ and that of the second one is equal to $ \log\frac{ \rho_1(y)} {\rho_1(\tilde y)}+O( 1)$. Therefore, we infer
$$
l(\tilde \gamma)=\log\frac{ \rho(y)} {\rho(\tilde y)}+O( 1),
$$
which yields
$$
d_{g^+}(y,\tilde y)=  \log\frac{ \rho(y)} {\rho(\tilde y)}+O( 1).
$$
Thus we have established  part (2) of Lemma \ref{Lemma 2.4} as well,  hence the whole Lemma.
\end{proof}
 
Now we set $\tilde \phi(s)=\ds\frac{\phi(s)}{|\nabla_+ r|_{g^+}},$ which measures the angle between  unit vectors $\frac{\nabla_+ r}{|\nabla_+ r|}$ and $\nabla_+ t$.  We recall the curvature of the Shouten tensor of $g$ in Lemma \ref{shoutenofg}, 
and  get
\beq
\label{eq4.18}
\tilde \phi'(s)+(\tilde \phi(s)^2-1)\frac{1+ |  \nabla_+ r|^2_{g^+}}{2|  \nabla_+ r|_{g^+}}=\frac{1}{|  \nabla_+ r|_{g^+}}A[g]( \nabla_+ t, \nabla_+ t-\frac{\tilde \phi(s)\nabla_+ r}{|  \nabla_+ r|_{g^+}}).
\eeq

\begin{lemm}
\label{Lemma 2.5} 
$\forall s \geq s_1$, 
\beq
\label{eq2.13bisbis}
\tilde \phi(s)=1+ O((1+(s-s_1))e^{-2(s-s_1)}).
\eeq
Moreover, for any $\alpha\in (0,1)$, we have $\forall s \geq s_1$
\beq
\label{eq2.13bisbisbis}
1-(1-\tilde \phi(s_1))e^{-2(s-s_1)}+ O(e^{-(2+\alpha)(s-s_1)})\le  \tilde \phi(s)\le 1.
\eeq
\end{lemm}

\begin{proof}
To prove (\ref{eq2.13bisbis}) in Lemma \ref{Lemma 2.5},  we write 
\beq
\tilde \phi(s)=1+v(s).
\eeq
Apply (\ref{eq4.18}) and the estimates in Lemmas \ref{Lemma 2.3} and  \ref{Lemma 2.4},
 we get 
\beq
\label{eq2.19}
v'(s)=1-(1+v)^2+O(e^{-2(s-s_1)})=-2v-v^2+O(e^{-2(s-s_1)})\ge -\frac32v+O(e^{-2(s-s_1)}).
\eeq
Here we use the facts that $\frac{1+ |  \nabla_+ r|^2_{g^+}}{2|  \nabla_+ r|_{g^+}}=1+O(\rho^2)$ and $A_g( \nabla_+ t, \nabla_+ t-\frac{\tilde \phi(s)\nabla_+ r}{|  \nabla_+ r|_{g^+}})=O(\rho^2)$. Thus we obtain
\beq
(e^{\frac32(s-s_1)}v)'\ge O(e^{-\frac12(s-s_1)}),
\eeq
so that 
\beq
v(s)\ge -Ce^{-\frac32(s-s_1)}.
\eeq
On the other hand,  from (\ref{eq2.13})
$$
v(s)=\tilde \phi(s)-1\le  1-1=0,
$$
Going back to (\ref{eq2.19}), this imples
\beq
v'(s)=1-(1+v)^2+O(e^{-2(s-s_1)})=-2v+O(e^{-2(s-s_1)}).
\eeq
Thus, $\forall s\ge s_1$
\beq
\label{eq2.20}
0\ge v(s)\ge -(C(s-s_1)+C)e^{-2(s-s_1)}.
\eeq
As before, we have
\beq
\begin{array}{ll}
v'(s)=1-(1+v)^2+O(e^{-(2+\alpha)(s-s_1)})
= -2v+O(e^{-(2+\alpha)(s-s_1)}),
\end{array}
\eeq
so that $\forall s\ge s_1$
\beq
0\ge v(s)\ge v(s_1)e^{-2(s-s_1)}+O(e^{-(2+\alpha)(s-s_1)}).
\eeq
Thus we have obtained the desired results (\ref{eq2.13bisbis}) and  (\ref{eq2.13bisbisbis}) in Lemma \ref{Lemma 2.5}.
\end{proof}

\vskip .2in

We observe that as a consequence of step 2, we can also deduct:
\vskip .1in

\noindent {\underline {\it Step 3.}}  
\vskip .1in

Claim: $ \nabla u$ is bounded.\\

To see the claim,  applying  Lemmas \ref{Lemma2.2}, \ref{Lemma 2.4} and \ref{Lemma 2.5}, we derive the  estimate in (\ref{eq2.9}) as 
$$
\nabla_+ t= (1+O(e^{-(s-s_1)}))\nabla_+ r+ O(e^{-(s-s_1)})\nu=(1+O(\rho)) \nabla_+ r+ O(\rho)\nu, 
$$
where $\nu $ is the unit vector with respect to the metric $g^+$ and $\perp \nabla_+ r$. This gives 
$$
|\nabla_+  u|_{g^+}=  O(\rho).
$$
which is equivalent to the desired claim $| \nabla  u|_{ g}=  O(1)$.\\

Because of the existence of cut loci,  the distance function is not regular. Hence it is necessary to understand the fine structure of the set of cut loci and the
behavior of the distance function near the cut loci. For  a fixed point $p\in X$, we denote by $C_p$ (resp. $Q_p$) the set of cut loci (resp. the set of conjugate points) with respect to $p$. We set $A_p:= C_p\setminus Q_p$ the set of non-conjugate cut loci. We divide $A_p$ into two parts: $N_p$ the normal cut loci is the set of conjugate points from which there are exactly two minimal geodesics connecting to $p$ and realizing the distance to $p$ in $(X, g^+)$, and $L_p:= A_p\setminus N_p$ the set of rest conjugate points. Hence we have the disjoint decomposition
$$
C_p=Q_p\bigcup L_p\bigcup N_p.
$$
Recall $X$ is a complete non-compact manifold. Hence, $X\setminus C_p$  is always diffeomorphic to the Euclidean space $\R^d$. Therefore any connected component of $C_p$ extends to the infinity, unless $p$ is a pole, that is, $\exp_p$ is   diffeomorphic from $T_p X$ to $X$. We now recall some results in \cite{IT, O} concerning some general properties of these sets.

\begin{lemm}
\label{Lemma 2.5.1} Assume $(X,g)$ is a complete connected $C^\infty$ Riemannian manifold of dimension $d\ge 2$. For any given $p\in X$,  we have 
\begin{itemize}
\item The closed set  $Q_p\bigcup L_p$   is of Hausdorff dimension no more than $d-2$.
\item The set $N_p$ of normal cut loci consists of possibly countably many disjoint smooth hypersurfaces in $X$. Moreover, at each normal cut locus $q\in N_p$, there is a small open neighborhood $U$ of $q$ such that $U\cap C_p = U\cap N_p$  is a piece of smooth hypersurface in $X$.
\end{itemize}
\end{lemm}

\vskip .1in

Now, we consider the special setting of CCE metric $(X,g^+)$ at a given point $p\in X$.  We can use $[0,\delta]\times \mathbb{R}^{3}$ as the adapted coordinates of $X_\delta$ along the flow $\frac{\nabla_{ g }\rho}{|\nabla_{ g }\rho|}$ as before. Let $\hat{p}\in \mathbb{R}^{3}$ the projection of $p$ on $\p X=\mathbb{R}^{3}$ along the flow $\frac{-\nabla_{ g }\rho}{|\nabla_{ g }\rho|}$. Given $a>0$, we denote $\hat{B}(p,a):= [0,1]\times \bar B_{\mathbb{R}^{3}}(\hat{p},a)$ cylinder where $\bar B_{\mathbb{R}^{3}}(\hat{p},a)$ is the close ball of center $\hat{p}$ and radius $a$ in the euclidean space $\mathbb{R}^{3}$. Given a large number $a>0$ and $r\in (0,+\infty)$, let us denote
$$
\gamma_{r,a}=\hat{B}(p,a)\bigcap (\Sigma_r \bigcap C_p)=( \hat{B}(p,a) \cap\Sigma_r  \cap(Q_p\cup L_p ))\bigcup (  \hat{B}(p,a) \cap\Sigma_r \cap N_p)=\gamma_{r,a}^{QL}\bigcup \gamma_{r,a}^{N}.
$$
We set 
$$
\gamma_{r}=\bigcup_{a>0}\gamma_{r,a},\; \gamma_{r}^{QL}=\bigcup_{a>0} \gamma_{r,a}^{QL},\;   \gamma_{r}^{N}=\bigcup_{a>0} \gamma_{r,a}^{N}.
$$
And we  define a mapping $\Delta: X_1\to  (0,+\infty)\times  \p X= (0,+\infty) \times \R^{3}$
\beq
\Theta(y)=(r(y), \pi(y))=(  r(y), (z_1,z_2, z_3)),
\eeq
where $\pi:  X_1\to \p X$ is the projection along $\rho$ (or $r$) foliation.

By the well-known Eilenberg inequality due to Federer \cite[Theorem 2.10.25]{Federer} and also the results in \cite[Lemmas 5.2 and 5.3]{LQS}, we have the following: 

\begin{lemm}
\label{Lemma 2.5.2} Assume $(X,g^+)$ is a $C^{2}$ AHE  of dimension $d= 4$ and $g=\rho^2g^+$ is the adapted metric. For any given $p\in X$,  we have 
\begin{enumerate}
\item For all $t$ large, the set $\Gamma_t\cap X_{\rho_1}$ is a Lipschitz graph over $\pi(\Gamma_t\cap X_{\rho_1})\subset \p X$;
\item When $t$ is sufficiently large, the outward angle of the corner at the normal cut locus on $\Gamma_t$ is always less than $\pi$;
\item For almost every $r>0$, $\gamma^N_r$ has the locally finite $2$ dimensional Hausdorff measure ${\mathcal H}^{2}$ and consists of possibly countably many disjoint smooth hypersurfaces in $\Sigma_r$.  Moreover, in such case, at each point on $\gamma^N_r$, one has
$$
\frac{\p t}{\p n^+}+ \frac{\p t}{\p n^-}\ge 0.
$$
\item  For almost every $r>0$, the $2$ dimensional Hausdorff measure on $\gamma_{r}^{QL}$ vanishes, i.e.
$$
{\mathcal H}^{2}(\gamma_{r}^{QL})=0.$$
\end{enumerate}
\end{lemm}
\begin{proof}
The results in (1) and (2)  and the key observation-- that is the second part of  (3),  are obtained in \cite[Lemmas 5.2 and 5.3]{LQS}.  It also follows from Lemma \ref{Lemma 2.5.1} that $Q_p\bigcup L_p$   has  locally finite $2$ dimensional Hausdorff measure and the set $N_p$ of normal cut loci consists of possibly countably many disjoint smooth hypersurfaces in $X$. Using the Eilenberg inequality \cite[Theorem 2.10.25]{Federer}, for almost every $r$,  $\gamma_{r}^{QL}$ has locally finite $1$ dimensional Hausdorff measure and $\gamma_{r}^{N}$ has locally finite $2$ dimensional Hausdorff measure. Hence,  ${\mathcal H}^{2}(\gamma_{r}^{QL})=0$ for almost every $r$. We state that $\Sigma_r$ is a $C^{2,\alpha}$ foliation so that we could get the rest result in (3). In fact, we could write locally $\Sigma_r$ as a graph $r=f(z_1,z_2,z_3)$ or $z_1=f(r,z_2, z_3)$ (we can change the index if necessary). In the latter case, the result is obvious. In the first case, by Sard's Lemma, the image of critical points by $f$ has the $1$-dimensional  Hausdorff measure zero. Hence, the desired result (4) follows. 
\end{proof}

\vskip .2in

\noindent {\underline {\it Step 4.}} \\

We denote 
$$
\bar\partial \hat{B}(p,a):=\{y\in \hat{B}(p,a), \|\pi(y)-\hat{p}\|=a\}
$$ 
the lateral boundary of the cylinder $ \hat{B}(p,a)$. It is clear that for any $y^1\in \bar\partial \hat{B}(p,a)$ we have
$$
d_g(y^1, p)\le d_g(y^1, (0,\pi(y^1)))+ d_g( (0,\pi(y^1)), (0,\hat{p}))+d_g( (0,\hat{p}), p)\le a+O(1)
$$
On the other hand, by the assumption (2) in Theorem \ref{aTheorem Liouville}
$$
\frac{a^{1/(1+\varepsilon)}}{d((0,\pi(y^1)), (0,\hat{p}))}=o(1)
$$
Hence, the following {\bf estimate} holds 
\beq
a^{1/(1+\varepsilon)}\le d_g(y^1, p)\le 2a
\eeq
provided $a$ is sufficiently large. Then $s_1\ge 2(1-\beta)\ln a$ for all $\beta>\frac{\varepsilon}{1+\varepsilon}$. In fact, we assume $M=\max_{s\in [0,s_1]} \rho(\gamma(s))$. If $M\le a^{1-\beta}$, then
$s_1\ge a^{-1+\beta} d_{g}(p,y^1)\ge a^{\beta-\frac{\varepsilon}{1+\varepsilon}}\ge 2\ln a$ provided $a$ is sufficiently large. Otherwise, $M\ge a^{1-\beta}$. Let $\bar s \in (0,s_1)$ such that $\rho(\gamma(\bar s))=M$. We denote $y^2=\gamma(\bar s)$.  We could use $\rho$ as parameter of $\gamma(s)$ thus 
$\bar s\ge \int_1^M \frac{1}{\rho}d\rho\ge (1-\beta)\ln a$ since $|\nabla \rho|\le 1$. With the similar argument, $s_1-\bar s\ge (1-\beta)\ln a$. Thus, we prove the claim. \\
Let $\hat{p}$ be the orthogonal projection of $p$ on the boundary $\R^{3}:= \p X_\infty$. With the same arguments, we can prove $s_1\ge \frac{2}{1+\varepsilon}\ln a-\ln (2\ln a)$. Set $\Psi= e^{\frac{1}{2}u}$ so that 
\beq
\label{boundconformal1}
0\le \Psi(y)= e^{\frac{1}{2}(u(y)-u(\gamma(s_1))} e^{\frac{1}{2}u(\gamma(s_1))}\le C  e^{\frac{1}{2}u(\gamma(s_1))} \le C\frac{(2\ln a)^{1/2}}{a^{1/(1+\varepsilon)}}
\eeq
on $\Sigma_r\cap\p  \hat{B}(p,a)$. Also, it follows from Step 3  we have  on $\Sigma_r\cap\p  \hat{B}(p,a)$
\beq
|\nabla\Psi(y)|= \left|\frac{1}{2}\Psi(y)\nabla u(y)\right| \le C\frac{(2\ln a)^{1/2}}{a^{1/(1+\varepsilon)}}.
\eeq
We now will give a more precise estimate. 

\begin{lemm} \label{Lemma2.8}
We assume that there exists some $\delta\in (\frac{16\varepsilon}{3},1),C>0$ such that 
\beq
\label{Riccidecaycondition}
|Ric_g(y)|\le C(1+d(y,p))^{-\delta}
\eeq
 for some fixed point $p$. Then
we have for almost every $r>-\ln \rho_1+2\ln a$ sufficiently large
\beq
\label{boundestimate1}
\oint_{\p (\hat{B}(p,a)\cap \Sigma_r)\setminus C_p}  \Psi\frac{\p \Psi}{\p n}=  O\left(\frac{2\ln a}{a^{\frac{2+\delta(1-\varepsilon)}{1+\varepsilon}-2}}\right), \eeq 
where $n$ is the outwards unit normal vector on the boundary of $\hat{B}(p,a)\cap \Sigma_r$.
\end{lemm}
\begin{proof}
We remark that since $(\Sigma_r, \bar g_r)$ converges to $(\mathbb{R}^{3}, g_{\mathbb{R}^{3}})$ in $C^{2}$ topology, hence for $r$ large, there holds
\beq
\label{boundvolume}
Vol(\p (\hat{B}(p,a)\cap \Sigma_r))=O(a^{2}).
\eeq
Set $\beta:=1-\delta/4$. As above, when $M=\max_{s\in [0,s_1]} \rho(\gamma(s))\le 2a^{1-\beta}$, then $s_1\ge a^{-1+\beta} d_{g}(p,y^1)=a^{\beta-\frac{\varepsilon}{1+\varepsilon}}\ge 4\ln a$ provided $a$ is sufficiently large (since $\beta-\frac{\varepsilon}{1+\varepsilon}>0$),
which yields
$$
\Psi(y)|\nabla\Psi(y)| \le C\frac{1}{a^4}.
$$
When $M=\max_{s\in [0,s_1]} \rho(\gamma(s))\ge 2a^{1-\beta}$, set $\bar s_1\in (0,s_1)$ the biggest $\bar s$ such that $\rho(\gamma(\bar s_1))=a^{1-\beta}$ and 
$\rho(\gamma(s))\le a^{1-\beta}$ for all $s\in (\bar s_1, s_1)$. We know $d_g(p,\gamma(s_1)) \gg a^{\frac{1}{1+\varepsilon}}$ and $  a^{\frac{\delta}{1+\varepsilon}}\gg a^{2(1-\beta)}$ since $\varepsilon <1$. For all $s\in (\bar s_1, s_1)$ there holds
$$
d_g(\gamma(s_1),\gamma(s))\le \int_s^{s_1}|\gamma'|_g\le  a^{1-\beta}\int_s^{s_1}|\gamma'|_{g^+}\le a^{1-\beta}(s_1-\bar s_1)\le a^{2(1-\beta)}\ll d_g(p,\gamma(s_1))
$$
provided $s_1-\bar s_1\le a^{1-\beta}$. If $s_1-\bar s_1\ge 4\ln a$, it is done as before. We consider $s_1-\bar s_1\le 4\ln a\le a^{1-\beta}$. By (\ref{shouten}) and (\ref{Riccidecaycondition}), we obtain
\beq
\begin{array}{cl}
\phi'(s)+\phi^2(s)&=\frac12(1+|\nabla_+ r|_{g^+}^2)+A[g](\nabla_+ t, \nabla_+t)\\
&\ge \frac12+A[g](\nabla_+ t, \nabla_+t)\ge \frac12- O(\rho^{2}a^{-\frac{\delta}{1+\varepsilon}})\\
&\ge \max\{\frac12, \phi^2(s)\}- O(a^{2 -2\beta}a^{-\frac{\delta}{1+\varepsilon}})\ge\frac12- O(a^{-\frac{\delta(1-\varepsilon)}{2(1+\varepsilon)}})\ge 3/8,
\end{array}
\eeq
provided $a$ is sufficiently large. Here we use the fact $\phi^2(s)\le |\nabla_+ r|_{g^+}^2$. We remark $\phi'(s)\ge -O(a^{2 -2\beta}a^{-\frac{\delta}{1+\varepsilon}})= - O(a^{-\frac{\delta(1-\varepsilon)}{2(1+\varepsilon)}})$. Hence $\forall \tilde s>s\in (\bar s_1,s_1)$, we have $\phi(\tilde s)- \phi(s)>-O(a^{-\frac{\delta(1-\varepsilon)}{2(1+\varepsilon)}}(s_1-\bar s_1))\ge -o(1)$ provided $a$ large. We know $\phi(\bar s_1)=(-\ln (\rho(\gamma)))'(\bar s_1)\ge 0$.
On the other hand
$$
(\ln (\rho(\gamma(s))))'=-\phi(s)
$$
which implies
$$
\ln (\rho(\gamma(\bar s_1)))-\ln (\rho(\gamma( s_1)))=\int_{\bar s_1}^{s_1}\phi(s)ds\le (s_1-\bar s_1).
$$
Thus, $s_1-\bar s_1\ge (1-\beta)\ln a+ O(1)$ since $\phi(s)\le 1$. We claim 
$$
\sup_{s\in [\bar t_1, \bar t_1+15]}\phi(s)\ge \frac{11}{20}
$$
provided $a$ large. Otherwise, $\forall s\in [\bar t_1, \bar t_1+15]$, we have $ -\frac 12 \le \phi(s)\le  \frac{11}{20}$. Therefore $\phi'(s)\ge \frac{29}{400}$ and we infer
$$
\phi(\bar s_1+15)\ge \phi(\bar s_1)+\int_{\bar s_1}^{\bar s_1+15}\phi'(s)ds>1.
$$
The contradiction yields the desired claim. Without loss of generality, we assume $\phi({\bar s_1})\ge \frac{11}{20}$ and $\rho(\phi({\bar s_1}))\ge \lambda a^{1-\beta}$ for some $\lambda>e^{-15}$. Thus, there holds for all $s\in (\bar s_1, s_1)$ $
\phi_{ s}\ge \frac{1}{2}$. Together with
$$
\ln (\rho(\gamma(\bar s_1)))-\ln (\rho(\gamma( s_1)))=\int_{\bar s_1}^{s_1}\phi(s)ds\ge \frac 12(s_1-\bar s_1).
$$
we obtain  $s_1-\bar s_1\le 2(1-\beta)\ln a+ O(1)$. On the other hand, due to (\ref{eq4.18}), for all $s\in (\bar s_1, s_1)$, we have
\beq
\tilde \phi'(s)+(\tilde \phi(s)^2-1)\frac{1+ |  \nabla_+ r|^2_{g^+}}{2|  \nabla_+ r|_{g^+}}=O(a^{-\frac{\delta(1-\varepsilon)}{2(1+\varepsilon)}})\sqrt{1-\tilde \phi(s)^2}.
\eeq
As before, we write $\tilde \phi(s)=1+v(s)$ and obtain for all $s\in (\bar s_1, s_1)$
\beq
v'(s)\ge -2v-v^2 -ca^{-\frac{\delta(1-\varepsilon)}{2(1+\varepsilon)}}\sqrt{-2v}\ge -3/2v -ca^{-\frac{\delta(1-\varepsilon)}{2(1+\varepsilon)}}\sqrt{-2v},
\eeq
which gives
\beq
\sqrt{-v(s_1)}\le ca^{-\frac{\delta(1-\varepsilon)}{2(1+\varepsilon)}}+ e^{-\frac32(s_1-\bar s_1)}.
\eeq
On the other hand, we can estimate $s_1-\bar s_1\ge (1-\beta)\ln a=\frac{\delta}2 \ln a$ so that
\beq
-v(s_1)\le ca^{-\frac{\delta(1-\varepsilon)}{1+\varepsilon}}.
\eeq
In view of Lemma \ref{Lemma 2.5}, we have for  $r>-\ln \rho_1+2\ln a$, 
\beq
\frac{\p \Psi}{\p n}= O\left(\frac{(2\ln a)^{1/2}}{a^{\frac{1+\delta(1-\varepsilon)}{1+\varepsilon}}}\right),
\eeq
due to the fact that $g(\nabla \rho, n)=0$ on  the boundary of $\hat{B}(p,a)\cap \Sigma_r$. Together with (\ref{boundconformal1}) and (\ref{boundvolume}), we infer (\ref{boundestimate1}). Finally, we have thus finished the proof of the Lemma.
\end{proof}

\vskip .2in

\noindent {\underline {\it Step 5.}} \\

  By the Gauss-Codazzi equation, we have
\begin{lemm}
On $X_1\bigcap (X\setminus C_p)$, there holds
\label{Lemma 2.9}
\beq
\label{eq2.34}
R_{\tilde g_r}\le \frac{6}{1-e^{-2t}}+e^{2t}(1-\frac{g^+(\nabla_+ r, \nabla_+ t)^2}{|\nabla_+ r|^2_{g^+}})(1-\nabla_+^2 t(\nu,\nu))+e^{2t}O(e^{-4r})
\eeq
for some unit vector $\nu\perp \nabla_+ t$.
\end{lemm}
\begin{proof}
Recall on the $\Sigma_r$, we have the induced metric $g^+_r$ (resp. $\tilde g_r$) from $g^+$ (resp.  $\tilde g=4e^{-2t}g^+$). Let $\{e_1,e_2, e_{3}\}$ be an orthonormal basis on $(\Sigma_r, g^+_r)$ so that $\{\tilde e_1,\tilde e_2, \tilde e_{3}\}= \{\frac{e^t}2e_1,\frac{e^t}2 e_2, \frac{e^t}2 e^t e_{3}\}$  is an orthonormal basis on $(\Sigma_r, \tilde g_r)$. And $N=e^t\nabla_+ r/2|\nabla_+ r|_{g^+}$ is the unit normal to the level set $(\Sigma_r, \tilde g_r)$. Let $\tilde \nabla$ be the Levi-Civita connection with respect to $\tilde g$. We could calculate the second fundamental form with respect to the metric $\tilde g$ on $(\Sigma_r, \tilde g_r)$ 
$$
II[\tilde g] (\tilde e_\alpha, \tilde e_\beta)=\tilde g(\tilde\nabla_{\tilde e_\alpha} N, \tilde e_\beta)=\frac{e^t}2|\nabla_+ r|_{g^+}(\nabla_+^2 r(e_\alpha,e_\beta)-g^+(\nabla_+ t,\nabla_+ r)\delta_{\alpha\beta}),
$$
which implies by Lemmas \ref{Lemma2.2} and \ref{Lemma 2.5} that 
\beq
\label{secondform}
h_{\alpha\beta}=II[\tilde g] (\tilde e_\alpha, \tilde e_\beta)=\tilde g(\tilde\nabla_{\tilde e_\alpha} N, \tilde e_\beta)=\frac{e^t}2|\nabla_+ r|_{g^+}(\nabla_+^2 r(e_\alpha,e_\beta)-g^+(\nabla_+ t,\nabla_+ r)\delta_{\alpha\beta}) =O( e^{t-2r}).
\eeq
Here we use the fact in Lemma \ref{Lemma 2.4} that $r(\gamma(s))= t(\gamma(s))-s_1+O(1)$ for all $s\ge s_1$.  Recall the Gauss-Codazzi equation
\beq
\label{Gauss-Codazzi}
R_{\tilde g_r}=R_{\tilde g}-2Ric_{\tilde g} (N,N)+(\sum_\alpha h_{\alpha\alpha})^2-\sum_{\alpha\beta} (h_{\alpha\beta})^2.
\eeq
Together with the equations (\ref{eq2.21}), (\ref{eq2.22}) and (\ref{secondform}), we infer
\beq
\label{scalarest}
R_{\tilde g_r}=e^{2t}(\triangle_+ t-3+(1-\frac{g^+(\nabla_+ r, \nabla_+ t)^2}{|\nabla_+ r|^2})-\nabla_+^2 t(\frac{\nabla_+ r}{|\nabla_+ r|},\frac{\nabla_+ r}{|\nabla_+ r|_{g^+}})+O(e^{-4r})).
\eeq
By the Laplace comparison theorem, we have
\beq
\label{eq2.38}
\triangle_+ t\le 3\coth t.
\eeq
On the other hand, we have the decomposition
\beq
\label{eq2.39}
\frac{\nabla_+ r}{|\nabla_+ r|_{g^+}}= \frac{g^+(\nabla_+ r, \nabla_+ t)}{|\nabla_+ r|_{g^+} }\nabla_+ t+ \sqrt{1-\frac{g^+(\nabla_+ r, \nabla_+ t)^2}{|\nabla_+ r|_{g^+}^2}}\nu,
\eeq
for some unit vector $\nu\perp \nabla_+ t$ and 
\beq
\label{eq2.40}
\nabla_+^2 t(\nabla_+ t, \cdot)=0
\eeq
since $|\nabla_+ t|_{g^+}\equiv 1$.  Gathering (\ref{scalarest}) to (\ref{eq2.40}), we deduce
\beq
\label{scalarestbis}
R_{\tilde g_r}\le \frac{6}{1-e^{-2t}}+ e^{2t}((1-\frac{g^+(\nabla_+ r, \nabla_+ t)^2}{|\nabla_+ r|_{g^+}^2})(1-\nabla_+^2 t(\nu,\nu))+O(e^{-4r})).
\eeq
Thus we have established  the desired result (\ref{eq2.34}).
\end{proof}

\vskip .2in

\noindent {\underline {\it Step 6.} } \\

We know $\Psi$ is a uniformly bounded and uniformly Lipschitz function on $(X\setminus C_p)\bigcap X_1$ so that it could be extended to a uniformly bounded and uniformly Lipschitz function on $\overline{(X\setminus C_p)\bigcap X_1}=\overline{X_1}$.
\begin{lemm} 
\label{Lemma 2.10}
Given any $a>1$ large and $\eta>0$ small, we have for almost every large $r$ there exists some finite union $B_{r,\eta}$ of balls with radius small than $\eta$ covering $\gamma^{QL}_{r,N}$ such that 
\beq
\label{eq2.42}
\begin{array}{ll}
&\ds  8\int_{\bar B_{\mathbb{R}^{3}}(\hat{p},a)} |\nabla \Psi|^2 dvol_{g_{\mathbb{R}^{3}}}\\
 \le &\ds  \int_{(\Sigma_r\cap \hat{B}(p,a))\setminus (B_{r,\eta}\cup \gamma_{r,a}^N)}R_{\tilde g_r} dvol_{\tilde g_r}+ O(\eta)+ O\left(\frac{(2\ln a)}{a^{\frac{2+\delta(1-\varepsilon)}{1+\varepsilon}-2}}\right)+o_r(1).
\end{array}
\eeq
And
\beq 
\label{eq2.43}
\displaystyle  \int_{\bar B_{\mathbb{R}^{3}}(\hat{p},a)}\Psi^{6} dvol_{g_{\mathbb{R}^{3}}}= vol_{\tilde g_r}(\Sigma_r\cap \hat{B}(p,R))+o_r(1).
\eeq
where $o_r(1)$ is independent of $\eta$ and $o_r(1)\to 0$ as $r\to\infty$.
\end{lemm}
\begin{proof}
Because $\rho^2 g^+$ is $C^{2,\alpha}$ and its injectivity radii is bounded below, we have $R_{\bar g_r}$ converges uniformly to $0$ as $r$ tends to $\infty$. Recall $\tilde g= \Psi^{4}  g$. By the conformal change, we infer on $(X\setminus C_p)\bigcap X_1$
$$
R_{\tilde g_r}= \Psi^{-5}(-8\triangle_{\bar g_r} \Psi+ R_{\bar g_r} \Psi). 
$$
On the other hand,  by Lemma \ref{Lemma 2.5.2}, for almost every $r>0$, we have 
$$
\mathcal{H}^{2}(\gamma_{r,a}^{QL})=0,\;\; \mathcal{H}^{2}(\gamma_{r,a}^{N})<\infty.
$$
And  $\gamma_{r,a}^{N}$ consists of possibly countably many disjoint smooth hypersurfaces in $\Sigma_r$. Denote $\eta>0$ some small positive number. Let $B_{r,\eta}=\cup_i B(y^i,r_i)$ be finite the union of the balls with $y^i\in \gamma_{r,a}^{QL}$ covering $\gamma_{r,a}^{QL}$ such that $\sum_i r_i^{2}\le \eta$ and $r_i<1$ for all $i$ (in the following, the geodesic balls are always related to the metric $ g$ if there is no confusion).  We estimate
$$
\begin{array}{ll}
&\ds\int_{(\hat{B}(p,a)\cap \Sigma_r)\setminus (\cup_i B(y^i,r_i)\cup \gamma_{r,a}^{N})} R_{\tilde g_r} dvol_{\tilde g_r}\\
=&\ds\int_{(\hat{B}(p,a)\cap \Sigma_r)\setminus (\cup_i B(y^i,r_i)\cup \gamma_{r,a}^{N})} 
\Psi(-8\triangle_{\bar g_r} \Psi+ R_{\bar g_r} \Psi)dvol_{\bar g_r}\\
=& \ds\int_{(\hat{B}(p,a)\cap \Sigma_r)\setminus (\cup_i B(y^i,r_i)\cup \gamma_{r,a}^{N})} 8|\nabla \Psi|^2+ R_{\bar g_r} \Psi^2dvol_{\bar g_r}\\
&\ds +\oint_{\p((\hat{B}(p,a)\cap \Sigma_r)\setminus (\cup_i B(y^i,r_i)\cup \gamma_{r,a}^{N}))} 8\Psi\frac{\p \Psi}{\p n}d\sigma_{\bar g_r}.
\end{array}
$$
We write
$$
\begin{array}{ll}
&\ds \oint_{\p((\hat{B}(p,a)\cap \Sigma_r)\setminus (\cup_i B(y^i,r_i)\cup \gamma_{r,a}^{N}))} \Psi\frac{\p \Psi}{\p n}d\sigma_{\bar g_r}\\
=&\ds \oint_{(\hat{B}(p,a)\cap \Sigma_r)\cap \p (\cup_i B(y^i,r_i))} \Psi\frac{\p \Psi}{\p n}d\sigma_{\bar g_r}+ \oint_{\p(\hat{B}(p,a)\cap \Sigma_r)\setminus (\cup_i B(y^i,r_i)\cup \gamma_{r,a}^{N})} \Psi\frac{\p \Psi}{\p n}d\sigma_{\bar g_r}\\
&+ \ds \oint_{\gamma_{r,a}^{N}\setminus (\cup_i B(y^i,r_i))} -\frac{1}{2}\Psi^2(\frac{\p t}{\p n^+}+\frac{\p t}{\p n^-})d\sigma_{\bar g_r}.
\end{array}
$$
Applying  Lemmas \ref{Lemma 2.5.2} and \ref{Lemma2.8}, we obtain
$$
\begin{array}{ll}
\ds -\frac{1}{2}\oint_{\gamma_{r,a}^{N}\setminus (\cup_i B(y^i,r_i))}\Psi^2(\frac{\p t}{\p n^+}+\frac{\p t}{\p n^-})d\sigma_{\bar g_r} \le 0\\
\ds\oint_{\p(\hat{B}(p,a)\cap \Sigma_r)\setminus (\cup_i B(y^i,r_i)\cup \gamma_{r,R}^{N})} \Psi\frac{\p \Psi}{\p n}d\sigma_{\bar g_r}=O\left(\frac{(2\ln a)}{a^{\frac{2+\delta(1-\varepsilon)}{1+\varepsilon}-2}}\right).
\end{array}
$$
On the other hand, we have
$$
\begin{array}{ll}
\ds \oint_{(\hat{B}(p,a)\cap \Sigma_r)\cap \p (\cup_i B(y^i,r_i))} \Psi\frac{\p \Psi}{\p n}d\sigma_{\bar g_r}=O(\eta)\\
\ds\int_{(\hat{B}(p,a)\cap \Sigma_r)\cap (\cup_i B(y^i,r_i))} 8|\nabla \Psi|^2=O(\eta),
\end{array}
$$
since $\sum_i r_i^{3}\le \sum_i r_i^{2}\le \eta$. Also, as  $R_{\bar g_r} $ converges uniformly to $0$, we infer
$$
\int_{\hat{B}(p,a)\cap \Sigma_r} | R_{\bar g_r}| \Psi^2dvol_{\bar g_r}= o_r(1).
$$
Gathering all above estimates, we derive that for almost every $r$,
$$
\begin{array}{ll}
&\ds8\int_{\hat{B}(p,a)\cap \Sigma_r} |\nabla \Psi|^2  dvol_{\bar g_r}\\
 \le&\ds  \int_{(\Sigma_r\cap \hat{B}(p,a))\setminus (B_{r,\eta}\cup \gamma_{r,a}^N)}R_{\tilde g_r} dvol_{\tilde g_r}+ O(\eta)+ O\left(\frac{(2\ln a)}{a^{\frac{2+\delta(1-\varepsilon)}{1+\varepsilon}-2}}\right)+o_r(1).
 \end{array}
$$
We know $\Psi$ is a uniformly bounded and uniformly Lipschitz function on $\overline{X_1}$ so that $\Psi|_{\Sigma_r}\to \Psi|_{\mathbb{R}^{3}}$ in $C^0$ topology and weakly in $H^1$ topology on all compact subset. Hence we get
$$
\begin{array}{ll}
\ds\int_{\bar B_{\mathbb{R}^{3}}(\hat{p},a)} |\nabla \Psi|^2dvol_{g_{\mathbb{R}^{3}}}\le \liminf \int_{\hat{B}(p,a)\cap \Sigma_r} |\nabla \Psi|^2  dvol_{\bar g_r},\\
\ds\int_{B(p,a)\cap \R^{3}}\Psi^{6}dvol_{g_{\mathbb{R}^{3}}}=  \int_{B(p,a)\cap \Sigma_r}\Psi^{6}dvol_{\bar g_r}+o(1).
 \end{array}
$$
Which yields the desired results (\ref{eq2.42}) and (\ref{eq2.43}).
\end{proof}

\vskip .2in

\noindent {\underline {\it Step 7.}}\\

 \begin{lemm} On $\R^{3}$, $\Psi $ does not vanish.\end{lemm}
\begin{proof} 
Let $y\in \hat{B}(p,2)\cap X_{\rho_1/2}$ and $\gamma$ be $g^+$ minimizing geodesic from $p$ to $y$.  From the proof of Lemma \ref{Lemma 2.3}, we have 
$$
d_{g^+}(p,y)=-\log \rho(y)+O(1)=O(1).
$$
Hence $u(y)$ is bounded on $ \hat{B}(p,2)\cap X_{\rho_1/2}$. As  a consequence the Lipschitz function $\Psi(y)= e^{u(y)/2}$ is bounded below by  some positive constant on $  B_{ g}(p,2)\cap X_{\rho_1/2}$. Hence,  $\Psi $ does not vanish on the boundary $\R^{3}$.
\end{proof}

\vskip .2in

\noindent {\underline{\it Step 8.} } \\
 
\begin{lemm} 
\label{Lemma6.13}
For $p\in X$ fixed before and $a>1$ fixed, we have
\beq
\lim_{r\to \infty}vol_{\tilde g_r}(\Sigma_r\cap \hat{B}(p,a))\le \lim_{t\to \infty}vol_{\tilde g_t}(\Gamma_t)\le vol(\S^{3}).
\eeq
\end{lemm}

\begin{proof}  
We know $(X,g^+)$ is Einstein. Thus, it is smooth and the exponential map is smooth. As same arguments as in Lemma \ref{Lemma 2.5.1}, for almost every $t>0$, 
$$
\mathcal{H}^{3}(\Gamma_t\cap C_p)=0.
$$
In view of  Lemma \ref{Lemma 2.5}, we can estimate for any $q\in X_{\rho_1}\setminus C_p$,  we infer (see also \cite[Lemma 6.1]{Dutta})
$$
\|\pi_*\tilde g_r- \pi_*\bar g_t\|_{g_{\mathbb{R}^{3}}}\le C e^{-r}.
$$
Hence, thanks of Lemma  \ref{Lemma 2.10}, we derive
\beq 
\begin{array}{ll}
&\displaystyle  \int_{\bar B_{\mathbb{R}^{3}}(\hat{p},a)}\Psi^{6} dvol_{g_{\mathbb{R}^{3}}}\\
=&\ds\lim_{r\to \infty} vol_{\tilde g_r}(\Sigma_r\cap \hat{B}(p,a))=\lim_{r\to \infty} vol_{\tilde g_r}((\Sigma_r \setminus C_p)\cap \hat{B}(p,a))\\
=& \ds\lim_{t\to \infty}vol_{\tilde g_t}((\Gamma_t \setminus C_p)\cap \hat{B}(p,a)\cap X_{\rho_1})\le \liminf_{t\to \infty}vol_{\tilde g_t}(\Gamma_t ).
\end{array}
\eeq
On the other hand, by Bishop's comparison theorem, for all $t>0$, there holds
$$
\frac{vol_{g^+}(\Gamma_t)}{vol_{g^{\mathbb{H}}}(\Gamma_t)}\le \frac{vol_{g^+}(B_{g^+}(p,t))}{vol_{g^{\mathbb{H}}}(B_{g^{\mathbb{H}}}(0,t))}\le 1.
$$
Here we use $\Gamma_t$ to denote geodesic sphere for the various metrics if there is no confusion. Therefore, 
$$
\lim_{t\to \infty}vol_{\tilde g_t}(\Gamma_t)\le \lim_{t\to \infty} e^{3t}vol_{g^{\mathbb{H}}}(\Gamma_t)=\lim_{t\to \infty}\frac{e^{3t}}{\sinh^{3}t} vol(\S^{3})=vol(\S^{3}).
$$
Thus we have finished the proof of the Lemma.
\end{proof}

\vskip .2in

\noindent {\underline{\it Step 9.} }\\
 
\begin{lemm}
\label{Lemma 2.13}
For fixed $p\in X$ and any fixed large $a>1$, then for any small $\eta>0$ and for almost all large $r>0$ with $\mathcal {H}^{2}(\gamma_{r,R}^{QL})=0$, we have
\beq
\int_{(\Sigma_r\cap \hat{B}(p,a))\setminus(B_{r,\eta}\cup\gamma_{r,a}^{N})}R_{\tilde g_r} dvol_{\tilde g_r}\le 6 vol_{\tilde g_r}((\Sigma_r\cap \hat{B}(p,a))\setminus(B_{r,\eta}\cup\gamma_{r,a}^{N}))+o_r(1)
\eeq
where $B_{r,\eta}$ is given in Lemma \ref{Lemma 2.10}, $o_r(1)$ is independent of $\eta$ and $o_r(1)\to 0$ as $r\to\infty$.
\end{lemm}

\begin{proof} The proof is similar as that one \cite[Theorem 6.8]{LQS}. We sketch the proof.  Let $\gamma(s)$ be a minimizing geodesic w.r.t. $g^+$ connecting $p$ to $q\in \hat{B}(p,a)\cap X_{\rho_1/2}$. Hence the corresponding $s_1$ with $\gamma(s_1)\in \p X_{\rho_1}$ is bounded above by some constant $C>0$ depending on $a$ so that $u$ is bounded by Lemma \ref{Lemma 2.4}.  Moreover, there exists $C>0$ such that $\forall s>s_1$ we have $d_{ g}(p, \gamma(s))\le C$.  We denote ${\tilde  S}=\nabla_+^2 t$ the shape operator of the geodesic sphere along $\gamma(s)$. $S$ satisfies the following Riccati equation
\beq
\label{Riccati1}
 \nabla_{\nabla_+ t} \tilde S+ {\tilde S} ^2=-R_{\nabla t}^+,
\eeq
 where $ R_{\nabla t}^+(\nu)=R(\nu,\nabla_+ t)\nabla_+ t$ for any tangent vector $\nu$ orthogonal to $\gamma'( s)$. Therefore, for all $s>s_1$ and the principal curvature $\mu$, we have 
 $$
 1-Ce^{-2(s-s_1)}\le \mu'(s)+  \mu^2(s)\le 1+ Ce^{-2(s-s_1)},
 $$
 which gives
\beq
\label{Riccati2}
1-C_0e^{-2s}\le \mu'(s)+  \mu^2(s)\le 1+ C_0e^{-2s}
\eeq 
where $C_0>0$ is some constant depending on $R$. Let us denote by $\mu_m$ and $\mu_M$ the smallest and the biggest principal curvature of the geodesic sphere $\Gamma_t$ respectively.  Remark $H=\triangle t$ is the mean curvature of the geodesic sphere $\Gamma_t$. The main idea in \cite{LQS} is the following: when $H$ is close to $3$,  all principal curvatures are close to 1; when $H$ is away from $3$, the integral on the such set $\int (1-\mu_m) \le o_r(1)$.  The key ingredients are based on a careful analysis of Riccati equations (\ref{Riccati1}) and (\ref{Riccati2}). We divide the proof into four sub-steps below .

We set $t_0=s_0$ and $t_1=s_1$.\\ 

\underline{\it Step a. } Given a large $a>0$ and $t_0$ with $(\Gamma_{t_0}\setminus C_p)\cap \hat{B}(p,a)\cap X_{\rho_1}\neq \emptyset$, there exists some $C>0$ such that for all $q\in (\Gamma_{t_0}\setminus C_p)\cap \hat{B}(p,a)\cap X_{\rho_1}$ we have (see \cite[Lemma 6.1]{LQS})
\beq
\label{eq2.49}
\mu_M(q)\le C.
\eeq
Moreover $\forall t>t_0$, we have for all $q\in (\Gamma_{t}\setminus C_p)\cap \hat{B}(p,a)\cap X_{\rho_1}$
\beq
\label{eq2.50}
\mu_M(q)\le 1+C(t+1)e^{-2t}.
\eeq
\underline{\it Step b. } For fixed $p\in X$ and any fixed large $a>1$, given any $1>\kappa>0$ small fixed, we set for large $r>-\log \rho_1$
\beq
U_{r,a}^\kappa=\{q\in (\Sigma_r\setminus C_p)\cap \hat{B}(p,a)\;: \;H(q)\le 3(1-\kappa)\}.
\eeq
Then we have $\ds \int_{U_{r,a}^\kappa} d vol_{\tilde g_r}\to 0$ as $r\to\infty.$ (see \cite[Proposition 6.5]{LQS}). \\
\underline{\it Step c. } For fixed $p\in X$ and any fixed large $a>1$, given any $1>\kappa>0$ small fixed, there exists a constant $C>0$ such that for $q\in (\Gamma_t\setminus  C_p)\cap \hat{B}(p,a)$ with $H(q)\le -6$ and $\gamma$ being the minimizing geodesic from $p$ to $q$, there holds
\beq
\label{eq2.52}
Hdvol_{\tilde g_t}\ge Ce^{-\frac{\kappa t}{4}},
\eeq
provided $t>t_\kappa+1$ where $t_\kappa$ is some positive number depending on $\kappa, R$. (see \cite[Proposition 6.7]{LQS})\\
\underline{\it Step d. }  Recall $u$ is bounded on $\hat{B}(p,a)\setminus C_p$. By Lemma \ref{Lemma 2.5}, it follows from (\ref{eq2.34}) on $\hat{B}(p,a)\setminus C_p$
\beq
\label{eq2.53}
R_{\tilde g_r}\le 6+C(1-\nabla_+^2 t(\nu,\nu))+o_r(1).
\eeq
where $C>0$ is some constant independent of $r$ and $o_r(1)\to 0$ as $r\to \infty$. On the other hand, by (\ref{eq2.50}),  we have on $\hat{B}(p,a)\setminus C_p$
\beq
\label{eq2.54}
\nabla_+^2 t(\nu,\nu)\ge \mu_m\ge H-C.
\eeq
Fix a small $\kappa>0$ and $\eta>0$. We divide the set $(\Sigma_r\cap \hat{B}(p,a))\setminus(B_{r,\eta}\cup\gamma_{r,a}^{N})$ into three subsets
$$
(\Sigma_r\cap \hat{B}(p,a))\setminus(B_{r,\eta}\cup\gamma_{r,a}^{N}):= A_{\kappa, \eta, 1}\cup A_{\kappa, \eta, 2}\cup A_{\kappa, \eta, 3},
$$
where 
$$
\begin{array}{ll}
A_{\kappa, \eta, 1}:=\{q\in (\Sigma_r\cap \hat{B}(p,a))\setminus(B_{r,\eta}\cup\gamma_{r,a}^{N}): \; H(q)>3(1-\kappa)\},\\
A_{\kappa, \eta, 2}:=\{q\in (\Sigma_r\cap \hat{B}(p,a))\setminus(B_{r,\eta}\cup\gamma_{r,a}^{N}): \; -6<H(q)\le 3(1-\kappa)\},\\
A_{\kappa, \eta, 2}:=\{q\in (\Sigma_r\cap \hat{B}(p,a))\setminus(B_{r,\eta}\cup\gamma_{r,a}^{N}): \; H(q)\le -6\}.
\end{array}
$$
On $A_{\kappa, \eta, 1}$, we infer by  (\ref{eq2.50}) 
$$
1-\nabla_+^2 t(\nu,\nu)\le 1-\mu_m\le 1-H+2\mu_M \le o_r(1)+O(\kappa). 
$$
On  $A_{\kappa, \eta, 2}\cup A_{\kappa, \eta, 3}$, we use (\ref{eq2.54}). 
Thus
we can estimate by (\ref{eq2.52}) to (\ref{eq2.54})
$$
\begin{array}{ll}
\ds\int_{A_{\kappa, \eta, 1}}  R_{\tilde g_r} dvol_{\tilde g_r}&\le \ds6 \int_{A_{\kappa, \eta, 1}}   dvol_{\tilde g_r} +O(\kappa)+ o_r(1),\\
\ds\int_{A_{\kappa, \eta, 2}}  R_{\tilde g_r} dvol_{\tilde g_r}&\le \ds6 \int_{A_{\kappa, \eta, 2}}   dvol_{\tilde g_r} +Cvol_{\tilde g_r}(U_{r,a}^\kappa) + o_r(1)\\
&=\ds 6 \int_{A_{\kappa, \eta, 2}}   dvol_{\tilde g_r} + o_r(1),\\
\ds\int_{A_{\kappa, \eta, 3}}  R_{\tilde g_r} dvol_{\tilde g_r}&\le \ds6 \int_{A_{\kappa, \eta, 3}}   dvol_{\tilde g_r} +Cvol_{\tilde g_r}(U_{r,a}^\kappa) + o_r(1)\\
&+\ds \int_{A_{\kappa, \eta, 3}}   Hdvol_{\tilde g_r}\\
&=\ds 6 \int_{A_{\kappa, \eta, 3}}   dvol_{\tilde g_r} +O(\kappa)+ o_r(1),
\end{array}
$$
which yields 
$$
\int_{(\Sigma_r\cap \hat{B}(p,a))\setminus(B_{r,\eta}\cup\gamma_{r,a}^{N})}R_{\tilde g_r} dvol_{\tilde g_r}\le 6 \int_{(\Sigma_r\cap \hat{B}(p,a))\setminus(B_{r,\eta}\cup\gamma_{r,a}^{N})}dvol_{\tilde g_r}+O(\kappa)+ o_r(1).
$$
As $\kappa>0$ could be chosen sufficiently small, we get the desired result in the Lemma.
\end{proof}

\vskip .2in

\noindent {\underline {\it Step 10.}}\\

\begin{lemm}
\label{Lemma6.15}
$g$ is flat
\end{lemm}
\begin{proof}
Gathering Lemmas (\ref{Lemma 2.10}) to (\ref{Lemma 2.13}), for fixed $R>1$ and $\eta>0$ there holds for almost every $r$ large
\beq
\begin{array}{ll}
&\frac{\ds 8\int_{\bar B_{\R^{3}}(\hat{p},a)} |\nabla \Psi |^2dvol_{g_{\mathbb{R}^{3}}}+ O(\eta)+ O\left(\frac{(2\ln a)}{a^{\frac{2+\delta(1-\varepsilon)}{1+\varepsilon}-2}}\right)+o_r(1)}{ \ds (\int_{\bar B_{\R^{3}}(\hat{p},a)}\Psi^{6}dvol_{g_{\mathbb{R}^{3}}}+ O(\eta) +o_r(1))^{\frac{1}{3}}}\\
\le &\frac{\ds  \int_{(\Sigma_r\cap \hat{B}(p,a))\setminus (B_{r,\eta}\cup \gamma_{r,R}^N)}R_{\tilde g_r} dvol_{\tilde g_r}}{\ds vol_{\tilde g_r}(\ds(\Sigma_r\cap \hat{B}(p,a))\setminus (B_{r,\eta}\cup \gamma_{r,a}^N))^{1/3}}\\
=&6 vol_{\tilde g_r}((\Sigma_r\cap \hat{B}(p,a))\setminus(B_{r,\eta}\cup\gamma_{r,a}^{N}))^{\frac{2}{3}}+o_r(1)\\
=&6 vol_{\tilde g_r}(\Sigma_r\cap \hat{B}(p,a))^{\frac{2}{3}}+o_r(1)+o(\eta)\\
\le &6 (vol(\S^{3}))^{\frac{2}{3}}+o_r(1)+o(\eta).
\end{array}
\eeq
Taking the limit successively as $r\to\infty$, $\eta\to 0$ and $a\to \infty$, we derive
\beq
\int_{\R^{3}} |\nabla \Psi |^2dvol_{g_{\mathbb{R}^{3}}}/( \int_{\R^{3}}\Psi^{6}dvol_{g_{\mathbb{R}^{3}}})^{\frac{1}{3}}\le \frac{3}{4}vol(\S^{3})^{\frac{2}{3}}, 
\eeq
since $\frac{2+\delta(1-\varepsilon)}{1+\varepsilon}-2>0$. Recall for any $v\in H^1(\R^{3})$
\beq
\int_{\R^{3}} |\nabla v |^2/( \int_{\R^{3}}|v|^{6})^{\frac{1}{3}}\ge \frac{3}{4}vol(\S^{3})^{\frac{2}{3}}.
\eeq
Thus, $\Psi$ is extremal function for the Sobolev embedding. Hence $g$ is flat. To see this, we have
\beq
\frac{3}{4}vol(\S^{3})^{\frac{2}{3}}= \frac{3}{4}vol(\S^{3})^{\frac{2}{3}}\lim_{t\to \infty}\frac{vol_{ g^+}(\Gamma_t)^{\frac{2}{3}}}{vol_{ g^\H}(\Gamma_t)^{\frac{2}{3}}}.
\eeq
which implies for all $t>0$ by Bishop's comparison theorem
\beq
vol_{g^+}(B(p,t))=vol_{g^\H}(B(p,t)).
\eeq
Hence $g^+$ is hyperbolic space form. 
\vskip .1in

\noindent {\underline {\it Step 11.}}
\vskip .1in

We now claim that $g$ is the flat metric on the upper half space $\R^4_+=\{(w,z)| w>0, z\in \R^3\}$. To see so, we write $g^+=\frac{dw^2+ dz^2}{w^2}$ the half space model for hyperbolic space where $w>0$ and $z\in \R^3$.

Thus 
\beq
\label{flatg}
g=\rho^2 g^+=\frac{ \rho^2}{w^2} (dw^2+ dz^2)= v^2(w,z) (dw^2+ dz^2)=v^2 g_{\R^4}.
\eeq
We next observe that since the scalar curvature of g is zero, from (\ref{flatg}), we deduce that $v$ is a harmonic function in the half space $\R^4_+$. Moreover, we have $v|_M\equiv 1$, and from Lemma \ref{freescalar}, we also have
$$
\|d \rho\|_g\le 1.
$$
Thus
$$
w^2\|d \rho\|_{g_{\R^4}}\le \rho^2,
$$
which implies $\p_w \ln \frac{\rho}{w}\le 0$. Thus we conclude  
$ v = \frac{ \rho}{w}\le 1$ in $ \R^4_+$ and $\|d \rho\|_{g_{\R^4}}\le 1$. Hence $v-1$ is a bounded harmonic function on the half space  $\R^4_+$ which vanishes on the boundary $\R^3$. By an odd extension of v to the lower half space $\R^4_-$, we obtain a bounded harmonic function on $\R^4$. Thus $v\equiv 1$, or equivalently by (\ref{flatg}) $g$ is the flat metric.
\end{proof}

\vskip .2in

\section{Proof of Theorem \ref{maintheorem}}
\label{section 7bis}
In this section, we will apply results from the previous sections to finish the proof of our main result Theorem \ref{maintheorem}.
We start by stating that, applying proofs similar to the corresponding results in \cite{CGQ}, we can obtain estimates of the lower bound of the injectivity radius. 

\begin{lemm} 
\label{bdy-injrad}
Suppose that $(X^4, g^+)$ is a conformally compact Einstein 4-manifold with the conformal infinity of Yamabe constant $Y(\p
X, [\hat g]) \ge C_0 >0$. And suppose that the scalar flat adapted metric $g$ with associated  metric $\hat g$ on the boundary satisfying  that $Rm_g $ is bounded;  the intrinsic injectivity radius $i(\p X, \hat g) \ge i_o>0$, and also that the boundary injectivity radius $ i_\p (X, g)\le  i_{\text{int}} (X, g)$ (the interior injectivity radius). Then there is a constant
$C_\p  > 0$, depending of $i_0$, such that the boundary injectivity radius $i_\p (X, g)$
\begin{equation}\label{Eq: bdy-injrad}
\max_X |Rm_g| (i_\p (X, g))^2  + i_\p (X, g) \ge C_\p.
\end{equation}
\end{lemm}

\begin{proof}  Same proof as \cite[Lemma 3.1]{CGQ} applies.

\end{proof}

\begin{lemm} 
\label{int-injrad}
Suppose that $(X^4, g^+)$ is a conformally compact Einstein 4-manifold with the conformal infinity of Yamabe constant $Y(\p
X, [\hat g]) \ge C_0 >0$.  And suppose that the scalar flat adapted metric $g$ with associated   
metric $\hat g$ on the boundary satisfy  $Rm_g $ being bounded, the intrinsic injectivity radius $i(\p X, \hat g) \ge i_o>0$  and  
 that $ i_\p (X, g)\ge  i_{\text{int}} (X, g)$. Then there is a constant
$C_{\text{int}} > 0$, depending on $Y_0$ and $i_0$, such that 
\begin{equation}\label{Eq: int-injrad}
\max_X |Rm_g|(i_{\text{int}} (X, g))^2  + i_{\text{int}} (X, g) \ge C_{\text{int}}.
\end{equation}

\end{lemm}
\begin{proof} 
Same proof as \cite[Lemma 3.2]{CGQ} applies with minor modifications.

\end{proof}

\vskip .1in



\begin{lemm} \label{Lem:curv-estimate} 
Suppose that $\{(X_i^4, g^+_i)\}$ is a sequence of conformally compact Einstein 4-manifolds satisfying
the assumptions (1)- (5) in Theorem \ref{maintheorem}. 
Then there is a positive constant $K_0$ such that, for the adapted compactifications $\{(X_i^4, g_i)\}$ 
associated with the  metric $\hat g_i$ of the conformal infinity $(\p X_i, [\hat g_i])$ 
\begin{equation}\label{Eq:curvature-bound}
\max_{X_i} |Rm_{g_i}| +|\nabla Rm_{g_i}|\le K_0
\end{equation}
for all $i$.
\end{lemm}

By the $\varepsilon$ regularity result
established in Corollary  \ref{epsilonregularity1}, to establish the 
$C^{k-3, \beta}$ compactness of the metrics $g_i$ for some $k \geq 6$ and for any $\beta\in (0,1)$, it suffices to prove the family is $C^3$ bounded. To see the latter fact, we will modify an argument similar to those used in  \cite[Lemma 4.1]{CGQ} and \cite[Section 4]{CG}.\\

\begin{proof}[Proof of Lemma \ref{Lem:curv-estimate}]
Assume the family $g_i$ is not $C^1$ bounded, or equivalently, the 
the $C^1$ norm of its curvature tends to infinity as $i$ tends to infinity, that is,
$$
\|Rm_{g_i}\|_{C^1}\to \infty\quad\quad\mbox{ as }i\to \infty
$$
We rescale the metric
$$
\bar g_i= K_i^2 g_i,
$$
where there exists some point $p_i\in X$ such that
$$
K_i^2=\max\{\sup_X |Rm_{g_i}|, {\sup_X |\nabla Rm_{g_i}|^{2/3}}\}= |Rm_{g_i}|(p_i) (\mbox{ or } { |\nabla Rm_{g_i}|^{2/3}(p_i)}).
$$
We mark the point $p_i$ as $0\in X$. Thus, we have 
$$
|Rm_{\bar g_i}|(0)=1 \mbox{ or }  |\nabla Rm_{\bar g_i}|(0)=1.
$$

We denote the corresponding defining function $\bar\rho_i= K_i\rho_i$, that is, $\bar g_i =\bar \rho_i^2g^+_i$. We observe  $C^{1}$ norm for the curvature $\|Rm_{\bar g_i}\|_{C^1}$ is uniformly bounded. Applying $\varepsilon$-regularity (Corollary \ref{epsilonregularity1} and remark),  we obtain $C^{k-5,\beta}$ norm for the curvature $\|Rm_{\bar g_i}\|_{C^{k-5,\beta}}$ is also uniformly bounded for any $\beta\in (0,1)$ \\

Applying Lemmas \ref{bdy-injrad} and \ref{int-injrad},  we conclude there is no collapse of the family of metrics $\{\bar g_i\}$, that is, that is, the volume of any geodesic ball with the radius equal to 1 is uniformly bounded below.  By a version of Cheeger-Gromov-Hausdorff's compactness theorem for the manifolds with boundary (\cite{AKKLT} and \cite[Appendix]{CG}), modulo a subsequence and modulo diffeomorphism group, $\{\bar g_i\}$ converges in pointed Gromov-Hausdorff's sense for $C^{k-3,\beta}$ norm to a non-flat  limiting metric $\bar g_\infty$.\\

We now have two possible types of blow-up; and we will now show each cannot occur. \\

Type (I) : No boundary blow-up.
\vskip .1in
If there exist two positive constants $C,\Lambda>0$ and a sequence of points $p_i\in X$ such that $d_{\bar g_i}(\p X, p_i)\le C$ and $\liminf_i |Rm_{\bar g_i}(p_i)|+|\nabla Rm_{\bar g_i}(p_i)| = \Lambda>0$. In such a case, we translate $p_i$ to $0$ and apply the blow-up analysis. Modulo a subsequence, we can assume
$$
\lim_i|Rm_{\bar g_i}(0)|+ |\nabla Rm_{\bar g_i}(0)|=\Lambda < \infty.
$$
 
Applying Lemmas \ref{boundedwyel}, \ref{freescalar} and Theorem \ref{Yamabe}, Theorem \ref{curvaturedecay} and Theorem \ref{aTheorem Liouville},  we conclude that this cannot occur.\\

Type (II): No interior blow-up.\\

If blow-up of type I does not occur, 
We first observe that  
the limit function $ \bar{\rho}_{\infty}$ of ${\bar \rho_i}$ tends
to infinity in any compact as $i$ tends to infinity. For this purpose,  we see  $\bar \rho_i(y)\ge C d_{\bar g_i}(y,\p X)$ for some uniform constant $C$ provided $d_{\bar g_i}(y,\p X)\ge 1$. Passing to the limit, we obtain the desired result as $\|\nabla \bar \rho_i\|\le 1$. Thus it follows that the limiting metric $\bar g_\infty$ is Ricci flat. (see the  argument in \cite[Lemmas 4.2 and 4.9]{CG}.) 

We remark that to simplify the notations, we will henceforth denote $\bar \rho_{\infty} $
as $\rho_{\infty}$ and 
$\bar g_{\infty}$ as $g_{\infty}$
for the rest of the proof in this section.\\


We will now reach a contradiction when the interior blow-up point exists under the topological assumption (4) in the statement of Theorem \ref{maintheorem}. The argument is in spirit similar to those used in \cite[Section 4]{CG} and  \cite[Section 4]{CGQ}, so we will just indicate the differences in the proof.

Let $E = X_{\infty}$ be an interior blow-up limit space.  As  ${\bar \rho_i}( p_i)$ tends to infinity, E is Ricci flat. By relative Mayer-Vietoris exact sequences for $(X-E, M)$ and $E$, we have
$$
\cdots\to H_2(\p E, \mathbb{Z})\to H_2(X-E, M, \mathbb{Z})\oplus H_2(E, \mathbb{Z})\to  H_2(X,M,\mathbb{Z})  \cdots
$$
When $E$ is Ricci flat, then there exists some finite group $\Gamma$ such that 
$$
\p E=\S^3/\Gamma
$$
We know by Poincar\'e duality and universal coefficient theorem for cohomology
(see \cite[section6]{zhangyongjia})
$$
H_2(\p E, \mathbb{Z})\simeq H^1(\p E, \mathbb{Z})\simeq Hom( H_1(\p E, \mathbb{Z}),\mathbb{Z}) =0,
$$
since $H_1(\p E, \mathbb{Z})$ is a finite torsion group. 
By Poincar\'e-Lefschetz duality, we have $H_2(X,M,\mathbb{Z})=H^2(X, \mathbb{Z})=0$, which implies $H_2(E, \mathbb{Z})=0$ and $H_2(X-E, M, \mathbb{Z})=0$. 
Using \cite[Lemma 12]{CG}, we have $H_3(E, \mathbb{Z})=0$. Applying the universal coefficient theorem for cohomology and Poincar\'e-Lefschetz duality, we obtain
$$
H_1(E,\p E, \mathbb{Z})\simeq H^3 (E, \mathbb{Z})\simeq Hom(H_3(E, \mathbb{Z}), \mathbb{Z})\oplus Ext(H_2(E, \mathbb{Z}), \mathbb{Z})=0
$$
Using a long exact sequence, we infer
$$
0=H_2(E, \mathbb{Z})\to H_2(E,\p E, \mathbb{Z})\to H_1(\p E, \mathbb{Z})\to H_1(E, \mathbb{Z})\to H_1(E,\p E, \mathbb{Z})=0
$$
As $ H_1(\p E, \mathbb{Z})$ is finite, we conclude from above homology  sequence that $H_1(E, \mathbb{Z})$ is a finite torsion. 
Again applying the universal coefficient theorem for cohomology and Poincar\'e-Lefschetz duality, we infer
$$
H_2(E,\p E, \mathbb{Z})\simeq H^2 (E, \mathbb{Z})\simeq Hom(H_2(E, \mathbb{Z}), \mathbb{Z})\oplus Ext(H_1(E, \mathbb{Z}), \mathbb{Z})=H_1(E, \mathbb{Z})
$$
Hence, the above sequence becomes
$$
0\to H_1(E, \mathbb{Z})\to H_1(\p E, \mathbb{Z})\to H_1(E, \mathbb{Z})\to 0
$$
On the other hand, by relative Mayer-Vietoris exact sequences for $(X-E, M)$ and $E$, we have
$$
\cdots0= H_2(X,M,\mathbb{Z}) \to H_1(\p E, \mathbb{Z})\to H_1(X-E, M, \mathbb{Z})\oplus H_1(E, \mathbb{Z})\to  H_1(X,M,\mathbb{Z}) =0 \cdots
$$
that is, $H_1(\p E, \mathbb{Z})=H_1(X-E, M, \mathbb{Z})\oplus H_1(E, \mathbb{Z})$ splits. Hence, 
$$
H_1(\p E, \mathbb{Z})\simeq H_1(E, \mathbb{Z})\oplus H_1(E, \mathbb{Z})
$$
It follows from \cite[Lemmas 6.3-6.5]{zhangyongjia} that $b_2(E)\ge 1$.
More precisely, we have
\begin{lemm}
\cite[Lemmas 6.3-6.4]{zhangyongjia}
\label{ZhangYJ1}
Let $\S^3\setminus\Gamma$ be a round space form, where $\Gamma$ is a finite group. If $H_1 (\S^3\setminus\Gamma, \mathbb{Z}) \simeq G \oplus G$, for some group $G$, then either $\Gamma$ is the binary dihedral
group $D_n^*$ with $n$ being even, or $\Gamma$ is the binary icosahedral group with order $120$. Moreover, there exists a complex structure on $\R^4$ such that $\Gamma < SU (2)$.
\end{lemm}

\begin{lemm} 
\cite[Lemmas 6.5]{zhangyongjia}
\label{ZhangYJ2}
Let $(E, g)$ be an Einstein ALE space which is asymptotic to $\S^3\setminus\Gamma$, where $\Gamma < SU (2)$ is isomorphic to the binary dihedral group $D_{2n}^*$
or to the binary icosahedral group. Then $b_2 (E) \ge 1$.
\end{lemm}
Applying Lemmas \ref{ZhangYJ1} and  \ref{ZhangYJ2}, we obtain $b_2 (E) \ge 1$, which contradicts to the 
fact  $H_2(E, \mathbb{Z})=0$. Thus we conclude there is no interior blow-up. 


We thus conclude from this contradiction argument that under the assumption of the Theorem \ref{maintheorem} that the family of metrics $g_i$ has uniform $C^3$ bound.
\end{proof}

\begin{lemm} 
\label{diameter}
Under the assumptions in Lemma \ref{Lem:curv-estimate} , the diameters of the adapted metrics
$g_i$ are uniformly bounded. 
\end{lemm}
\begin{proof}
By conformal change, it follows 
\beq
\label{freescalarN1}
-\rho_i \triangle_{ g_i} \rho_i=2(1-|\nabla_{ g_i} \rho_i|^2)
\eeq
so that
$$
-\triangle_{ g_i}\log \rho_i=\frac{2(1-|\nabla_{ g_i} \rho_i|^2)}{\rho_i^2}+ \frac{|\nabla_{ g_i} \rho_i|^2}{\rho_i^2}
$$
From here, we may apply the same argument as \cite[Lemma 4.2]{CGQ} to finish the proof of the Lemma.
\end{proof}
\vskip .2in

Combining the arguments above together with Lemmas
\ref{Lem:curv-estimate} and \ref{diameter}, 
we conclude

\begin{coro} \label{compactb}
Under the assumptions of Theorem \ref{maintheorem}, the family of the adapted scalar flat metrics $\{g_i\}$ is
$C^{3, \beta}$ compact for any $\beta \in (0,1)$.
\end{coro}

To reach the full conclusion of Theorem 
\ref{maintheorem},  i.e. under the assumption that the family of the boundary metric $\{\hat g_i\}$ is uniformly bounded in $C^{k, \alpha} (M)$ for some $k \geq 6$, the family 
of metrics $\{g_i\}$ is in fact $C^{k, \alpha'}$ compact for any $ \alpha' < \alpha$, we first quote a result of  Wang-Zhou \cite{HW}, which allows us to pass over the compactness of the sequence $g_i$ to that of the compactness of the Fefferman-Graham metrics $({g_i})_{FG}$ (which we have discussed in section \ref{section 3.2}), and which when restricted to the boundary has the same metric $\hat g_i$ as $g_i$.    Here we will quote a special case of their result. \\

\begin{theo} \cite[corollary1.1]{HW} \label{HWresult} Under the assumption (1) in Theorem \ref{maintheorem}, if the family of the adapted scalar flat metrics $\{g_i\}$ is
$C^{k, \alpha}$ compact with $k\ge 3$, the family of the FG metrics
$\{(g_i)_{FG}\}$ is
$C^{k, \beta}$ compact for any $\beta<\alpha$. The inverse is also true.

\end{theo}

\vskip .2in

We will now proceed to finish the proof of 
Theorem \ref{maintheorem}.

\begin{proof}[Proof of Theorem \ref{maintheorem}]

From Corollary \ref{compactb} the sequence  $(X,g_i)$  is compact in $C^{3,\beta}$ norm for any $\beta\in (0,1)$ (up to a diffeomorphism fixing the boundary).
Apply Theorem \ref{HWresult}, $(X,(g_i)_{FG})$  is also compact in $C^{3,\beta'}$ with $\beta'<\beta$. Apply the elliptic gain property of the $g_{FG}$ metric as stated in Corollary \ref{gain of FG}, we then conclude the 
sequence $(X,(g_i)_{FG})$  is compact in $C^{k,\alpha'}$ norm with $0<\alpha'<\alpha$. Applying Theorem \ref{HWresult} back again, we conclude the sequence of adapted metrics $(X,g_i)$  is compact in $C^{k,\alpha'}$ norm for any $\alpha'\in (0,\alpha)$ up to a diffeomorphism fixing the boundary and for any  $k\geq 6$. This finishes the proof of the theorem.
\end{proof}

\begin{rema} To get $C^{3,\beta}$ norm bound for the sequence  $(X,g_i)$, we use the $\varepsilon$-regularity property Corollary \ref{epsilonregularity1}. In such case, we need $k-2\ge 4$. This is the reason for which we require $k\ge 6$.
    
\end{rema}

\newpage

\noindent {\large PART II:  Existence Result}

\vskip .2in

In this section, we will apply the compactness result in Part I to establish some existence results for conformal filling  of Poincar\'e Einstein metrics for a class of metrics with conformal infinity either $ S^3$ or $S^1 \times S^2$. Our  results are built upon the earlier existence results of Graham-Lee {\cite{GL}} and Lee \cite[Theorem A]{Lee1}.

\section{linearized operator}
\label{Section7}
 
We now begin the set up to establish  the existence results. We first recall some results in the study of the linearized operator of the Einstein equation, here we mainly follow the notations and results in Lee \cite{Lee1},  Biquard \cite{Biquard} in the subject, see also Chang-Ge-Qing \cite[section 5]{CGQ} \\

Given two asymptotically hyperbolic metrics $g^+$ and $k^+$, we consider the gauged Einstein equation
\beq
\label{gauged Einstein equationbis}
 F(g^+,k^+):=Ric[g^+]+(d-1)g^+- \delta_{g^+}^*(B_{k^+}(g^+)),
\eeq
where $B_{k^+}$ is a linear differential operator on symmetric (0,2) tensor,  which is the infinitesimal version of the harmonicity condition
$$
B_{k^+}(g^+):=\delta_{k^+} g^++\frac12 \mathfrak{d}\tr_{k^+}(g^+).
$$
Here,  $\delta$ denotes the divergence operator of $2$-tensors, $\delta^*$ the symmetrized covariant derivative of the vector field and $\mathfrak{d}$ the exterior derivative. 

We now recall the Lichnerowicz Laplacian $\triangle_L$ on symmetric $2$-tensors given by:
$$
\triangle_L:=\nabla^*\nabla +2\overset{\circ}{Ric}[k^+]-2\overset{\circ}{Rm}[k^+]; 
$$
where 
$$
\overset{\circ}{Ric}[k^+](u)_{ij}=\frac{1}{2}(R_{im}[g^+]{u_j}^m+R_{jm}[k^+]{u_i}^m), 
$$
and
$$
\overset{\circ}{Rm}[k^+](u)_{ij}=R_{imjl}[k^+]u^{ml}.
$$
We have for any asymptotically hyperbolic metrics $k^+$,  the linearized operator of $F$ satisfies 
$$
D_1 F(k^+,k^+)=\frac12(\triangle_L+2(d-1)),
$$
where $D_1$ denotes the differentiation of $F$ with respective to its first variable. 
It is clear that for any asymptotically hyperbolic Einstein metrics $g^+$,
$$
 F(g^+, g^+)=0.
$$

We recall in section \ref{SectionYamabeconstant},  denote  $Y= Y(M, [h])$ (resp. $Y_1$, $Y_2$) as  the Yamabe constant on the boundary $M$ (resp. the first and the second  Escoba-Yamabe constants on the manifold $X$ with boundary $M$).  We remark  on a CCE manifold,  the inequalities  (\ref{Ya-to-Y})  and (\ref{Yb-to-Y}) in section \ref{SectionYamabeconstant} hold, thus $Y_1$ and $Y_2$ each has a positive lower bound when $Y$ has a positive lower bound.

We now recall an earlier result due to Wang \cite{FWang} and independently by Gursky \cite{Gursky}.
\begin{lemm}
\label{localinvertible}
Let $(X, M, g^+)$ be a CCE manifold of dimension $4$. Assume
\beq
\label{weyl-pinching}
48\int_X |W|^2_+ dvol_{g^+}<  Y_1 (X,M,[g])^2.
\eeq
Then the linearized operator $ \triangle_L+6$ is invertible. Moreover, there exists a neighborhood $V$ of the conformal infinity $h$, such that for any metric $\tilde h\in V$, we can find an unique CCE filling in metric close to $g^+$ with the conformal infinity $[\tilde h]$ up to a diffeomorphism.
\end{lemm}

\begin{proof}
We observe that for any symmetric $2$-tensor $u$, we have
$$
(\triangle_L+6)u=-\triangle_+u -2u-2 W(u)+2\tr(u) g_+,
$$
where $W(u)_{ij}=W_{ikjl} u^{kl}$. In particular, for traceless symmetric $2$-tensor $u$, we obtain
$$
(\triangle_L+6)u=-\triangle_+u -2u-2 W(u).
$$
Thus by  a result of J. Lee \cite{Lee95}, under the assumption $Y(M [h])> 0$, any real function $v\in H^1(X,g^+)$ satisfies

$$
\int_X |\nabla_+ v|_+^2\ge \frac{9}{4}\int_X v^2.
$$
In particular, for any traceless symmetric  $2$-tensor $u\neq 0$, we infer
$$
\int_X |\nabla_+ u|_+^2 -2|u|^2\ge \int_X |\nabla_+ |u||_+^2 -2|u|^2\ge\frac{1}{4}\int_X |u|^2\ge 0.
$$
Here we have applied the Kato's inequality $|\nabla_+ |u||\le |\nabla_+ u|$. On the other hand, we know (see \cite[Lemma 3.6]{Hui})
$$
\langle W(u), u\rangle\le \frac{1}{\sqrt{3}}|W||u|^2.
$$
We therefore deduce
\begin{align}
\int_X \langle (\triangle_L+6)u, u \rangle &= \int_X |\nabla_+ u|^2-2|u|^2- 2\langle W(u),u\rangle\\
&\ge  \int_X |\nabla_+ |u||^2-2|u|^2- \frac{2}{\sqrt{3}}|W||u|^2\\
&\ge \frac{1}{6}(Y(X,M,[g^+])- \frac{12}{\sqrt{3}}\|W\|_2)\||u|\|_4^2>0.
\end{align}
Thus the desired result follows. The second part is a direct consequence of \cite[Theorem A]{Lee1}.
\end{proof}

Our strategy for the existence result is to find a class of metrics $h$ on $S^3$ or $S^2\times S^1$ so that condition (\ref{weyl-pinching}) is satisfied along a path connecting $h$ to the canonical  metric $h_c$ on $S^3$ or $S^2\times S^1$. To achieve this goal, we will first derive a preliminary estimate of the $L^2$ norm of the curvatures  in terms of the Yamabe constant $Y_1(X, M, [g])$  where $g$ corresponds to the adapted scalar flat metric whose restriction to the boundary is the given metric $h$.
\vskip .1in

\begin{rema}
The proof of the results in this and the next sections (sections  \ref{section 7bis} and \ref{Section7})  works for all metrics $g$ which are compactification of some Poincar\'e Einstein metrics $g^+$, with vanishing scalar curvature and positive mean curvature $H$. We denote $ \mathcal{ B} \mathcal{M} (\bar X)$ this class of metrics. In particular, the adapted scalar flat metrics introduced in part I of the paper are in $ \mathcal{ B} \mathcal{M} (\bar X)$.  Although the PDE satisfied by the adapted metric was never used in the proof of the results, for simplicity  we will just state all the results in these two sections just for the adapted flat metric.
\end{rema}

\begin{lemm}
\label{LemmaRicci}
Let $(X,M, g^+)$ be a CCE manifold of dimension $4$ and $g$ be an adapted scalar flat metric.
We assume in addition  that the restriction of $g$ on the conformal infinity has positive scalar curvature, then 
\beq
\label{GradientRicci1bis} 
\begin{array}{llll}
&\ds(Y_1 -\frac{2}{\sqrt{3}}\|W\|_2-\frac{2}{\sqrt{3}} \|E\|_2) \|E\|_4^2+\frac{48}{27} \oint H|\hat{\nabla}H|^2\\
\le&\ds 
-4\oint \langle \hat{E}, S\rangle -\frac{4}{3}\oint H\hat{\triangle}\hat{R}.
\end{array}
\eeq
and
\beq
\label{GradientRicci1bisbis} 
\begin{array}{llll}
&\ds(\frac{Y_1}{2} -\frac{2}{\sqrt{3}}\|W\|_2-\frac{2}{\sqrt{3}} \|E\|_2) \|E\|_4^2\\
\le&\ds 
-4\oint \langle \hat{E}, S\rangle +\frac{1}{Y_2}\|\hat{\nabla}\hat{R}\|_{L^{3/2}(M)}^2.
\end{array}
\eeq
\end{lemm}

\begin{proof}

Recall from the definition of the Yamabe constants $Y_1 = Y_1(X, M , [g])$, we have the Sobolev inequality on manifolds $X$ of dimension four:
\beq
\label{Y1}
  Y_1 ||E||_4^2 \leq \int_X 6 |\nabla E|^2 +\oint_M 2H  |E|^2 .
\endeq

we also have (\cite[section 4]{CGY-IHES})
\begin{align}
|W_{ikjl}E^{kl}E^{ij}|\le \frac{1}{\sqrt{3}}|W||E|^2, \\
|tr(E^3)|\le \frac{1}{\sqrt{3}} |E|^3.
\end{align}
Thus the estimate  (\ref{GradientRicci1bis}) follows directly from these estimates applying to the inequality (\ref{Y1}), together with the following formula  (\ref{GradientRicci1}) in Lemma B.1 in the appendix B.
For $g\in \mathcal{ B} \mathcal{M} (\bar X)$, we have
$$
\begin{array}{llll}
  \ds\int |\nabla E|^2
& =   \ds 2\int W_{ikjl}E^{kl}E^{ij}  - 2\int tr(E^3) + \frac{1}{12} \int |\nabla R|^2 - \int \frac{1}{3} R |E|^2 \\ &-\ds \{ 
4\oint \langle \hat{E}, S\rangle+\frac{1}{3}H\hat{\triangle}\hat{R} + \frac{H}{3}|\hat{E}|^2+ \frac{H}{9}|\hat{R}-\frac{2H^2}{3}|^2+\frac{14H}{27}|\hat{\nabla}H|^2 \\
&\ds-\oint R(-\frac49 \hat{\triangle}H-\frac1{18}H\hat{R}-\frac1{12}HR +\frac1{27}H^3).
\end{array}
$$

To see (\ref{GradientRicci1bisbis}), we first recall the Sobolev trace inequality which we have stated in terms of the Yamaba constant $Y_2= Y_2( X, M, [g])$  for $X$ of dimension 4, and apply the inequality to the function  $|E|$ restricted to M, we get
\beq
\label{Y2}
 Y_2 ||E||_{L^3(M)} ^{\frac {2}{3}} \leq \int _X (6 |\nabla E|^2 +\oint_M 2H  |E|^2).
\endeq
Apply  Lemma \ref{order0}, we have when restricted to the boundary, 
\begin{align*}
|E|^2=4|\hat{E}|^2+\frac{\hat{R}^2}{3}-\frac{4}{9}\hat{R}H^2 +\frac{4}{27}H^4+\frac{8}{9}|\hat{\nabla}\hat{R}|^2.
\end{align*}
Together with (\ref{GradientRicci1}), we infer
$$
\begin{array}{llll}
  &\ds\int |\nabla E|^2+\oint \frac{H}{3}|E|^2\\
 =&   \ds 2\int W_{ikjl}E^{kl}E^{ij}  - 2\int tr(E^3)  - 
4\oint \langle \hat{E}, S\rangle+\frac{1}{3}H\hat{\triangle}\hat{R} \\
&\ds+ \oint H(-\frac{\hat{R}^2}{3}+\frac{4}{81}H^2\hat{R}-\frac{4H^4}{27})-\oint \frac{48H}{27}|\hat{\nabla}H|^2  \\
\le &\ds 2\int W_{ikjl}E^{kl}E^{ij}  - 2\int tr(E^3)  - 
4\oint \langle \hat{E}, S\rangle+\frac{1}{3}H\hat{\triangle}\hat{R} -\oint \frac{48H}{27}|\hat{\nabla}H|^2 
\end{array}
$$
Again from  Lemma \ref{order0}, we have on the boundary $M$
$$
|E|\ge \frac{2\sqrt{2}}{3}|\hat{\nabla}\hat{R}|
$$
which in term implies
\begin{align*}
|\frac{4}{3}\oint H\hat{\triangle}\hat{R}|=|\frac{4}{3}\oint \langle\hat{\nabla}H,\hat{\nabla}\hat{R}\rangle|\le| \sqrt{2} \oint|E||\hat{\nabla}\hat{R}||\\
\le \sqrt{2}\|E\|_{L^3(M)}\|\hat{\nabla}\hat{R}\|_{L^{3/2}(M)}\le \frac{\sqrt{2}}{\sqrt{Y_2}}\sqrt{ \|\nabla E\|_{L^2(X)}^2+\|\sqrt{\frac{H}{3} }E\|_{L^2(M)}^2 }\|\hat{\nabla}\hat{R}\|_{L^{3/2}(M)}\\
\le \frac{1}{2}(\|\nabla E\|_{L^2(X)}^2+\|\sqrt{\frac{H}{3}} E\|_{L^2(M)}^2)+\frac{1}{Y_2}\|\hat{\nabla}\hat{R}\|_{L^{3/2}(M)}^2.
\end{align*}
From this we establish (\ref{GradientRicci1bisbis}) from the estimate (\ref{GradientRicci1bis}).
\end{proof} 
We now begin the estimate of Weyl curvature
\begin{lemm}
\label{LemmaWeyl}
Let $(X,M, g^+)$ be a CCE manifold of dimension $4$ and $g$ be the adapted metric. Then
if we assume further that the restriction of $g$ on the conformal infinity has positive scalar curvature. Then
\beq
\label{gradientweylestimate}
(Y_1-\sqrt{3}\|W\|_2) \|W\|_4^2\le \frac{2}{\sqrt{3}} \|W\|_2 \|E\|_4^2  -8\oint \langle \hat{E}, S\rangle.
\eeq
\end{lemm}

\begin{proof}

We recall the formula (\ref{gradientweyl}) in Lemma B.2, appendix B,

$$
\int |\nabla W|^2 = + \int 3{W_{ms}}^{ij} {W_{ij}}^{kl} {W^{ms}}_{kl} - \frac{1}{2} R |W|^2 + 2W_{ikjl}  E^{ij} E^{kl}  -8\oint \langle \hat{E}, S\rangle.
$$

Thus, together with the estimates
\begin{align}
|{W_{ms}}^{ij} {W_{ij}}^{kl} {W^{ms}}_{kl}|\le\frac{1}{\sqrt{3}} |W|^3,\\
|W_{ikjl}  E^{ij} E^{kl}|\le\frac{1}{\sqrt{3}} |W||E|^2, 
\end{align}
 applied to the Sobolev inequality (\ref{Y1}) with respect to the Weyl curvature, to obtain (\ref{gradientweylestimate}).
\end{proof}

As a  direct consequence  of Lemmas \ref{LemmaRicci} and \ref{LemmaWeyl}, we obtain  the following result
\begin{coro}
\label{Coro9.3}
Let $(X, M, g^+)$ be a CCE manifold of dimension $4$ and $g$ be an adapted metric such that the restriction of $g$ on the conformal infinity has positive scalar curvature. Then
\beq
\label{Gradientcurv}
\begin{array}{lllll}
&\ds(\frac{Y_1}{2} -{\sqrt{3}}\|W\|_2-\frac{2}{\sqrt{3}} \|E\|_2) (\|E\|_4^2+ \|W\|_4^2)\\
\le&\ds 
-12\oint \langle \hat{E}, S\rangle +\frac{1}{Y_2}\|\hat{\nabla}\hat{R}\|_{L^{3/2}(M)}^2.
\end{array}
\eeq
\end{coro}

\section{Estimate of the S tensor}
\label{section8}

We now get to the key point of the main existence result, that is, under suitable assumptions, we can estimate the size of the non-local term $S$ tensor in terms of information of the boundary metric $h$, 

To handle the $S$-tensor, we first recall a result on $d=n+1=4$ dimensional Einstein manifold, which can be found for example in \cite[Corollary 4.2]{catino16}
\begin{lemm}
Let $(X, M, g^+)$ be a CCE manifold of dimension $4$. Then for $W = W(g^+)$, we have
\begin{align}
\frac{1}{2}\triangle_+|W|_+^2 = |\nabla W|_+^2 -6|W|_+^2 - 3W_{ijuv} W_{uvkl} W_{ijkl}.
\end{align}
Which implies, in particular, we have
\begin{align}
\frac{1}{2}\triangle_+|W|_+^2 \ge  |\nabla W|_+^2 -6|W|_+^2 - 3|W|_+^3.
\end{align}
\end{lemm}

We now recall a key fact, namely the Kato inequality (\ref{kato}) below on Einstein manifolds, which has been used in the literature in  Bando-Kasue-Nagajima\cite{BKN} in the derivation of decay of ALE ends for classes of Einstein manifolds with no boundary.

The Kato inequality on the 4-dimensional Einstein manifold
\beq
\label{kato}
\begin{array}{llll}
|\nabla_+ W|_+^2\ge \frac{5}{3}|\nabla_+ |W|_+|^2
\end{array}
\eeq
This leads to the following result.
\begin{coro}
Let $(X, M, g^+)$ be a CCE manifold of dimension $4$. Then
\beq
\label{WeylonEinstein}
- \triangle_+|W|_+^{1/3} \le 2|W|_+^{1/3} + |W|_+^{4/3}.
\endeq
\end{coro}
That is, $L_+ |W|_+^{1/3} \le |W|_+^{4/3}$, where $L_+=-\triangle_+ -2$ is the conformal Laplace operator.  \\

Our main observation is that, for metric conformal to an Einstein metric, we can apply the conformal invariance of conformal Laplace operator, and obtain an analogue result of (\ref{WeylonEinstein}) for the function $U$.

\begin{lemm}
\label{Lem8.3}
Let $(X,M, g^+)$ be a CCE manifold of dimension $4$. Then for the metric $g=\rho^2g_+$, we have
\beq
\label{U-pde}
 -\triangle_g U + \frac{1}{6}  R_g U \le |W|_g U.
\endeq
where $\ds U=\left(\frac{|W|_g}{\rho}\right)^{1/3}$. Furthermore, we have
\begin{equation}
\label{eq8.1}
\begin{array}{llll}
\ds(\frac{1}{2}Y_1-\frac{54}{5}\|W\|_2)\|U^3\|_4^2\le \oint\frac{192}{5}S^{\beta\mu}\hat{\nabla}^\theta \hat{C}_{\beta\mu\theta}
\end{array}
\end{equation}
and
\begin{equation}
\label{eq8.1bis}
\begin{array}{llll}
& \ds(\frac14 Y_1 -\frac{54}{5}\|W\|_2)\|U^3\|_4^2+ 2 Y_2(\oint |S|^3)^{2/3}
 \\
 \le&\ds\oint\frac{192}{5}S^{\beta\mu}\hat{\nabla}^\theta \hat{C}_{\beta\mu\theta};
\end{array}
\end{equation}
here $\hat{C}$ denotes the Cotton tensor on the boundary.\\
In particular, when $\|W\|_2\le \frac{5}{216}Y_1$, we have
\begin{equation}
\label{eq8.1bisbis}
  Y_2 \|S\|_3\le 20\|\hat{\nabla}\hat{C}\|_{\frac{3}{2}}.
\end{equation}
\end{lemm}

Before proving Lemma 8.3, we will first derive some first and second order asymptotic
expansion formulas of the Weyl and U for the metric $g$ near the conformal infinity. The derivations,
though quite routine, are quite tedious and lengthy mainly due to the fact that our choice of the scalar flat adapted metric $g$ in general does not have totally geodesic boundary.  We have put the computations in the
appendix A ; namely (A.16), (A.18) and (A.19) in Corollary A.4. Here we will introduce
the notations and summarize the relevant formulas.

\begin{lemm}
\label{asy W and U}
For the adapted metric $g$,  we have the expansion of Weyl tensor near the boundary
$$
 |W|^2=e_2 r^2+e_3 r^3 +o(r^3),
$$
where $r$ is the distance function to the boundary and
\begin{align}
e_2=\langle \nabla_0  W_{ijkl},\nabla_0 W_{ijkl}\rangle =  8|S|^2+4|\hat C|^2,\\
e_3=\frac{2}{3}\langle \nabla_0\nabla_0  W_{ijkl},\nabla_0 W_{ijkl}\rangle.
\end{align}
\begin{align*}
\oint e_3&\ds=\frac{2}{3} \langle \nabla_0 \nabla_0 W_{ijkl},\nabla_0 W_{ijkl}\rangle \\
&\ds=\frac{2}{3}\oint \frac{40H}{3}|S|^2-32S^{\beta\mu}\hat{\nabla}^\theta \hat{C}_{\beta\mu\theta}+\frac{8H}{3} \hat{C}_{ \alpha\beta\mu }\hat{C}^{\mu\beta\alpha}.\\
\end{align*}
\end{lemm}

We now return to the proof of Lemma \ref{Lem8.3}.

\begin{proof}
Taking $U^5$ as a test function, we get
\beq
\label{estU}
\frac16\oint \frac{\p U^6}{\p r}+ \frac59\int |\nabla U^3|^2=-\int U^5\triangle U\le \int |W|U^6\le \|W\|_2\|U^6\|_2.
\eeq
Recall for $r(x)= dist_g (x, M)$ and $g= \rho^2 g^+$ we have
 the expansion that $\rho=r(1-\frac{H}{6}r)+o(r^2)$. Thus, it follows from Lemma \ref{asy W and U} that
$$
U^6=\frac{|W|^2}{\rho^2}=(e_2+e_3r)(1+\frac{H}{3}r)+o(r),
$$
which implies on the boundary
\beq
\label{gradU}
\frac{\p U^6}{\p r}= \frac{H}{3} e_2 +e_3,
\eeq
and
\beq
\label{eq9.8new}
U^6|_M =e_2.
\eeq
Next we observe that since 
$$
\hat{A}_{\alpha\beta,\mu}\hat{A}^{\mu\beta,\alpha}= \hat{A}_{\alpha\mu, \beta}\hat{A}^{\mu\beta,\alpha}=\hat{A}_{\alpha\beta,\mu}\hat{A}^{\mu \alpha, \beta},
$$
we have
\begin{align*}
\hat{C}_{ \alpha\beta\mu }\hat{C}^{\mu\beta\alpha}&=(\hat{A}_{\alpha\beta,\mu}- \hat{A}_{\alpha\mu, \beta})(\hat{A}^{\mu\beta,\alpha}- \hat{A}^{\mu \alpha, \beta})\\
&=\hat{A}_{\alpha\mu, \beta} \hat{A}^{\mu \alpha, \beta}- \hat{A}_{\alpha\beta,\mu} \hat{A}^{\mu \alpha, \beta}=|\hat{\nabla}\hat{A}|^2- \hat{A}_{\alpha\beta,\mu} \hat{A}^{\mu \alpha, \beta}\\
&=\frac{1}{2}\hat{C}_{ \alpha\beta\mu }\hat{C}^{\alpha\beta\mu}\ge 0.
\end{align*}

Thus from the formula of $e_3$ in Lemma \ref{asy W and U}, we get
\begin{align*}
\oint e_3&\ds=\frac{2}{3} \langle \nabla_0 \nabla_0 W_{ijkl},\nabla_0 W_{ijkl}\rangle \\
&\ds=\frac{2}{3}\oint \frac{40H}{3}|S|^2-32S^{\beta\mu}\hat{\nabla}^\theta \hat{C}_{\beta\mu\theta}+\frac{8H}{3} \hat{C}_{ \alpha\beta\mu }\hat{C}^{\mu\beta\alpha}\\
&=\frac{2}{3}\oint \frac{40H}{3}|S|^2-32S^{\beta\mu}\hat{\nabla}^\theta \hat{C}_{\beta\mu\theta}+\frac{4H}{3} \hat{C}_{ \alpha\beta\mu }\hat{C}^{\alpha\beta\mu}\\
&\ge \frac{2}{3}\oint H\frac{e_2}{3} -32S^{\beta\mu}\hat{\nabla}^\theta \hat{C}_{\beta\mu\theta}.
\end{align*}
Thus it follows from (\ref{gradU}) that 
\begin{align*}
\oint \frac{\p U^6}{\p r}&\ge  \oint \frac{5He_2}{9} -\frac{64}{3}S^{\beta\mu}\hat{\nabla}^\theta \hat{C}_{\beta\mu\theta}.
\end{align*}
Combining this with (\ref{estU}) and (\ref{eq9.8new}), we get
\beq
\label{estimateUandS}
\oint H U^6-\oint\frac{192}{5}S^{\beta\mu}\hat{\nabla}^\theta \hat{C}_{\beta\mu\theta}+ 6\int |\nabla U^3|^2\le \frac{54}{5}\|W\|_2\|U^6\|_2.
\endeq
Coupling (\ref{estimateUandS}) with the Sobolev inequality,
\beq
 Y_1 \|U^3\|_4^2 \leq 6\int_X |\nabla U^3|^2+2\oint_M HU^6.
\eeq
we obtain (\ref{eq8.1}).\\
Similarly, by applying the Sobolev trace inequality together with the formula of $e_2=8|S|^2+4|\hat{C}|^2$, we get
\beq
8 Y_2 \|S\|^2_3 \le Y_2 (\oint U^9)^{2/3} \le 6\int_X |\nabla U^3|^2 +2\oint_M HU^6.
\eeq
From this we obtain the estimate (\ref{eq8.1bis}) of $S$.
It also follows from (\ref{eq8.1bis}) that 
\begin{align}
 2 Y_2 \|S\|_3^2\le \frac{192}{5}\oint S^{\beta\mu}\hat{\nabla}^\theta \hat{C}_{\beta\mu\theta}
\le 40  \|S\|_3 \|\hat{\nabla}\hat{C}\|_{\frac{3}{2}},
\end{align}
provided $\|W\|_2\le \frac{5}{216}Y_1$,
which implies (\ref{eq8.1bisbis}). We have thus finished the proof of the Lemma.
\end{proof}

\begin{lemm} 
\label{newlemma1}
Let  $(X,M, g^+)$ be a $C^4$ CCE and $g$ the adapted compactified metric with the free scalar curvature.  We assume the conformal infinity has the positive Yamabe invariant and 
$ \frac{216}{5}\|W\|_2< Y_1. $ Then we have 
\beq
\label{newkeyestimate1}
\|U^3\|_4^2\le \frac{3200}{Y_1}\|\hat{\nabla}\hat{C}\|_{\frac{3}{2}}^2.
\eeq
\end{lemm}
\begin{proof}
In view of (\ref{eq8.1}), we have
$$
(\frac{1}{2}Y_1(X,M,[g])-\frac{54}{5}\|W\|_2)\|U^3\|_4^2\le  \frac{9}{5}\oint \frac{64}{3}S^{\beta\mu}\hat{\nabla}^\theta \hat{C}_{\beta\mu\theta}.
$$
Together with (\ref{eq8.1bisbis}), we infer
$$
(\frac{1}{2}Y_1(X,M,[g])-\frac{54}{5}\|W\|_2)\|U^3\|_4^2\le 40\oint S^{\beta\mu}\hat{\nabla}^\theta \hat{C}_{\beta\mu\theta}\le 800 \|\hat{\nabla}\hat{C}\|_{\frac{3}{2}}^2.
$$
It is clear that
$$
\frac{1}{4}Y_1(X,M,[g])\|U^3\|_4^2\le (\frac{1}{2}Y_1(X,M,[g])-\frac{54}{5}\|W\|_2)\|U^3\|_4^2.
$$
Finally, the desired inequality (\ref{newkeyestimate1}) follows.
\end{proof}

In the following lemma, we will impose some sufficient conditions (
(\ref{condition3}) and (\ref{condition4}) on Lemma \ref{newlemma2}) on the boundary metric, the conditions will help us later to reach the desired estimate bounded the $L^2$ norm of the Weyl curvature by the Yamabe constant $Y_1$.  

\begin{lemm} 
\label{newlemma2}
$(X,M, g^+)$ be a $C^4$ CCE and $g$ the adapted compactified metric with the free scalar curvature.  We assume the conformal infinity has a positive Yamabe invariant. 
Under the assumptions
\beq
\label{condition3}
\frac{c_0} {Y_1^4} vol(h)^{4/3} \|\hat{\nabla}\hat{C}\|_{\frac{3}{2}}^2\le 3,
\eeq
and
\beq
\label{condition4}
\frac {c_0 vol(h)^{4/3} \|\hat{\nabla}\hat{C}\|_{\frac{3}{2}}^2}{Y_1^4}\left[\left( \frac{2\sqrt{3}}{9} \|\hat{R}\|_{3/2}^{3/2}+ 8\pi^2 |\chi (X)|\right)^{2/3}+\frac13\right]\le \frac{25Y_1^2}{(432)^2},
\eeq
for the constant $c_0 : = 460800$, we have 
\beq
\label{conclusion2}
\|W\|_2^2 \le  \frac{c_0} {Y_1^4} vol(h)^{4/3} \|\hat{\nabla}\hat{C}\|_{\frac{3}{2}}^2\left[\left( \frac{2\sqrt{3}}{9} \|\hat{R}\|_{3/2}^{3/2}+ 8\pi^2 |\chi (X)|\right)^{2/3}+\frac13\right]\le \frac{25Y_1^2}{(432)^2}.
\eeq
\end{lemm}
\begin{proof}
Apply the Sobolev inequality, we have
$$
\|\rho^2\|_2\le\frac{6}{Y_1} \|\nabla \rho\|^2_ 2,
$$
combining with the Lemma \ref{newlemma1}, we get
$$
\|W\|_2^2\le \|U^6\|_2\|\rho^2\|_2\le \frac{19200}{Y_1^2}\|\hat{\nabla}\hat{C}\|_{\frac{3}{2}}^2 \|\nabla \rho\|^2_ 2.
$$
On the other hand, apply (\ref{freescalarN1}) we have
\begin{align*}
\|\nabla \rho\|^2_ 2=\frac{2}{3}vol(g).
\end{align*}
The following part is moved into the proof of the previous lemma:\\
Note that from the definition of $Y_2$, we always have
$$
\|H\|_1\ge \frac{1}{2} Y_2 vol(h)^{2/3}.
$$
It also follows from (\ref{GBC}) that
\begin{align*}
\frac{2}{27}\|H\|^3_3-\frac{1}{3}\|H\|_3 \|\hat{R}\|_{3/2}&\le 
- 8\pi^2 \chi (X)+\frac14\int_X |W|^2 -\int_X \frac12|Ric|^2\\
&\le 8\pi^2 |\chi (X)|+\frac14\|W\|^2_2,
\end{align*}
which implies
\begin{align*}
\frac{1}{27}\frac{\|H\|^3_1}{vol(h)^2} \le\frac{1}{27}\|H\|^3_3&\le \frac{2\sqrt{3}}{9} \|\hat{R}\|_{3/2}^{3/2}+ 8\pi^2 |\chi (X)|+\frac14\|W\|^2_2 \\
\end{align*}
Therefore, we infer 
\begin{align}\label{equationnewsection9}
vol(g)^{1/2}\le \frac{2\|H\|_1}{Y_1}\le \frac{6 vol(h)^{2/3}}{Y_1} \left( \frac{2\sqrt{3}}{9} \|\hat{R}\|_{3/2}^{3/2}+ 8\pi^2 |\chi (X)|+\frac14\|W\|^2_2\right)^{1/3}.
\end{align}
Thus, it follows from (\ref{equationnewsection9}) that 
\begin{align*}
&\|W\|_2^2\le \frac{12800 }{Y_1^2}\|\hat{\nabla}\hat{C}\|_{\frac{3}{2}}^2 vol(g)\\
\le&\ \frac{c_0} {Y_1^4} vol(h)^{4/3} \|\hat{\nabla}\hat{C}\|_{\frac{3}{2}}^2\left( \frac{2\sqrt{3}}{9} \|\hat{R}\|_{3/2}^{3/2}+ 8\pi^2 |\chi (X)|+\frac{\|W\|_2^2}{4}\right)^{2/3}\\
\le&\frac {c_0} {Y_1^4} vol(h)^{4/3} \|\hat{\nabla}\hat{C}\|_{\frac{3}{2}}^2\left[\left( \frac{2\sqrt{3}}{9} \|\hat{R}\|_{3/2}^{3/2}+ 8\pi^2 |\chi (X)|\right)^{2/3}+\left( \frac{\|W\|_2^2}{4}\right)^{2/3}\right].
\end{align*}
As it is clear
$$
\left(\frac{\|W\|_2^2}{4}\right)^{2/3}\le \frac{\|W\|_2^2}{6}+\frac13.
$$
Hence the desired inequality  ({\ref{conclusion2}}) follows.

\end{proof}


\section{Existence results}
\label{section9}

In this section, we will establish that for a large class of metrics defined on the boundary $\S^3$ or $\S^1 \times \S^2$, the sufficient conditions (\ref{condition3}) and (\ref{condition4}) in the previous section are automatically satisfied. From there we will establish the existence and uniqueness results in Theorem \ref{maintheoremexistencebis1} and Theorem \ref{amaintheoremexistencebis1}.

\begin{theo} 
\label{maintheoremexistencebis1}
On either $(X=\mathbb{B}^4,M=\p X=\S^3)$, $h_c$ the canonical metric on $\S^3$; or
$(X = \S^1 \times \mathbb{B}^3, M=\p X = \S^1 \times \S^2) $, $h_c$  the product metric on $\S^1 \times \S^2$,  Denote $h$ denote a $C^6$ metric on $\p X$ with positive scalar curvature. Then there exists some dimensional positive constant $\bar C>0$, such that if 
$$
\|h-h_c\|_{C^4}\le \bar C,
$$
one can find a CCE filling-in metric with the conformal infinity $[h]$ satisfying \beq
\label{eq10.1n}
\frac{108}{5}\|W\|_2\le \frac{Y_1}{3}.
\eeq
Moreover, the metric satisfying such bound is unique.
\end{theo}

\begin{rema}
$\bar C$ is a small but definite size constant. In the proof we present here, 
$\bar C \gtrapprox 10^{-6}.$
\end{rema}

\begin{rema}
In the work of Graham-Lee \cite{GL}  existence of fill in metric was established in a $C^{3, \alpha} $ neighborhood of $h_c$ of size $\epsilon$. It is not clear at the moment how is the size $\epsilon$ compared to the constant $\bar C$; but as we have mentioned in the introduction of this paper, the two existence results are established via different methods.
\end{rema}

\begin{proof}
We write
$$
\theta:=h-h_c,
$$
and we assume that 
$$
\|\theta\|_{C^4}\le C_5\le \frac{1}{4}.
$$
Here the norm is respect to the metric $h_c$. We know
$$
h^{-1}-h_c^{-1}=h_c^{-1}((Id- \theta h_c^{-1})^{-1}-Id),
$$
so that
$$
\|h^{-1}-h_c^{-1}\|_{C^0}\le \frac{4}{3}\|\theta\|_{C^0},
$$
that is 
$$
| h^{ij}-h_c^{ij}|\le \frac{4}{3}\|\theta\|_{C^0}.
$$
Next we compute We the christoffel's symbols
$$
\hat{\Gamma}_{ij}^k=\frac{1}{2}h^{kl}(h_{lj,i}+h_{li,j}-h_{ij,l}).
$$
Thus
$$
|\hat{\Gamma}_{ij}^k- \hat{(\Gamma_c)}_{ij}^k|\le C_6\|\theta\|_{c^1}(1+ \|h\|_{C^1}+\|h_c\|_{C^1})),
$$
where $C_6$ is some explicit constant.\\
Similarly
$$
{\hat{Rm}^i}_{klm}=\p_l \hat{\Gamma}_{km}^i-\p_m \hat{\Gamma}_{kl}^i
+\hat{\Gamma}_{lp}^i\hat{\Gamma}_{mk}^p-\hat{\Gamma}_{mp}^i\hat{\Gamma}_{lk}^p,
$$
so that
$$
|{\hat{Rm}^i}_{klm}-{\hat{(Rm_c)}^i}_{klm}|\le C_7 \|\theta\|_{c^2}(1+\|h\|_{C^2}+\|h_c\|_{C^2})^2
$$
where $C_7$ is some explicit constant.\\
 Similarly
 $$
\|\hat{\nabla}\hat{\nabla}{\hat{Rm}}-\hat{\nabla_c}\hat{\nabla_c}{\hat{(Rm_c)}}\|_{C^0}\le C_8 \|\theta\|_{c^4}(1+\|h\|_{C^4}+\|h_c\|_{C^4})^4
$$
where $C_8$ is some  explicit constant\\
As a consequence, we have
$$
\|\hat{\nabla}\hat{\nabla}{\hat{Rm}}-\hat{\nabla_c}\hat{\nabla_c}{\hat{(Rm_c)}}\|_{C^0}+ \|\hat{\nabla}{\hat{Rm}}-\hat{\nabla_c}{\hat{(Rm_c)}}\|_{C^0}+ 
\|{\hat{Rm}}-{\hat{(Rm_c)}}\|_{C^0}\le C_8' \|\theta\|_{c^4},
$$
where $C_8'$ is some other explicit constant depending only on $C_8$.\\
On the other hand, with the suitable $a<1/4$, if $\|h-h_c\|_{C^2}\le a$,  we have
$$
\frac{1}{2}h^{-1}\le h^{-1}_c \le 2h^{-1},
$$
$$
\frac{1}{2}h\le h_c \le 2h,
$$
$$
\frac{1}{2}\det{h}\le \det{h_c}\le 2\det{h},
$$
$$
\frac{1}{2}\hat{R}\le \hat{R_c}\le 2\hat{R}.
$$
Hence, for any smooth function $\varphi$, the Yamabe functional
$$
Y(\varphi, h)\ge \frac{1}{16}Y(\varphi, h_c),
$$
so that
\beq
\label{Yamabeinequality}
 Y([h_c],\S^3)\ge Y([h],M)\ge \frac{1}{16} Y([h_c],M).
\eeq
We set
$$
\bar C_0:=\sup_{\|\theta\|_{C^4}\le C_5} \frac{c_0}{Y_1^6} vol(h)^{4/3}\frac{\|\hat{\nabla}\hat{C}\|_{\frac{3}{2}}^2}{\|\theta\|_{C^4}^2}\left[\left( \frac{2\sqrt{3}}{9} \|\hat{R}\|_{3/2}^{3/2}+ 8\pi^2 |\chi (X)|\right)^{2/3}+3\right]
$$
Here $\bar C_0$ is some explicit constant.\\   
If
\beq
\label{condition1bisbis}
\|\theta\|_{C^4}\le \bar C:= \frac{5}{432\sqrt{\bar C_0}}
\eeq
it is clear 
$$
\frac{c_0}{Y_1^6} vol(h)^{4/3}\|\hat{\nabla}\hat{C}\|_{\frac{3}{2}}^2\left[\left( \frac{2\sqrt{3}}{9} \|\hat{R}\|_{3/2}^{3/2}+ 8\pi^2 |\chi (X)|\right)^{2/3}+3\right]\le \bar C_0 \|\theta\|_{C^4}^2\le \frac{25}{(432)^2}
$$
which implies (\ref{condition4}) in Lemma \ref{newlemma2} holds. Similarly, (\ref{condition3}) in Lemma \ref{newlemma2} holds. To see this, it is known from Lemma \ref{Yamabe0} that
$$
Y_1\le 8\pi\sqrt{3}
$$
Thus the desired claim follows. As a consequence of Lemma \ref{newlemma2}, 
then there holds
$$
\frac{432}{5}\|W\|_2\le Y_1
$$
Given $h$ satisfying  (\ref{condition1bisbis}), we consider $\forall t\in [0,1]$
$$
h_t=h_c+t(h-h_c)
$$
Thus $h_t$ satisfy also (\ref{condition1bisbis}) for all $t\in [0,1]$. We will now run the continuous method, to establish  the existence result. \\
To this end, we set
$$
I=\{t\in [0, 1] \,\, | \,\, \exists \mbox{ CCE filling in with the condition }(\ref{eq10.1n}) \mbox{ on }[0,t]\}
$$
{\sl Step 1}. For some large $T_1\in (0,1)$, we have $[0, T_1]\subset I$ since we could use the local existence result of Lee-Graham for the metrics close to the standard one on $\S^3$ \cite{GL} or on $\S^1 \times \S^2 $ \cite[Theorem A]{Lee1}, that is, there exists some $t_0$ such that for any $t\le t_0$, we could find a CCE filling in metric $g^+_t$. Moreover, $\lim_{t\to 0} \|W(t)\|_2=0$.  Hence (\ref{eq10.1n}) is satisfied on $[0, t_1]$ with some $t_1\le t_0$.\\

{\sl Step 2}. By Lemma \ref{localinvertible} and Lemma \ref{newlemma2}, $I$ is open.\\
For this purpose, for any $t\in I$, we have 
$$
\|W(t)\|_2\le \frac{5}{216}Y_1<\frac{1}{\sqrt{48}}Y_1
$$
Applying Lemma \ref{localinvertible}, the Lichnerowicz Laplacian $\triangle_L$ is invertible. Applying Theorem \cite[Theorem A]{Lee1}, we can find some CCE filling in metric for $h$ in a neighborhood of $h(t)$. Moreover, applying Lemma \ref{newlemma2}, we have
$$
\frac{108}{5}\|W(t)\|_2< \frac{Y_1(t)}{4}
$$
Thus there exists some $\eta>0$ such that $(t-\eta,t+\eta)\cap [0, 1]\subset I$. That is, $I$ is open.\\

{\sl Step 3}. By the compactness result Theorem \ref{maintheorem}, $I$ is closed. \\
To see so, let $t_n\in I$ such that $t_n\to s$. 
Applying Lemma \ref{newlemma2}, we have
$$
\frac{108}{5}\|W(t_n)\|_2< \frac{Y_1(t_n)}{4}
$$
On the other hand, by our compactness theorem, the sequence of adapted metrics $\{g(t_n)\}$ is a compact family. Hence, modulo a subsequence, $\{g(t_n)\}$ converges in Gromov-Hausdorff's sense to some scalar free metric $g(t_{\infty})$ with the boundary metric $h(t_{\infty})$. Moreover, $g(s)$ is conformal to some CCE metric. It is clear that
$$
\frac{108}{5}\|W(g(t_{\infty}))\|_2\le \frac{Y_1(g(t_{\infty}))}{4}
$$
Therefore $t_{\infty}\in I$ and $I$ is close. 
\\
Now we conclude $I=[0,1]$ and establish the existence result.\\

{\sl Step 4}: The uniqueness part of the theorem will be a consequence of the compactness result (Theorem \ref{maintheorem}) together with the local uniqueness result Lemma \ref{localinvertible} and the global uniqueness when $t$ is sufficiently small \cite[Theorem 1.9]{CGQ}.\\
To see so, first we observe we know from the proof in Step 2 above that for any CCE filling in metric $g^+(t)$ with the conformal infinity $[h(t)]$ satisfying (\ref{eq10.1n}), we can find some CCE filling in metric $g^+(s)$ with the conformal infinity $[h(s)]$ satisfying (\ref{eq10.1n}) for all $s\in (t-\eta,t+\eta)$ for some $\eta>0$. On the other hand, it follows from \cite[Theorem 1.9]{CGQ} that the CCE filling in 
is unique on $[0,t_0]$ with some small positive number $t_0\in (0,1)$. \\
Now we set
$$
I_1=\{t\in [0,1] | \,\, \exists  \mbox{ \,\,  unique CCE filling in with the condition }(\ref{eq10.1n})\}.
$$

We first claim that $I\setminus I_1$ is an open set. To see so, for each $t \in T\setminus I_1$,  there exists two different CCE fillings $g^+_1(t)$ and $g^+_2(t)$, both satisfy the condition (\ref{eq10.1n}). Applying Lemma \ref{localinvertible}. there exists some sufficiently small $\eta>0$ such that for all $s\in (t-\eta, t+\eta)\cap [0,1]$, we have two different CCE fillings in $g^+_1(s)$ and $ g^+_2(s)$. We thus conclude $I \setminus I_1$ is open.


Now suppose the set $I\setminus I_1$ is not empty. Denote $t_N=\sup_{t\in I\setminus I_1}$, note that $t_N \geq t_0 >0$. Since $I\setminus I_1$ is open, we have  $t_N \in I_1$. There also exists an increasing sequence $t_n$ which converges to $t_N$ such that for each large $n$, there exist two fillings in metrics $g^+_{t_n,1}$ and $g^+_{t_n,2}$, both exist in some neighborhood of $g^+_{t_N}$. Since $g^+_{t_N}$ 
satisfies (\ref{eq10.1n}), this contradicts the local uniqueness result Lemma  \ref{localinvertible}.  From this we finished the proof of step 4 and the theorem.
\end{proof}


We now prove two auxiliary lemmas, which will later be used in the proof of Theorem \ref{amaintheoremexistencebis1}. The main goal of the lemmas is, in addition to the estimate of $L^2$
norm of the Weyl tensor, we can also estimate the $L^2$ norm of the traceless Ricci $E$. Since in the model case $(S^3, h_c)$, with the unique CCE extension the hyperbolic metric $g^{H}$ is Einstein with $E =0$, which allows us in applying the continuity method later to prove Theorem \ref{amaintheoremexistencebis1}, to start at $h_c$ and control (see the estimate (\ref{conclusion1}))  $||E||_2$ along the whole path.  

\begin{lemm} 
Let  $(X,M, g^+)$ be a $C^4$ CCE and $g$ the adapted compactified metric with the free scalar curvature.  We assume the conformal infinity has the positive Yamabe invariant and 
\beq
\label{condition1bis}
Y_1>\max\{{2\sqrt{3}}\|W\|_2+\frac{4}{\sqrt{3}} \|E\|_2, \frac{216}{5}\|W\|_2\}
\eeq
Then, we have
\beq
\label{Gradientcurv1}
\begin{array}{llll}
&\ds(\frac{Y_1}{2}-{\sqrt{3}}\|W\|_2-\frac{2}{\sqrt{3}} \|E\|_2) (\|E\|_4^2+ \|W\|_4^2)\\
\le&\ds 
\frac{240}{Y_2}\|\hat{E}\|_{L^\frac{3}{2}(M)}\|\hat{\nabla }\hat{C}\|_{L^\frac{3}{2}(M)}
+ \frac{1}{Y_2} \|\hat{\nabla}\hat{R}\|_{L^\frac{3}{2}(M)}^2\\
\le&\ds 
\frac{240}{Y_2}\|\hat{E}\|_{L^\frac{3}{2}(M)}\|\hat{\nabla }\hat{C}\|_{L^\frac{3}{2}(M)}
+ \frac{1}{Y_2} \|\hat{R}-\overline{\hat{R}}\|_{L^\frac{3}{2}(M)}\|\hat{\triangle}\hat{R}\|_{L^3(M)}vol(h)^{2/3}.
\end{array}
\eeq
\end{lemm}
\begin{proof}
By Cauchy-Schwarz's inequality and H\"older's inequality, we have
\begin{align}
\|\hat{\nabla }\hat{R}\|_{L^\frac{3}{2}(M)}^2\le \|\hat{\nabla }\hat{R}\|_{L^2(M)}^2 vol(h)^{2/3}\le  \|\hat{R}-\overline{\hat{R}}\|_{L^\frac{3}{2}(M)}\|\hat{\triangle}\hat{R}\|_{L^3(M)}vol(h)^{2/3}.
\end{align}
Together with Corollary \ref{Coro9.3} and Lemma \ref{Lem8.3}, we deduce
\beq
\begin{array}{llll}
&\ds(\frac{Y_1}{2}-{\sqrt{3}}\|W\|_2-\frac{2}{\sqrt{3}} \|E\|_2) (\|E\|_4^2+ \|W\|_4^2)\\
\le&\ds 
12\|\hat{E}\|_{\frac{3}{2}}\|S\|_3 + \frac{1}{Y_2}\|\hat{\nabla}\hat{R}\|^2_{L^{3/2}}\\
\le & \ds 
\frac{240}{Y_2}\|\hat{E}\|_{\frac{3}{2}}\|\hat{\nabla} \hat{C}\|_{\frac{3}{2}} + \frac{1}{Y_2} \|\hat{R}-\overline{\hat{R}}\|_{\frac{3}{2}}\|\hat{\triangle}\hat{R}\|_{3}vol(h)^{1/3}.
\end{array}
\eeq
Thus the desired inequality (\ref{Gradientcurv1}) follows.
\end{proof}

We now establish  a dichotomy type result.

\begin{lemm} 
\label{estimate1}
$(X,M, g^+)$ be a $C^4$ CCE and $g$ the adapted compactified metric with the free scalar curvature.  We assume the conformal infinity has the positive Yamabe invariant. 
Then, if 
\beq
\label{condition1}
\begin{array}{lc}
\ds (\frac{240}{Y_2}\|\hat{E}\|_{\frac{3}{2}}\|\hat{\nabla }\hat{C}\|_{\frac{3}{2}}
+\frac{1}{Y_2}\|\hat{\nabla }\hat{R}\|_{L^\frac{3}{2}(M)}^2) \\
\ds \times\frac{6 vol(h)^{2/3}}{Y_1} \left( \frac{2\sqrt{3}}{9} \|\hat{R}\|_{3/2}^{3/2}+ 8\pi^2 |\chi (X)|+\frac{Y_1^2}{48}\right)^{1/3}\le \frac{43Y_1^3}{5\times(88)^3}
\end{array}
\eeq
and (\ref{condition1bis}) holds, then
we have either 
\beq
\label{conclusion1}
\frac{108}{5}\|W\|_2+\frac{2}{\sqrt{3}} \|E\|_2\le \frac{Y_1}{4}
\eeq
or 
\beq
\label{conclusion1bis}
\frac{108}{5}\|W\|_2+\frac{2}{\sqrt{3}} \|E\|_2\ge \frac{3Y_1}{4}
\eeq
\end{lemm}
\begin{proof}
We denote by $Y_c$ the Yamabe invariant on the boundary $M$. 
Thanks of Theorem \ref{Yamabe} and Lemma \ref{Yamabe2}, we have
$$
Y_1\ge (24)^{1/4}Y_c^{3/4}, \;\mbox{and}\; Y_2\ge (6Y_c)^{1/2}
$$
On the other hand, it is known from Lemma \ref{Yamabe0} that
$$
Y_1\le 8\pi\sqrt{3}, \;\mbox{and}\; Y_2\le 6(2\pi^2)^{1/3}
$$
Thanks to (\ref{Gradientcurv1}) and by H\"older's inequality, we infer
\begin{align*}
&\ds\frac{1}{5}(\frac{Y_1}{2}-{\sqrt{3}}\|W\|_2-\frac{2}{\sqrt{3}} \|E\|_2) ({\sqrt{3}}\|W\|_2+\frac{2}{\sqrt{3}} \|E\|_2)^2\\
\le&\ds(\frac{Y_1}{2}-{\sqrt{3}}\|W\|_2-\frac{2}{\sqrt{3}} \|E\|_2) (\|E\|_2^2+ \|W\|_2^2)\\
\le &\ds( \frac{240}{Y_2}\|\hat{E}\|_{\frac{3}{2}}\|\hat{\nabla }\hat{C}\|_{\frac{3}{2}}
+\frac{1}{Y_2}\|\hat{\nabla }\hat{R}\|_{L^\frac{3}{2}(M)}^2 )(vol(g))^{1/2}\\
\le &  \ds( \frac{240}{Y_2}\|\hat{E}\|_{\frac{3}{2}}\|\hat{\nabla }\hat{C}\|_{\frac{3}{2}}
+\frac{1}{Y_2}\|\hat{\nabla }\hat{R}\|_{L^\frac{3}{2}(M)}^2 )\\
&\times  \frac{6 vol(h)^{2/3}}{Y_1} \left( \frac{2\sqrt{3}}{9} \|\hat{R}\|_{3/2}^{3/2}+ 8\pi^2 |\chi (X)|+\frac{Y_1^2}{48}\right)^{1/3}
&\\
\end{align*}
We consider a polynome $f(t)=\frac{1}{5}(\frac{Y_1}{2}-t)t^2$ on $[0,+\infty)$. It increases on $[0,\frac{Y_1}{3}]$ and decreases on $[\frac{Y_1}{3},+\infty)$. Hence, if
$$
\begin{array}{ll}
&\ds ( \frac{240}{Y_1}\|\hat{E}\|_{\frac{3}{2}}\|\hat{\nabla }\hat{C}\|_{\frac{3}{2}}
+\frac{1}{Y_1}\|\hat{\nabla }\hat{R}\|_{L^\frac{3}{2}(M)}^2 ) \frac{6 vol(h)^{2/3}}{Y_1} \left( \frac{2\sqrt{3}}{9} \|\hat{R}\|_{3/2}^{3/2}+ 8\pi^2 |\chi (X)|+\frac{Y_1^2}{48}\right)^{1/3}\\
\le& \ds\frac{43Y_1^3}{5\times(88)^3}\le \min\{f(\frac{Y_1}{88}),f(\frac{3Y_1}4) \}
\end{array}
$$
that is, the condition (\ref{condition1}) is satisfied. Then, the desired results (\ref{conclusion1}) and  (\ref{conclusion1bis})  follow and we finish the proof.
\end{proof}

\begin{atheorem} 
Let  $(X=B^4,M=\p X=\S^3)$ and $h$ be a metric on  $\S^3$ with the positive scalar curvature. Assume that $h$ is in $C^6$ and denote $h_c$ the canonical metric on $\S^3$.
Denote ${\bar C}_1:=(\|\hat{\nabla }\hat{C}\|_{3/2}+ \|\hat{\triangle }\hat{R}\|_{3/2})( \frac{2\sqrt{3}}{9} \|\hat{R}\|_{3/2}^{3/2}+ 16\pi^2)^{1/3} $.
There exists two positive dimensional constants $a$ and $b$ such that if   
$$
\|h-h_c\|_{C^2}\le \min (a,\frac{b}{{\bar C}_1}),
$$
then we can find a CCE filling-in metric with the conformal infinity $[h]$ satisfying 
\beq
\frac{108}{5}\|W\|_2+\frac{2}{\sqrt{3}}\|E\|_2\le \frac{Y_1}{3}.
\eeq
Moreover, such solution with the above bound is unique.
\end{atheorem}

\begin{rema}
When the boundary $M$ is $\S^3$, instead of making 
the $C^2$ bound assumption on $\theta:=h-h_c$ in Theorem \ref{amaintheoremexistencebis1}, the argument in the proof could be modified under the weaker assumptions of the bounds of $h$ in $W^{4, 3/2}$ and of $\theta\in C^0\cap W^{2,3}$ to achieve the same existence result.
\end{rema}


\begin{proof}
We note the Euler characteristic $\chi(X)=1$ and $\frac{Y_1^2}{48}\le 8\pi^2$. 
We know
\begin{align*}
&\|\hat{\nabla }\hat{R}\|_{3/2}^2\le \|\hat{\nabla }\hat{R}\|_{2}^2 (vol(h))^{1/3}
\le \|\hat{R}-\overline{\hat{R}}\|_{3} \|\hat{\triangle }\hat{R}\|_{3/2}(vol(h))^{1/3}\\
\le& \|\hat{R}-\overline{\hat{R}}\|_{C^0} \|\hat{\triangle }\hat{R}\|_{3/2}(vol(h))^{2/3},
\end{align*}
so that
\begin{align*}
&(240\|\hat{E}\|_{3/2}\|\hat{\nabla }\hat{C}\|_{3/2}
+\|\hat{\nabla }\hat{R}\|_{3/2}^2)(vol(h))^{2/3} \\
\le& 
(240\|\hat{E}\|_{C^0}\|\hat{\nabla }\hat{C}\|_{3/2}
+\|\hat{R}-\overline{\hat{R}}\|_{C^0} \|\hat{\triangle }\hat{R}\|_{3/2})(vol(h))^{4/3}\\
\le & 241 \max \{\|\hat{E}\|_{C^0}, \|\hat{R}-\overline{\hat{R}}\|_{C^0}\} {\bar C}_1 (vol(h))^{4/3}\\
\le & C_9\|h-h_c\|_{C^2} {\bar C}_1 (vol(h))^{4/3},
\end{align*}
where $C_9$ is some explicit constant. Hence,  by choosing $a$ be the constant in the proof of Theorem \ref{maintheoremexistencebis1} so that (\ref{Yamabeinequality}) holds   and  $b$ be a constant such that (\ref{condition1}) in Lemma \ref{estimate1} holds,  we then follow the same lines of arguments as in the proof of Theorem \ref{maintheoremexistencebis1} to establish  Theorem \ref{amaintheoremexistencebis1}.
\end{proof}

\appendix
\section{}\label{A}
 In this section, we will derive formulas for general compact metric $g$ with the boundary, defined on $\bar X$, which is Bach flat and with umbilic boundary and the Weyl tensor vanishes on the boundary.  We denote the class of such metric $\mathcal{BM} (\bar X)$.
\subsection*{\underline{$0$-order expansion of Schouten tensor}}

\begin{mylemm} 
\label{order0}
Suppose $g$ is a metric in $\mathcal{BM} (\bar X)$. Then,   we have on the boundary $M$ 
\begin{enumerate}
\item 
$g_{\alpha \beta}=h_{\alpha\beta}, g_{0 \alpha }=0, g_{0 0 }=1$;
\item $2 A_{\alpha\beta}=  
2 \hat{A}_{\alpha\beta}-\frac19 H^2 h_{\alpha\beta},$
\item $2 A_{0\beta} 
=\frac23 \hat{\nabla}_\beta H,$
\item $2 A_{00} 
=\frac13 H^2-\frac12 \hat{R}+\frac{1}{3}R$
\end{enumerate}
\end{mylemm}
\begin{proof}
(1) is trivial. (2) is proved in \cite[Lemma 2.5]{CG}\\
Let $g^{(*)}=r^2g_+$ be the geodesic compactification. Then $g^{(*)}$ has the total geodesic boundary. Let $(x_0=r, x_1,x_2,x_3)$ be the local chart. For the metric $g^{(*)}$, we have on the boundary
\beq
\begin{array}{lll}
g^{(*)}_{\alpha \beta}=h_{\alpha\beta},\\
g^{(*)}_{0 \beta}=0,\\
g^{(*)}_{0 0}=1,\\
Ric[g^{(*)}]_{0\beta}=0, Ric[g^{(*)}]_{00}=\frac{\hat R}{4}.
\end{array} 
\eeq
We write $g=u^{-2}g^{(*)}$ with $u=1$ on the boundary and denote the covariant derivative $\bar \nabla$ w.r.t. $g^{(*)}$. For this conformal metric, we have
\beq
\label{eq3.3}
\begin{array}{lll}
Ric[g]_{ij}= Ric[g^{(*)}]_{ij}+2u^{-1}\bar\nabla_i\bar\nabla_j u+ (u^{-1}\bar\triangle u-3u^{-2}|\bar \nabla u|^2_{g^{(*)}})g_{ij}^{(*)},\\
R[g]=u^2R[g^{(*)}]+6u\bar\triangle u-12|\bar \nabla u|^2_{g^{(*)}}, \\
2A[g]_{ij}=2A[g^{(*)}]_{ij}+ 2u^{-1}\bar\nabla_i\bar\nabla_j u-u^{-2}|\bar \nabla u|^2_{g^{(*)}}g_{ij}^{(*)}. \\
\end{array} 
\eeq
Thus, we have
\beq
\label{conformalschouten}
A[g]_{ij}=A[g^*]_{ij}+ u^{-1}\bar\nabla_i\bar\nabla_j u - u^{-2}|\bar \nabla u|^2_{g^{(*)}}g_{ij}^{(*)}.
\eeq
Also due to the Gauss equation, we infer on the boundary
$$
Ric[g]_{\alpha\beta}= \hat{Ric}[g]_{\alpha\beta}+ Rm_{\alpha 0\beta 0}+L_{\alpha\gamma}{L^{\gamma}}_\beta-H L_{\alpha\beta},
$$
where $L$ is the second fundamental form. On the other hand, using the decomposition of Riemann curvature and the fact Weyl curvature $W=0$ on the boundary,
$$
Rm_{\alpha 0\beta 0}=(A\circledwedge g)_{\alpha 0\beta 0}= A[g]_{\alpha\beta}+ A[g]_{00}h_{\alpha\beta}.
$$
Thus we have
\beq
\label{tangentialricci}
Ric[g]_{\alpha\beta}= \hat{Ric}[g]_{\alpha\beta}+A[g]_{\alpha\beta}+ A[g]_{00}h_{\alpha\beta} -\frac{2H^2}{9}h_{\alpha\beta}.
\eeq
Again by the Gauss equation (\ref{eq2.4bis}),
$$
R[g]= \hat{R}+ 2Ric[g]_{00}-\frac23H^2=  \hat{R}+ 4A[g]_{00}+\frac{R}{3}-\frac23H^2,
$$
so that
$$
2A[g]_{\alpha\beta} =Ric[g]_{\alpha\beta}-\frac{R}{6}h_{\alpha\beta}= \hat{A}[g]_{\alpha\beta} + A[g]_{\alpha\beta}-\frac{H^2}{18} h_{\alpha\beta}.
$$
Hence, we obtin (3). By the conformal change, we have 
\beq
\begin{array}{lll}
2H[g]=u^2(-6\p_0 u^{-1}+2H[g^{(*)}])=6\p_0 u,
\end{array} 
\eeq
due to the fact that  the boundary is total geodesic so that $H[g^{(*)}]=0$.  On the other hand,
\beq
\label{eq3.4bis}
\begin{array}{lll}
\bar\nabla_\alpha\bar\nabla_0 u=\frac13 \hat{\nabla}_\alpha H, \bar\nabla_\alpha\bar\nabla_\beta u=0,
\end{array} 
\eeq
since $u=1$ on the boundary and the boundary is totally geodesic. Going back to (\ref{eq3.3}), we have $2 A_{0\beta} = Ric[g]_{0\beta}=\frac23 \hat{\nabla}_\beta H$. Thus, we have established (4). \\ 
It follows from (\ref{eq2.4bis}) that $2 A_{00} = Ric[g]_{00}-\frac{1}{6}R=\frac13 H^2-\frac 12 \hat{R}+\frac{1}{3}R$.  This finishes  the proof of (5).
\end{proof}

\subsection*{\underline{$1$-order expansion of Schouten tensor}}

 To continue the discussion, we need to introduce a non-local 3rd order curvature pointwise conformal invariant quantity called the $S$ tensor. We refer the reader to the definition and basic properties of the $S$ tensor in \cite[Lemma2.1]{CG}.

\begin{mylemm} \label{order1}Suppose $g$ is a metric in $\mathcal{BM} (\bar X)$. Then,  we have on the boundary $M$
\begin{enumerate}
\item 
$A_{00,\gamma}=\frac59 H\hat{\nabla}_\gamma H- \frac14 \hat{\nabla}_\gamma \hat{R}$;
\item $A_{0\beta,\gamma}=\frac13 \hat{\nabla}_\gamma \hat{\nabla}_\beta H+ \frac H3 \hat{A}_{\gamma\beta}-\frac{2H^3}{27}h_{\gamma\beta} + \frac{H\hat{R}}{12}h_{\gamma\beta} $;
\item $A_{\alpha\beta,\gamma}=  \hat{\nabla}_\gamma\hat{A}_{\alpha\beta}-  \frac H9(\hat{\nabla}_\gamma H h_{\alpha\beta} + \hat{\nabla}_\beta H h_{\alpha\gamma} +\hat{\nabla}_\alpha H h_{\gamma\beta})$;
\item $A_{00,0}=\frac{1}{6}\nabla_0 R
-(\frac13 \hat{\triangle } H+\frac{H\hat{R}}{3}-\frac {2H^3}9)$;
\item $A_{0\beta,0}=\frac{1}{6}\nabla_\beta R-\frac14 \hat{\nabla}_\beta \hat{R}+ \frac {5H}9   \hat{\nabla}_\beta H$;
\item $A_{\alpha\beta,0}= S_{\alpha\beta}+ \frac13 \hat{\nabla}_\alpha \hat{\nabla}_\beta H+  \frac H3 \hat{A}_{\alpha\beta}   -\frac{2H^3}{27} h_{\alpha\beta} + \frac{H\hat{R}}{12} h_{\alpha\beta} +\frac{1}{9}RH h_{\alpha\beta}$,
 where $S$ tensor 
 $$S_{\alpha\beta}=\nabla^i W_{i\alpha 0\beta}+\nabla^i W_{i\beta 0\alpha}-\nabla^0 W_{0\alpha 0\beta}+\frac{4}{3}H{W_{\alpha 0 \beta}}^0. $$ 
 we remark S is trace free (see \cite[Lemma2.1]{CG})
\end{enumerate}
\end{mylemm}

\begin{proof} Denote $ \Gamma$ the Christoffel symbol. We have by Lemma \ref{order0}
$$
\begin{array}{lll}
A_{00,\gamma}&=\p_\gamma A_{00}-2\Gamma_{\gamma 0}^i A_{i0}= \p_\gamma A_{00}+2{L_\gamma}^\beta A_{\beta0}\\
&=\frac{H}{3} \hat{\nabla}_\gamma H- \frac{1}{4} \hat{\nabla}_\gamma \hat{R}+\frac{2H}{9} \hat{\nabla}_\gamma H=\frac59 H\hat{\nabla}_\gamma H- \frac14 \hat{\nabla}_\gamma \hat{R}.
\end{array} 
$$
since $\Gamma_{\gamma 0}^0=0$ and $\Gamma_{\gamma 0}^\beta=- {L_\gamma}^\beta$. We prove (1) \\
Similarly, by Lemma \ref{order0}
$$
\begin{array}{lll}
A_{0\beta,\gamma}&=\p_\gamma A_{0\beta}-\Gamma_{\gamma 0}^i A_{i\beta}- \Gamma_{\gamma \beta}^jA_{0j}= \p_\gamma A_{0\beta}+{L_\gamma}^\alpha A_{\alpha\beta}-L_{\gamma\beta}A_{00}- \Gamma_{\gamma \beta}^\alpha A_{0\alpha}\\
&=\p_\gamma A_{0\beta}- \hat{\Gamma}_{\gamma \beta}^\alpha A_{0\alpha}+ \frac{H}{3}A_{\gamma\beta}- \frac{H}{3}A_{00} h_{\gamma \beta}\\
&=\frac{1}{3} \hat{\nabla}_\gamma \hat{\nabla}_\beta H+ \frac{H}{3}(\hat{A}_{\gamma\beta}-\frac{H^2}{18}h_{\gamma\beta})- \frac{H}{3} h_{\gamma \beta}(\frac{H^2}{6}-\frac{\hat{R}}{4})\\
&=\frac13 \hat{\nabla}_\gamma \hat{\nabla}_\beta H+ \frac H3 \hat{A}_{\gamma\beta}-\frac{2H^3}{27}h_{\gamma\beta} + \frac{H\hat{R}}{12}h_{\gamma\beta}.
\end{array} 
$$
Here we use $ \Gamma_{\gamma \beta}^\alpha=  \hat{\Gamma}_{\gamma \beta}^\alpha$. Thus, we have established (2).\\
We now apply  Lemma \ref{order0} and obtain
$$
\begin{array}{lll}
A_{\alpha\beta,\gamma}&=\p_\gamma A_{\alpha\beta}-\Gamma_{\gamma \alpha}^i A_{i\beta}- \Gamma_{\gamma \beta}^jA_{\alpha j}\\
&= \p_\gamma A_{\alpha\beta}-\Gamma_{\gamma \alpha}^\delta A_{\delta\beta}- \Gamma_{\gamma \beta}^\delta A_{\alpha\delta}- L_{\gamma \alpha} A_{0\beta}-L_{\gamma \beta}A_{\alpha 0}\\
&= \p_\gamma A_{\alpha\beta}-\hat{\Gamma}_{\gamma \alpha}^\delta A_{\delta\beta}- \hat{\Gamma}_{\gamma \beta}^\delta A_{\alpha\delta}- L_{\gamma \alpha} A_{0\beta}-L_{\gamma \beta}A_{\alpha 0}\\
&=\hat{ \nabla}_\gamma A_{\alpha\beta}- L_{\gamma \alpha} A_{0\beta}-L_{\gamma \beta}A_{\alpha 0}\\
&= \hat{\nabla}_\gamma\hat{A}_{\alpha\beta}-  \frac H9(\hat{\nabla}_\gamma H h_{\alpha\beta} + \hat{\nabla}_\beta H h_{\alpha\gamma} +\hat{\nabla}_\alpha H h_{\gamma\beta}).
\end{array} 
$$
Therefore (3) is deduced.\\
By the second Bianchi identity, $A_{00,0}+A_{0\gamma,\gamma}=\frac{1}{6} \p_0R=\frac{1}{6} \nabla_0R, $ so that
$$
A_{00,0}=-A_{0\gamma,\gamma}+\frac{1}{6} \nabla_0R= \frac{1}{6} \nabla_0R-(\frac13 \hat{\triangle } H+\frac{H\hat{R}}{3}-\frac {2H^3}9).
$$
Thus (4) is obtained.\\
Similarly
$$
A_{\beta 0,0}=\frac{1}{6} \nabla_\beta R-A_{\beta\gamma,\gamma}=\frac{1}{6} \nabla_\beta R-\frac14 \hat{\nabla}_\beta \hat{R}+ \frac {5H}9   \hat{\nabla}_\beta H.
$$
which etablishes (5).\\
Recall the result \cite[Lemma 2.1]{CG}. We infer
$$
\begin{array}{lll}
A_{\alpha\beta, 0}&=&S_{\alpha\beta}+A_{\alpha0, \beta}- A_{\beta0, \alpha}-\frac{\hat{\triangle} H}3 h_{\alpha \beta}+\frac13 H_{,\hat{\alpha}\hat{\beta}}+ \frac13 HA_{\alpha\beta}\\
&&-\frac{HR}{18}h_{\alpha\beta}+HA_{00} h_{\alpha \beta}+ A_{\gamma 0, \gamma} h_{\alpha\beta}\\
&=&S_{\alpha\beta} +\frac13 H_{,\hat{\alpha}\hat{\beta}}+\frac H3 A_{\alpha\beta}+HA_{00} h_{\alpha \beta}+(\frac{H\hat{R}}{3}-\frac{2H^3}{9}-\frac{HR}{18}) h_{\alpha\beta}\\
&=&S_{\alpha\beta}+ \frac13 \hat{\nabla}_\alpha \hat{\nabla}_\beta H  +\frac H3 \hat{A}_{\alpha\beta}-\frac{2H^3}{27} h_{\alpha\beta} + \frac{H\hat{R}}{12} h_{\alpha\beta}+\frac{HR}{9} h_{\alpha\beta}.
\end{array} 
$$
Therefore, we have proved (6) and finish the proof of Lemma \ref{order1}. 
\end{proof}

\subsection*{\underline{$0$-order expansion of Cotton tensor}}

The direct calculations lead to the following results for the cotton tensor. 

\begin{mycoro} \label{corocotten}Under the same assumptions as in Lemma \ref{order1},  we have on the boundary $M$
\begin{enumerate}
\item $C_{\alpha\beta\gamma}= \hat{C}_{\alpha\beta\gamma}$;
\item $-C_{\alpha 0\beta}= C_{\alpha \beta 0 }= S_{\alpha\beta} +\frac{HR}{9} h_{\alpha\beta}$;
\item $ C_{0\alpha \beta}=0$ and $ C_{00 \alpha} =-\frac{1}{6}\nabla_\alpha R$.
\end{enumerate}
\end{mycoro}

By  taking the orthonormal basis $\{e_\alpha\}$ on the boundary, another result is the following.

\begin{mycoro} \label{corocotten1}Under the same assumptions as in Lemma \ref{order1},  we have on the boundary $M$
\beq
\label{bdy}
\begin{array}{lll}
 \langle A_{ij, 0}, A_{ij}\rangle&=&  \hat{A}_{\alpha\beta}S_{\alpha\beta} +\frac13 \hat{A}_{\alpha\beta} \hat{\nabla}_\alpha \hat{\nabla}_\beta H+\frac H3|\hat{A}_{\alpha\beta}|^2-\frac{4}{27} H^3\hat{R}+ \frac{4}{81} H^5+ \frac{5}{48} H\hat{R}^2\\
&&+\frac{10}{27}H|\hat{\nabla }H|^2+ \frac{1}{12} \hat{R}\hat{\triangle }H- \frac{2}{27} H^2\hat{\triangle }H-\frac16\langle \hat{\nabla }H,\hat{\nabla }\hat{R}\rangle.
 \end{array}
\eeq
\end{mycoro}

\begin{proof}
From Lemmas \ref{order0} and \ref{order1},  and using the fact $tr(S)=0$, we can calculate
$$
\begin{array}{lll}
 \langle A_{ij, 0}, A_{ij}\rangle&=&A_{\alpha \beta, 0} A^{\alpha \beta}+2 A_{\alpha 0, 0} A^{\alpha 0}+  A_{0 0, 0} A^{0 0}\\
 &=&( \hat{A}_{\alpha\beta}-\frac1{18} H^2 h_{\alpha\beta}) (S_{\alpha\beta}+ \frac13 \hat{\nabla}_\alpha \hat{\nabla}_\beta H+  \frac H3 \hat{A}_{\alpha\beta}   -\frac{2H^3}{27} h_{\alpha\beta} + \frac{H\hat{R}}{12} h_{\alpha\beta})\\
&&+ \frac23 \hat{\nabla}_\beta H (-\frac14 \hat{\nabla}_\beta \hat{R}+ \frac {5H}9   \hat{\nabla}_\beta H)\\
&&-(\frac16 H^2-\frac14 \hat{R})(\frac13 \hat{\triangle } H+\frac{H\hat{R}}{3}-\frac {2H^3}9)\\
&=&  \hat{A}_{\alpha\beta}S_{\alpha\beta} +\frac13 \hat{A}_{\alpha\beta} \hat{\nabla}_\alpha \hat{\nabla}_\beta H+\frac H3|\hat{A}_{\alpha\beta}|^2-\frac{4}{27} H^3\hat{R}+ \frac{4}{81} H^5+ \frac{5}{48} H\hat{R}^2\\
&&+\frac{10}{27}H|\hat{\nabla }H|^2+ \frac{1}{12} \hat{R}\hat{\triangle }H- \frac{2}{27} H^2\hat{\triangle }H-\frac16\langle \hat{\nabla }H,\hat{\nabla }\hat{R}\rangle.
 \end{array}
$$
\end{proof}

\subsection*{\underline{$1$-order expansion of Weyl tensor}}

\begin{mylemm} 
\label{Weylorder1}
Under the same assumptions as in Lemma \ref{order1},  we have on the boundary $M$
\begin{enumerate}
\item $\nabla_\gamma W_{ijkl}=0$;
\item $\nabla_0 W_{ \alpha\beta \gamma \delta}=-S_{\alpha\gamma}h_{\beta\delta}-S_{\beta\delta}h_{\alpha\gamma}+S_{\alpha\delta}h_{\beta\gamma}+S_{\beta\gamma}h_{\alpha\delta}$;
\item $\nabla_0  W_{ \alpha\beta \gamma 0}=-\hat{C}_{\gamma\beta\alpha}=-C_{\gamma\beta\alpha}$;
\item $\nabla_0  W_{ 0 \alpha 0\beta}=S_{\alpha\beta}$.
\end{enumerate}
\end{mylemm}

\bpf (1) Recall $W|_{M}= 0$ on the boundary so that $\gra_{\ga} W|_{M}= 0$.\\
(2) By the second Bianchi identity and corollary \ref{corocotten},
             \begin{eqnarray*}
             \gra_0 W_{\ga \gb \gc \gd} &=& -W_{\ga \gb 0 \gc, \gd}- W_{\ga \gb \gd 0, \gc}- 
(C_{\ga \gc 0} g_{\gb \gd}+ C_{\gb \gd 0} g_{\ga \gc}
             + C_{\ga 0 \gd} g_{\gb \gc} + C_{\gb 0 \gc} g_{\ga \gd}) \\
             &=&  - C_{\ga \gc 0} h_{\gb \gd}- C_{\gb \gd 0} h_{\ga \gc}  - 
C_{\ga 0 \gd} h_{\gb \gc} - C_{\gb 0 \gc} h_{\ga \gd}\\
             &=& -S_{\ga \gc} h_{\gb \gd} - S_{\gb \gd} h_{\ga \gc} + 
S_{\ga \gd} h_{\gb \gc}+ S_{\gb \gc} h_{\ga \gd}.
               \end{eqnarray*}

 (3) By Corollary \ref{corocotten},
     $\gra_0 W_{\ga \gb \gc 0}= -C_{\gc \gb \ga} - \gra_{\gd} W_{ \ga \gb \gc \gd}= 
-C_{\gc \gb \ga}= -\hat{C}_{\gc \gb \ga}.$

 (4) $\gra_0 W_{0 \ga 0 \gb}= -C_{\gb 0 \ga} - \gra_{\gc} W_{0 \ga \gc \gb}= -C_{\gb 0 \ga}.$
     Using Corollary \ref{corocotten}, we  get $\gra_0 W_{0 \ga 0 \gb}= S_{\ga \gb}.$

 \end{proof}

We split the tangent bundle on the boundary $T_xX=\mathbb{R}\vec{\nu}\oplus T_xM$ for all $x\in M$, where $\vec{\nu}$ is unit normal vector on the boundary. Given a tensor $T$, we decompose tensor $\nabla^{(k)}  T$ on $X$ along $M$ that are related to the splitting $T_xX=\mathbb{R}\vec{\nu}\oplus T_xM$:
let us denote by $\nabla_{odd}^{(k)}  T$ (resp. $\nabla_{even}^{(k)}  T$) the normal component $\vec{\nu}$ appeared odd time (resp. even time) in the tensor $\nabla^{(k)}  T$.

\subsection*{\underline{$1$-order expansion of Cotton tensor}}

\begin{mylemm} \label{Cottenorder2}
Suppose $g$ is a metric in $\mathcal{ B} \mathcal{M} (\bar X)$,  we have on the boundary $M$ 
\beq
\label{cottenorder2.1}
C_{\gamma\alpha\beta,0}=\hat{\nabla}_\beta (S_{\gamma\alpha})-\hat{\nabla}_\alpha (S_{\gamma\beta})+\frac{2H}{3}\hat{C}_{\gamma\alpha\beta}.
\eeq
\beq
\label{cottenorder2.2}
C_{\gamma \beta \alpha, \mu}=\hat{\nabla}_\mu \hat{C}_{\gamma\beta \alpha}+ \frac{H}{3}h_{\mu\beta} S_{ \gamma\alpha}- \frac{H}{3}h _{\mu\alpha} S_{\gamma\beta}.
\eeq
\beq
\label{cottenorder2.3}
-C_{\gamma 0 \beta, 0}=C_{\gamma \beta 0, 0}=-\hat{\nabla}^\theta \hat{C}_{\gamma \beta \theta} + \frac{2H}{3} S_{\gamma\beta}.
\eeq
\end{mylemm}
\begin{proof}
It is clear that $C_{\gamma\alpha\beta,0}= A_{\gamma\alpha,\beta0}-A_{\gamma\beta, \alpha0}$. Using the Ricci identity (\ref{Ricci}), we get
\begin{align*}
A_{\gamma\alpha,\beta0}=& A_{\gamma\alpha,0\beta}-Rm_{m\gamma0\beta}A_{m\alpha}- Rm_{m\alpha0\beta}A_{\gamma m}\\
=&A_{\gamma\alpha,0\beta} - (A_{m0}g_{\gamma\beta}+A_{\gamma\beta}g_{m0}-A_{m\beta}g_{0\gamma}- A_{0\gamma}g_{m\beta})A_{m\alpha}\\
&-(A_{m0}g_{\alpha\beta}+A_{\alpha\beta}g_{m0}-A_{m\beta}g_{0\alpha}- A_{0\alpha}g_{m\beta})A_{\gamma m}\\
=&A_{\gamma\alpha,0\beta}-A_{\gamma\beta}A_{0\alpha}-A_{\alpha\beta}A_{\gamma0}- A_{00}A_{\gamma0}h_{\alpha\beta}-A_{\theta0}A_{\gamma\theta}h_{\alpha\beta}\\
&-A_{00}A_{0\alpha}h_{\gamma\beta}-A_{\theta0}A_{\theta\alpha}h_{\gamma\beta}+ A_{0\alpha}A_{\gamma\theta}h_{\theta\beta}+A_{0\gamma}A_{\theta\alpha}h_{\theta\beta.}
\end{align*}
Here we use the fact the Weyl tensor vanishes in the boundary. Similarly,
\begin{align*}
A_{\gamma\beta,\alpha0}=& A_{\gamma\beta,0\alpha}-A_{\gamma\alpha}A_{0\beta}-A_{\beta\alpha}A_{\gamma0}- A_{00}A_{\gamma0}h_{\beta\alpha}-A_{\theta0}A_{\gamma\theta}h_{\beta\alpha}\\
&-A_{00}A_{0\beta}h_{\gamma\alpha}-A_{\theta0}A_{\theta\beta}h_{\gamma\alpha}+ A_{0\beta}A_{\gamma\theta}h_{\theta\alpha}+A_{0\gamma}A_{\theta\beta}h_{\theta\alpha}.
\end{align*}
Hence 
\begin{align*}
C_{\gamma\alpha\beta,0}=& A_{\gamma\alpha,0\beta}-A_{\gamma\beta, 0\alpha}
-A_{\gamma\beta}A_{0\alpha}+A_{\gamma\alpha}A_{0\beta}\\
&-A_{00}A_{0\alpha}h_{\gamma\beta}-A_{\theta0}A_{\theta\alpha}h_{\gamma\beta}+ A_{0\alpha}A_{\gamma\theta}h_{\theta\beta}+A_{0\gamma}A_{\theta\alpha}h_{\theta\beta}\\
&+A_{00}A_{0\beta}h_{\gamma\alpha}+A_{\theta0}A_{\theta\beta}h_{\gamma\alpha}- A_{0\beta}A_{\gamma\theta}h_{\theta\alpha}-A_{0\gamma}A_{\theta\beta}h_{\theta\alpha}.
\end{align*}
As we use the normal coordinates at the point on the boundary, we infer from Lemma \ref{order0}
\begin{align*}
C_{\gamma\alpha\beta,0}=& A_{\gamma\alpha,0\beta}-A_{\gamma\beta, 0\alpha}
-A_{00}A_{0\alpha}h_{\gamma\beta}-A_{\theta0}A_{\theta\alpha}h_{\gamma\beta}
+A_{00}A_{0\beta}h_{\gamma\alpha}+A_{\theta0}A_{\theta\beta}h_{\gamma\alpha}\\
=&A_{\gamma\alpha,0\beta}-A_{\gamma\beta, 0\alpha}+(\frac{H^2}{27}\hat{\nabla}_\beta H-\frac{\hat{R}}{12}\hat{\nabla}_\beta H+\frac{1}{3} \hat{\nabla}_\theta H\hat{A}_{\theta\beta}) h_{\gamma\alpha}\\
&-(\frac{H^2}{27}\hat{\nabla}_\alpha H-\frac{\hat{R}}{12}\hat{\nabla}_\alpha H+\frac{1}{3} \hat{\nabla}_\theta H\hat{A}_{\theta\alpha}) h_{\gamma\beta}.
\end{align*}
Using the normal coordinates at the point on the boundary, we deduce
\begin{align*}
A_{\gamma\alpha, 0\beta}=&\hat{\nabla}_\beta (A_{\gamma\alpha, 0})-\Gamma^0_{\beta\gamma} A_{0\alpha, 0}- \Gamma^0_{\beta\alpha} A_{\gamma0, 0}- \Gamma^\theta_{\beta 0} A_{\gamma\alpha, \theta}\\
=&\hat{\nabla}_\beta (A_{\gamma\alpha, 0})-\frac{H}{3}h_{\gamma\beta} A_{0\alpha, 0}- \frac{H}{3}h_{\alpha\beta} A_{\gamma0, 0}+ \frac{H}{3} A_{\gamma\alpha, \beta}.
\end{align*}
and
\begin{align*}
A_{\gamma\beta, 0\alpha}=&\hat{\nabla}_\alpha (A_{\gamma\beta, 0})-\frac{H}{3}h_{\gamma\alpha} A_{0\beta, 0}- \frac{H}{3}h_{\beta\alpha} A_{\gamma0, 0}+\frac{H}{3} A_{\gamma\beta, \alpha},
\end{align*}
so that
\begin{align*}
A_{\gamma\alpha, 0\beta}- A_{\gamma\beta, 0\alpha}
=&\hat{\nabla}_\beta (A_{\gamma\alpha, 0}) - \hat{\nabla}_\alpha (A_{\gamma\beta, 0})  -\frac{H}{3}h_{\gamma\beta} A_{0\alpha, 0}+\frac{H}{3}h_{\gamma\alpha} A_{0\beta, 0}+\frac{H}{3}\hat{C}_{\gamma\alpha\beta}.
\end{align*}
Here we use $\Gamma_{\alpha\beta}^\gamma=0$, $\Gamma_{\alpha\beta}^0=L_{\alpha\beta}=\frac{H}{3}h_{\alpha\beta}$ and $\Gamma_{\alpha 0}^\beta=-L_{\alpha }^\beta= -\frac{H}{3}\delta_{\alpha }^\beta$. 
In view of Lemma \ref{order1}, we infer
\begin{align*}
C_{\gamma\alpha\beta,0}=& \hat{\nabla}_\beta (A_{\gamma\alpha, 0}) - \hat{\nabla}_\alpha (A_{\gamma\beta, 0}) + \frac{H}{3}\hat{C}_{\gamma\alpha\beta}
\\
&+(\frac{H^2}{27}\hat{\nabla}_\beta H-\frac{\hat{R}}{12}\hat{\nabla}_\beta H+\frac{1}{3} \hat{\nabla}_\theta H\hat{A}_{\theta\beta}+\frac{H}{3}A_{0\beta,0}) h_{\gamma\alpha}\\
&-(\frac{H^2}{27}\hat{\nabla}_\alpha H-\frac{\hat{R}}{12}\hat{\nabla}_\alpha H+\frac{1}{3} \hat{\nabla}_\theta H\hat{A}_{\theta\alpha}+\frac{H}{3}A_{0\alpha,0}) h_{\gamma\beta}\\
=&\frac{H}{3}\hat{C}_{\gamma\alpha\beta}+ \hat{\nabla}_\beta (S_{\gamma\alpha}+ \hat{\nabla}_\alpha \hat{\nabla}_\gamma\frac{H}{3}+ \frac{H}{3}\hat{A}_{\gamma\alpha}-\frac{2H^3}{27}h_{\gamma\alpha}+\frac{H\hat{R}}{12}h_{\gamma\alpha})\\
& -\hat{\nabla}_\alpha (S_{\gamma\beta}+ \hat{\nabla}_\beta \hat{\nabla}_\gamma\frac{H}{3}+ \frac{H}{3}\hat{A}_{\gamma\beta}-\frac{2H^3}{27}h_{\gamma\beta}+\frac{H\hat{R}}{12}h_{\gamma\beta})\\
&+ (\frac{2H^2}{9}\hat{\nabla}_\beta H-\frac{1}{12}\hat{\nabla}_\beta (\hat{R}H)+\frac{1}{3} \hat{\nabla}_\theta H\hat{A}_{\theta\beta}) h_{\gamma\alpha}\\
&- (\frac{2H^2}{9}\hat{\nabla}_\alpha H-\frac{1}{12}\hat{\nabla}_\alpha (\hat{R}H)+\frac{1}{3} \hat{\nabla}_\theta H\hat{A}_{\theta\alpha}) h_{\gamma\beta}\\
=& \frac{H}{3}\hat{C}_{\gamma\alpha\beta}+ \hat{\nabla}_\beta (S_{\gamma\alpha})- \hat{\nabla}_\alpha (S_{\gamma\beta})+ \hat{\nabla}_\beta \hat{\nabla}_\alpha \hat{\nabla}_\gamma\frac{H}{3} - \hat{\nabla}_\alpha \hat{\nabla}_\beta \hat{\nabla}_\gamma\frac{H}{3}\\
& + \hat{\nabla}_\beta(\frac{H}{3} \hat{A}_{\gamma\alpha})- \hat{\nabla}_\alpha(\frac{H}{3} \hat{A}_{\gamma\beta}) +  \frac{1}{3}( \hat{\nabla}_\theta H)\hat{A}_{\theta\beta} h_{\gamma\alpha}-\frac{1}{3}( \hat{\nabla}_\theta H)\hat{A}_{\theta\alpha} h_{\gamma\beta}.
\end{align*}

On the other hand, by the Ricci identity (\ref{Ricci}) and the decomposition of Riemannian curvature of the boundary metric, we get
\begin{align*}
& \hat{\nabla}_\beta \hat{\nabla}_\alpha \hat{\nabla}_\gamma\frac{H}{3} - \hat{\nabla}_\alpha \hat{\nabla}_\beta \hat{\nabla}_\gamma\frac{H}{3}\\
=& -\hat{A}_{\theta\beta}(\hat{\nabla}_\theta(\frac{H}{3}))h_{\gamma\alpha}  -\hat{A}_{\gamma\alpha}\hat{\nabla}_\beta(\frac{H}{3})+ \hat{A}_{\theta\alpha}(\hat{\nabla}_\theta(\frac{H}{3}))h_{\gamma\beta}  +\hat{A}_{\gamma\beta}\hat{\nabla}_\alpha(\frac{H}{3}).
\end{align*}
Finally
\begin{align*}
C_{\gamma\alpha\beta,0}=& \frac{H}{3}\hat{C}_{\gamma\alpha\beta}+\hat{\nabla}_\beta (S_{\gamma\alpha})- \hat{\nabla}_\alpha (S_{\gamma\beta})+ \frac{H}{3}(\hat{A}_{\gamma\alpha,\beta}-\hat{A}_{\gamma\beta,\alpha})\\
=& \hat{\nabla}_\beta (S_{\gamma\alpha})- \hat{\nabla}_\alpha (S_{\gamma\beta})+ \frac{2H}{3}\hat{C}_{\gamma\alpha\beta}.
\end{align*}
Thus, we have established the desired formula (\ref{cottenorder2.1}).\\

Applying Corollary \ref{corocotten} and using the normal coordinates at the point on the boundary, we deduce
\begin{align*}
C_{\gamma \beta \alpha, \mu}&=\p_\mu C_{\gamma\beta \alpha}-\Gamma_{\mu\gamma}^0 C_{0\beta \gamma}- \Gamma_{\mu\beta}^0 C_{ \gamma0\alpha}- \Gamma_{\mu\alpha}^0 C_{\gamma\beta 0}\\
&=\hat{\nabla}_\mu \hat{C}_{\gamma\beta \alpha}- \Gamma_{\mu\beta}^0 C_{ \gamma0\alpha}- \Gamma_{\mu\alpha}^0 C_{\gamma\beta 0}\\
&=\hat{\nabla}_\mu \hat{C}_{\gamma\beta \alpha}+ \frac{H}{3}h_{\mu\beta} S_{ \gamma\alpha}- \frac{H}{3}h _{\mu\alpha} S_{\gamma\beta},
\end{align*}
which implies (\ref{cottenorder2.2}).\\
Using the Bach flat condition, the normal coordinates at the point on the boundary and the fact that the Weyl tensor vanishes on the boundary,  we have
\begin{align*}
C_{\gamma \beta i, i}=W_{\beta i\gamma l, li}=-B_{\beta\gamma}A_{li}W_{\beta i\gamma l, li}=0.
\end{align*}
Therefore, 
\begin{align*}
C_{\gamma \beta 0, 0}=C_{\gamma \beta i, i}- C_{\gamma \beta \theta, \theta}=-\hat{\nabla}_\theta \hat{C}_{\gamma \beta \theta} - \frac{H}{3}h_{\theta\beta} S_{ \gamma\theta}+ \frac{H}{3}h _{\theta\theta} S_{\gamma\beta}= -\hat{\nabla}_\theta \hat{C}_{\gamma \beta \theta} + \frac{2H}{3}S_{\gamma\beta}.
\end{align*}
Thus we have established (\ref{cottenorder2.3}) and Lemma \ref{Cottenorder2}.
\end{proof}

\subsection*{\underline{$2$-order expansion of Weyl tensor}}

\begin{mylemm} \label{Weylorder2} Suppose $g$ is a metric in $\mathcal{ B} \mathcal{M} (\bar X)$,   we have on the boundary $M$
\beq
\label{Weylorder2.1}
\begin{array}{llll}
W_{\alpha\beta\gamma\mu,00}&=&\hat{\nabla}_\gamma \hat{C}_{\mu\beta\alpha}- \hat{\nabla}_\mu \hat{C}_{\gamma\beta\alpha}\\
&&-h_{\beta\mu}(-\hat{\nabla}_\theta \hat{C}_{\alpha\gamma\theta}+\frac{5H}{3}S_{\alpha\gamma})
-h_{\alpha\gamma}(-\hat{\nabla}_\theta \hat{C}_{\beta\mu\theta}+\frac{5H}{3}S_{\mu\beta})\\
&&+ h_{\beta\gamma}(-\hat{\nabla}_\theta \hat{C}_{\alpha\mu\theta}+\frac{5H}{3}S_{\alpha\mu})
+ h_{\alpha\mu}(-\hat{\nabla}_\theta \hat{C}_{\beta\gamma\theta}+\frac{5H}{3}S_{\beta\gamma})
;
\end{array}
\eeq
\beq
\label{Weylorder2.2}
\begin{array}{lllll}
W_{\alpha\beta\gamma0,00}=2\hat{\nabla}_\beta (S_{\gamma\alpha})-2\hat{\nabla}_\alpha (S_{\gamma\beta}) +\frac{H}{3} \hat{C}_{\beta \alpha\gamma }-\frac{H}{3} \hat{C}_{ \alpha\beta\gamma };
\end{array}
\eeq
\beq
\label{Weylorder2.3}
\begin{array}{lllll}
W_{\alpha0\beta0,00}=\frac{5H}{3}S_{\alpha\beta}+\hat{\nabla}_\gamma \hat{C}_{\alpha\gamma\beta}+ \hat{\nabla}_\gamma \hat{C}_{\beta\gamma\alpha}
.
\end{array}
\eeq
\end{mylemm}

\begin{proof}
In view of \cite[Formula (2.2)]{CG}, we have
\begin{align*}
W_{\alpha\beta\gamma\mu,00}=&-\nabla_0(W_{\alpha\beta 0\gamma,\mu}+W_{\alpha\beta \mu 0,\gamma})-\nabla_0C_{\alpha\gamma 0}g_{\beta\mu}-\nabla_0 C_{\beta\mu  0}g_{\alpha\gamma}\\
&-\nabla_0C_{\alpha 0\mu}g_{\beta\gamma}-\nabla_0 C_{\beta  0\gamma}g_{\alpha\mu},
\end{align*}
since $g_{\alpha 0}= g_{\beta 0}=0$. Using the Ricci identity (\ref{Ricci}) and the fact the Weyl tensor vanishes on the boundary, we infer
$$
W_{\alpha\beta 0\gamma,\mu 0 }= W_{\alpha\beta 0\gamma,0 \mu},\; W_{\alpha\beta \mu 0,\gamma 0}= W_{\alpha\beta \mu 0, 0\gamma}.
$$
For the simplicity, we use the normal coordinates at a point on the boundary. Thus $\Gamma_{\alpha\beta}^\gamma=0$, $\Gamma_{\alpha\beta}^0=L_{\alpha\beta}$ and $\Gamma_{\alpha 0}^\beta=-L_{\alpha }^\beta$. By Lemma \ref{Weylorder1}, we have
$$
W_{\alpha\beta 0\gamma,0 }= C_{\gamma\beta \alpha},\; W_{\alpha\beta \mu 0, 0}= -C_{\mu\beta \alpha},
$$
so that it follows from Lemma \ref{Weylorder1} and Corollary \ref{corocotten}
\begin{align*}
&-\nabla_0 (W_{\alpha\beta 0\gamma,\mu})\\
=&-\nabla_\mu (W_{\alpha\beta 0\gamma,0})\\
=&-\p_\mu C_{\gamma\beta \alpha}+\Gamma_{\mu\alpha}^0 W_{0\beta 0\gamma,0}+ \Gamma_{\mu\beta}^0 W_{ \alpha0 0\gamma,0}+ \Gamma_{\mu0}^\theta W_{\alpha\beta \theta\gamma,0}+\Gamma_{\mu\gamma}^0 W_{\alpha\beta 00,0}+
\Gamma_{\mu0}^\theta W_{\alpha\beta 0\gamma,\theta}\\
=&-\hat{\nabla}_\mu \hat{C}_{\gamma\beta \alpha}+ \frac{H}{3}h_{\mu\alpha} S_{ \beta\gamma}- \frac{H}{3}h _{\mu\beta} S_{\gamma\alpha}- \frac{H}{3}W_{\alpha\beta \mu\gamma,0}\\
=&-\hat{\nabla}_\mu \hat{C}_{\gamma\beta \alpha}+ \frac{H}{3}(2h_{\mu\alpha} S_{ \beta\gamma}-2h _{\mu\beta} S_{\gamma\alpha}-h_{\gamma\alpha} S _{\mu\beta} +h_{ \beta\gamma}S_{\mu\alpha} ).
\end{align*}
Here we have used the property that  $W_{ijkl,\gamma}=0$ on the boundary by Lemma \ref{Weylorder1}.  Similarly, 
\beq
\label{eq7.1}
-\nabla_0 (W_{\alpha\beta \mu 0, \gamma})=\hat{\nabla}_\gamma \hat{C}_{\mu\beta \alpha}- \frac{H}{3}(2h_{\gamma\alpha} S_{ \beta\mu}-2h _{\gamma\beta} S_{\mu\alpha}-h_{\mu\alpha} S _{\gamma\beta} +h_{ \beta\mu}S_{\gamma\alpha} ).
\eeq
Together with (\ref{cottenorder2.2}) and (\ref{cottenorder2.3}), we deduce that
\begin{align*}
W_{\alpha\beta\gamma\mu,00}=&-\nabla_0(W_{\alpha\beta 0\gamma,\mu}+W_{\alpha\beta \mu 0,\gamma})-\nabla_0C_{\alpha\gamma 0}h_{\beta\mu}-\nabla_0 C_{\beta\mu  0}h_{\alpha\gamma}\\
&+\nabla_0C_{\alpha \mu0}h_{\beta\gamma}+\nabla_0 C_{\beta  \gamma 0}h_{\alpha\mu}\\
=&\hat{\nabla}_\gamma \hat{C}_{\mu\beta \alpha}-\hat{\nabla}_\mu \hat{C}_{\gamma\beta \alpha}- Hh_{\mu\beta} S_{ \gamma\alpha}+ Hh _{\mu\alpha} S_{\gamma\beta} + Hh_{\gamma\beta} S_{ \mu\alpha}- Hh _{\gamma\alpha} S_{\mu\beta}\\
&+ h_{\beta\mu}(\hat{\nabla}_\theta \hat{C}_{\alpha\gamma  \theta} - \frac{2H}{3} S_{\alpha\gamma}) +h_{\alpha\gamma}(\hat{\nabla}_\theta \hat{C}_{\beta\mu  \theta}- \frac{2H}{3} S_{\beta\mu })\\
&-h_{\beta\gamma}(\hat{\nabla}_\theta \hat{C}_{\alpha\mu  \theta} - \frac{2H}{3} S_{\alpha\mu }) -h_{\alpha\mu}(\hat{\nabla}_\theta \hat{C}_{\beta\gamma  \theta} - \frac{2H}{3} S_{\beta\gamma  }) \\
=&\hat{\nabla}_\gamma \hat{C}_{\mu\beta \alpha}-\hat{\nabla}_\mu \hat{C}_{\gamma\beta \alpha}\\
&+ h_{\beta\mu}(\hat{\nabla}_\theta \hat{C}_{\alpha\gamma  \theta} - \frac{5H}{3} S_{\alpha\gamma}) +h_{\alpha\gamma}(\hat{\nabla}_\theta \hat{C}_{\beta\mu  \theta} - \frac{5H}{3}S_{\beta\mu })\\
&-h_{\beta\gamma}(\hat{\nabla}_\theta \hat{C}_{\alpha\mu  \theta} - \frac{5H}{3} S_{\alpha\mu }) -h_{\alpha\mu}(\hat{\nabla}_\theta \hat{C}_{\beta\gamma  \theta} - \frac{5H}{3} S_{\beta\gamma  }) .
\end{align*}
Thus we have established (1).\\
We calculate with normal coordinates at a point on the boundary
\begin{align*}
W_{\alpha\beta\gamma0,00}= W_{\alpha\beta\gamma l,l0}-  W_{\alpha\beta\gamma \theta,\theta 0}= C_{\gamma\alpha\beta ,0}-  W_{\alpha\beta\gamma \theta,\theta
 0}.
\end{align*}
On the other hand, it follows from the Ricci identity (\ref{Ricci}) and the fact the Weyl tensor vanishes on the boundary that
$$
W_{\alpha\beta\gamma \theta,\mu 0}= W_{\alpha\beta\gamma \theta,0\mu},
$$
which implies
\begin{align*}
W_{\alpha\beta\gamma \theta,0\mu}= &\hat{\nabla}_\mu (W_{\alpha\beta\gamma \theta,0})-\Gamma^0_{\mu\alpha}W_{0\beta\gamma \theta,0}- \Gamma^0_{\mu\beta}W_{\alpha0\gamma \theta,0}-\Gamma^0_{\mu\gamma}W_{\alpha\beta0 \theta,0}- \Gamma^0_{\mu\theta}W_{\alpha\beta\gamma 0,0}\\
&-\Gamma^\nu_{\mu0}W_{\alpha\beta\gamma \theta,\nu} \\
=&\hat{\nabla}_\mu (W_{\alpha\beta\gamma \theta,0})-\frac{H}{3}h_{\alpha\mu}W_{0\beta\gamma \theta,0}- \frac{H}{3}h_{\beta\mu}W_{\alpha0\gamma \theta,0} \\
&-\frac{H}{3}h_{\gamma\mu}W_{\alpha\beta0 \theta,0}- \frac{H}{3}h_{\theta\mu}W_{\alpha\beta\gamma 0,0}.
\end{align*}
Here we have used the fact $W_{ijkl,\gamma}=0$ by Lemma \ref{Weylorder1}. Also due to Lemma \ref{Weylorder1}, we deduce 
\begin{align*}
W_{\alpha\beta\gamma \theta,0\mu} =&\hat{\nabla}_\mu (-S_{\alpha\gamma }h_{\beta\theta}-S_{\beta\theta}h_{\alpha\gamma }+ S_{\alpha\theta }h_{\beta\gamma}+S_{\beta\gamma}h_{\alpha\theta} )\\
&-\frac{H}{3}h_{\alpha\mu}\hat{C}_{\beta\theta\gamma}+ \frac{H}{3}h_{\beta\mu}\hat{C}_{\alpha\theta\gamma}-\frac{H}{3}h_{\gamma\mu}\hat{C}_{\theta\beta\alpha}+ \frac{H}{3}h_{\theta\mu}\hat{C}_{\gamma\beta \alpha}\\
=& -(\hat{\nabla}_\mu S_{\alpha\gamma })h_{\beta\theta}-(\hat{\nabla}_\mu S_{\beta\theta})h_{\alpha\gamma }+ (\hat{\nabla}_\mu S_{\alpha\theta )}h_{\beta\gamma}+ (\hat{\nabla}_\mu S_{\beta\gamma})h_{\alpha\theta} \\
&-\frac{H}{3}h_{\alpha\mu}\hat{C}_{\beta\theta\gamma}+ \frac{H}{3}h_{\beta\mu}\hat{C}_{\alpha\theta\gamma}-\frac{H}{3}h_{\gamma\mu}\hat{C}_{\theta\beta\alpha}+ \frac{H}{3}h_{\theta\mu}\hat{C}_{\gamma\beta \alpha}.
\end{align*}
Combining with (\ref{cottenorder2.1}), we get
\begin{align*}
W_{\alpha\beta\gamma0,00}= &\hat{\nabla}_\beta (S_{\gamma\alpha})-\hat{\nabla}_\alpha (S_{\gamma\beta})+\frac{2H}{3}\hat{C}_{\gamma\alpha\beta}\\
&+(\hat{\nabla}_\theta S_{\alpha\gamma })h_{\beta\theta}+(\hat{\nabla}_\theta S_{\beta\theta})h_{\alpha\gamma }- (\hat{\nabla}_\theta S_{\alpha\theta )}h_{\beta\gamma}- (\hat{\nabla}_\theta S_{\beta\gamma})h_{\alpha\theta} \\
&+\frac{H}{3}h_{\alpha\theta}\hat{C}_{\beta\theta\gamma}- \frac{H}{3}h_{\beta\theta}\hat{C}_{\alpha\theta\gamma}+\frac{H}{3}h_{\gamma\theta}\hat{C}_{\theta\beta\alpha}- \frac{H}{3}h_{\theta\theta}\hat{C}_{\gamma\beta \alpha}\\
=& 2\hat{\nabla}_\beta (S_{\gamma\alpha})-2\hat{\nabla}_\alpha (S_{\gamma\beta}) +(\hat{\nabla}_\theta S_{\beta\theta})h_{\alpha\gamma }- (\hat{\nabla}_\theta S_{\alpha\theta )}h_{\beta\gamma} \\
& +\frac{H}{3} \hat{C}_{\beta \alpha\gamma }-\frac{H}{3} \hat{C}_{ \alpha\beta\gamma }\\
=& 2\hat{\nabla}_\beta (S_{\gamma\alpha})-2\hat{\nabla}_\alpha (S_{\gamma\beta}) +\frac{H}{3} \hat{C}_{\beta \alpha\gamma }-\frac{H}{3} \hat{C}_{ \alpha\beta\gamma },
\end{align*}
since $S$ is of divergence free. Hence we have established (\ref{Weylorder2.2}).\\
Now, we first calculate
\begin{align*}
W_{\alpha0\beta0,00}=W_{\alpha k\beta l,kl} -W_{\alpha \mu\beta \nu,\mu\nu}-
W_{\alpha 0\beta \nu,0\nu}-W_{\alpha \mu\beta 0,\mu0}.
\end{align*}
As before, we have on the boundary
\begin{align*}
W_{\alpha k\beta l,kl}=0,\;
W_{\alpha \mu\beta 0,\mu0}= W_{\alpha \mu\beta 0,0\mu}.
\end{align*}
Thus it follows from (\ref{eq7.1})
\begin{align*}
W_{\alpha 0\beta \nu,0\nu}
= -\hat{\nabla}\nu \hat{C}_{\alpha \nu\beta}-H S_{\alpha \beta},
\end{align*}
and
\begin{align*}
W_{\alpha \mu\beta 0,\mu 0}
=-\hat{\nabla}^\mu \hat{C}_{ \beta\mu\alpha}-H S_{ \beta\alpha}.
\end{align*}
On the other hand, it follows from Lemma \ref{Weylorder1} that
\begin{align*}
W_{\alpha \mu\beta \nu,\mu\nu}=&
\hat{\nabla}_\nu (W_{\alpha \mu\beta \nu,\mu})-\Gamma_{\nu\alpha}^0 W_{0 \mu\beta \nu,\mu}-\Gamma_{\nu\mu}^0 W_{\alpha 0\beta \nu,\mu}\\
&- \Gamma_{\nu\beta}^0 W_{\alpha \mu0 \nu,\mu}-\Gamma_{\nu\nu}^0 W_{\alpha \mu\beta 0,\mu}-\Gamma_{\nu\mu}^0 W_{\alpha \mu\beta \nu,0}\\
=&
\hat{\nabla}_\nu (W_{\alpha \mu\beta \nu,\mu})-\frac{H}{3}h_{\nu\alpha} W_{0 \mu\beta \nu,\mu}-\frac{H}{3}h_{\nu\mu} W_{\alpha 0\beta \nu,\mu}\\
&- \frac{H}{3}h_{\nu\beta} W_{\alpha \mu0 \nu,\mu}-\frac{H}{3}h_{\nu\nu} W_{\alpha \mu\beta 0,\mu}-\frac{H}{3}h_{\nu\mu} W_{\alpha \mu\beta \nu,0}\\
=&\frac{H}{3}(S_{\alpha\beta}h_{\mu\mu}+S_{\mu\mu} h_{\alpha\beta}-S_{\alpha\mu} h_{\beta\mu}-S_{\beta\mu}h_{\alpha\mu} )= \frac{H}{3}S_{\alpha\beta}.
\end{align*}
Here we use the fact that $S$ is of trace free. Finally, we have
\begin{align*}
W_{\alpha0\beta0,00}=& - \frac{H}{3}S_{\alpha\beta}-
(-\hat{\nabla}^\nu \hat{C}_{\alpha \nu\beta}-H S_{\alpha \beta})-(-\hat{\nabla}^\mu \hat{C}_{ \beta\mu\alpha}-H S_{ \beta\alpha})\\
=&\frac{5H}{ 3} S_{\alpha\beta}+\hat{\nabla}^\gamma \hat{C}_{\alpha\gamma\beta}+ \hat{\nabla}^\gamma \hat{C}_{\beta\gamma\alpha}.
\end{align*}
Thus we have proved (\ref{Weylorder2.3}).
\end{proof}

\subsection*{\underline{Expansion of Weyl tensor for adapted metrics}}

s a consequence, we obtain the following asymptotic expansion of the Weyl curvature on CCE manifolds..

\begin{mycoro}
\label{Weylexpansion}  Suppose $g \in \mathcal{ B} \mathcal{M} (\bar X)$,   we have on the boundary
\beq
\label{coeff2}
 \langle \nabla_0 W_{ijkl},\nabla_0 W_{ijkl}\rangle= 8|S|^2+4|\hat C|^2.
\eeq
Moreover
\beq
\label{integralcoeff3}
\begin{array}{llll}
&\oint\langle \nabla_0 \nabla_0 W_{ijkl},\nabla_0 W_{ijkl}\rangle \\
=& \oint \frac{40H}{3}|S|^2-32S^{\beta\mu}\hat{\nabla}^\theta \hat{C}_{\beta\mu\theta}+\frac{8H}{3} \hat{C}_{ \alpha\beta\mu }\hat{C}^{\mu\beta\alpha}.
\end{array}
 \eeq
\end{mycoro}
\begin{proof}
It follows from Lemma \ref{Weylorder1} that
\begin{align*}
 \langle \nabla_0 W_{ijkl},\nabla_0 W_{ijkl}\rangle
 = & \langle \nabla_0 W_{\alpha\beta\mu\nu},\nabla_0 W_{\alpha\beta\mu\nu}\rangle \\
 &  + 4 \langle \nabla_0 W_{\alpha\beta\mu0},\nabla_0 W_{\alpha\beta\mu0}\rangle  +4 \langle \nabla_0 W_{\alpha0\beta0},\nabla_0 W_{\alpha0\beta0}\rangle
\\
=& 4|S|^2 +  4|\hat{C}|^2+ 4|S|^2= 8|S|^2 +  4|\hat{C}|^2.
\end{align*}
Similarly, we write
\begin{align*}
\langle \nabla_0 \nabla_0 W_{ijkl},\nabla_0 W_{ijkl}\rangle
 = & \langle \nabla_0\nabla_0 W_{\alpha\beta\mu\nu},\nabla_0 W_{\alpha\beta\mu\nu}\rangle \\
&+4 \langle \nabla_0\nabla_0 W_{\alpha\beta\mu0},\nabla_0 W_{\alpha\beta\mu0}\rangle +4 \langle \nabla_0\nabla_0 W_{\alpha0\beta0},\nabla_0 W_{\alpha0\beta0}\rangle
\\
:=& I+II+III.
\end{align*}
We calculate by Lemmas \ref{Weylorder1} and \ref{Weylorder2}
\begin{align*}
III= \frac{20}{3} H|S|^2+ 4S^{\alpha\beta}( \hat{\nabla}^\theta\hat{C}_{\alpha\theta\beta}+ \hat{\nabla}^\theta\hat{C}_{\beta\theta\alpha}),
\end{align*}
and
\begin{align*}
II=& 4[2\hat{\nabla}_\beta (S_{\mu\alpha})-2\hat{\nabla}_\alpha (S_{\mu\beta}) +\frac{H}{3} \hat{C}_{\beta \alpha\mu }-\frac{H}{3} \hat{C}_{ \alpha\beta\mu }](-\hat{C}^{\mu\beta\alpha})\\
=& 4[-2\hat{\nabla}_\beta (S_{\mu\alpha})+2\hat{\nabla}_\alpha (S_{\mu\beta}) -\frac{H}{3} \hat{C}_{\beta \alpha\mu }+\frac{H}{3} \hat{C}_{ \alpha\beta\mu }]\hat{C}^{\mu\beta\alpha}.
\end{align*}
To compute $I$, we decompose $I$ into 4 parts
$$
I:=I_1+I_2+I_3+ I_4,
$$
where
\begin{align*}
I_1= -\nabla_0\nabla_0 W_{\alpha\beta\mu\nu}S^{\alpha\mu}h^{\beta\nu},\\
I_2= -\nabla_0\nabla_0 W_{\alpha\beta\mu\nu}S^{\beta\nu} h^{\alpha\mu},\\
I_3=\nabla_0\nabla_0 W_{\alpha\beta\mu\nu}S^{\alpha\nu} h^{\beta\mu,}\\
I_4=\nabla_0\nabla_0 W_{\alpha\beta\mu\nu} S^{\beta\mu}h^{\alpha\nu}.
\end{align*}
We have
\begin{align*}
I_1=& -S^{\alpha\mu}\hat{\nabla}_\mu ((\hat{C}_\nu)^\nu)_\alpha+ S^{\alpha\mu}\hat{\nabla}^\nu \hat{C}_{\mu\nu\alpha}+S^{\alpha\mu}(-\hat{\nabla}^\theta \hat{C}_{\alpha\mu\theta}+HS_{\alpha\mu})\\
=&\frac{5H}{3}|S|^2- S^{\alpha\mu}\hat{\nabla}^\theta \hat{C}_{\alpha\mu\theta}-S^{\alpha\mu}\hat{\nabla}_\mu ((\hat{C}_\nu)^\nu)_\alpha+ S^{\alpha\mu}\hat{\nabla}^\nu \hat{C}_{\mu\nu\alpha}.
\end{align*}
Here we use the property that  $S$ is  trace free. Similarly, we have
\begin{align*}
I_2=&\frac{5H}{3}|S|^2- S^{\beta\nu}\hat{\nabla}^\theta \hat{C}_{\beta\nu\theta}-S^{\beta\nu}\hat{\nabla}^\theta\hat{C}_{\nu\beta\theta}+ S^{\beta\nu}\hat{\nabla}_\nu (\hat{C}_{\theta\beta})^\theta.
\end{align*}
\begin{align*}
I_3=&\frac{5H}{3}|S|^2+S^{\alpha\nu}\hat{\nabla}^\theta \hat{C}_{\nu\theta\alpha}-S^{\alpha\nu}\hat{\nabla}^\theta\hat{C}_{\alpha\nu\theta}- S^{\alpha\nu}\hat{\nabla}_\nu ((\hat{C}_{\theta})^\theta)_\alpha.
\end{align*}
\begin{align*}
I_4=&\frac{5H}{3}|S|^2-S^{\beta\mu}\hat{\nabla}^\theta \hat{C}_{\beta\mu\theta}+S^{\beta\mu}\hat{\nabla}^\mu(\hat{C}_{\theta\beta})^\theta- S^{\beta\mu}\hat{\nabla}^\theta \hat{C}_{\mu\beta\theta}.
\end{align*}
Thus
\begin{align*}
I=&\frac{20H}{3}|S|^2-4S^{\beta\mu}\hat{\nabla}^\theta \hat{C}_{\beta\mu\theta}-4S^{\beta\mu}\hat{\nabla}^\theta \hat{C}_{\mu\beta\theta}
-4S^{\beta\mu}\hat{\nabla}_\mu((\hat{C}_{\theta})^\theta)_\beta),
\end{align*}
and 
\begin{equation}
\label{coeff3}
\begin{array}{llll}
I+II+III&=&\frac{40H}{3}|S|^2-8S^{\beta\mu}\hat{\nabla}^\theta \hat{C}_{\beta\mu\theta}-8S^{\beta\mu}\hat{\nabla}^\theta \hat{C}_{\mu\beta\theta}
-4S^{\beta\mu}\hat{\nabla}_\mu((\hat{C}_{\theta})^\theta)_\beta)\\
&& +4[-2\hat{\nabla}_\beta (S_{\mu\alpha})+2\hat{\nabla}_\alpha (S_{\mu\beta}) -\frac{H}{3} \hat{C}_{\beta \alpha\mu }+\frac{H}{3} \hat{C}_{ \alpha\beta\mu }]\hat{C}^{\mu\beta\alpha}\\
&=&\frac{40H}{3}|S|^2-8S^{\beta\mu}\hat{\nabla}^\theta \hat{C}_{\beta\mu\theta}-8S^{\beta\mu}\hat{\nabla}^\theta \hat{C}_{\mu\beta\theta}\\
&& +4[-2\hat{\nabla}_\beta (S_{\mu\alpha})+2\hat{\nabla}_\alpha (S_{\mu\beta}) -\frac{H}{3} \hat{C}_{\beta \alpha\mu }+\frac{H}{3} \hat{C}_{ \alpha\beta\mu }]\hat{C}^{\mu\beta\alpha}
\end{array}
\end{equation}
since $((\hat{C}_{\theta})^\theta)_\beta=0$. Thus we established (\ref{coeff2}). By the integration by parts, we get (\ref{integralcoeff3}).
\end{proof}


A direct calculation leads to
\begin{mycoro}
\label{Weylexpansion1} For the adapted metric $g$,  we have the expansion of Weyl tensor near the boundary
\beq
\label{Weylexpansion2}
 |W|^2=e_2 r^2+e_3 r^3 +o(r^3),
\eeq
where $r$ is the distance function to the boundary and
\begin{align}
\label{coffecient_e2}
e_2=\langle \nabla_0  W_{ijkl},\nabla_0 W_{ijkl}\rangle =  8|S|^2+4|\hat C|^2,\\
\label{coffecient_e3}
e_3=\frac{2}{3}\langle \nabla_0\nabla_0  W_{ijkl},\nabla_0 W_{ijkl}\rangle.
\end{align}
 We remark the expansion of $e_3$ is  expressed in (\ref{coeff3}). 
\end{mycoro}
\begin{proof}
This follows from Corollary (\ref{Weylexpansion2}) due to the fact that the Weyl tensor vanishes on the boundary.
\end{proof}

\subsection*{\underline{High order expansion of curvatures}}

\begin{mylemm}
\label{highorderexpansion}
Under the same assumptions as in Lemma \ref{order1},  we have on the boundary $M$ for $k\ge 1,$
$$
\nabla^{(k)}_{even} A = L(\hat \nabla^{(k)} \hat A)+H*\nabla^{(k-1)}_{odd} A+ \sum_{l=0}^{k-2}\nabla^{(l)}Rm*P_{k-2-l}(H, \nabla A).
$$
$$
\nabla^{(k)}_{odd} A = L(\hat \nabla^{(k-1)}S,\hat \nabla^{(k+1)} H )+ H*\nabla^{(k-1)}_{even} A+\sum_{l=0}^{k-2}\nabla^{(l)}Rm*P_{k-2-l}(H, \nabla A)
$$
$$
\nabla^{(k)}_{even} W =L(\hat\nabla^{(k)} \hat A)+ H*\nabla^{(k-1)}_{odd} W+\sum_{l=0}^{k-2}\nabla^{(l)}Rm*P_{k-2-l}(H, \nabla Rm).
$$
$$
\nabla^{(k)}_{odd} W =L(\hat\nabla^{(k-1)} \hat S, ,\hat \nabla^{(k+1)} H)+ H*\nabla^{(k-1)}_{even} W+\sum_{l=0}^{k-2}\nabla^{(l)}Rm*P_{k-2-l}(H, \nabla Rm), 
$$
where $L$ is some linear function and the sum disappears when $k=1$, $P_{k-2-l}(H, \nabla A)=\sum_{ i_1+\cdots +i_j=k+2-l-2j} H^{i_1}* \nabla^{(i_2)} A*\cdots* \nabla^{(i_j)} A$ and  $P_{k-2-l}(H, \nabla Rm)=\sum_{ i_1+\cdots +i_j=k+2-l-2j} H^{i_1}* \nabla^{(i_2)} Rm*\cdots* \nabla^{(i_j)} Rm$. Moreover, $\nabla^{(k)} A$ and $\sum_{l=0}^{k-2}\nabla^{(l)}Rm*P_{k-2-l}(H, \nabla A)$ (resp. $\nabla^{(k)}W$ and $\sum_{l=0}^{k-2}\nabla^{(l)}Rm*P_{k-2-l}(H, \nabla Rm)$) have the same parity. In particular, when the boundary metric $h$ is flat, we have
$$
\begin{array}{ll}
\nabla^{(k)}_{even} A = H*\nabla^{(k-1)}_{odd} A+ \sum_{l=0}^{k-2}\nabla^{(l)}Rm*\nabla^{(k-2-l)}A.\\
\nabla^{(k)}_{even} W =H*\nabla^{(k-1)}_{odd} W+\sum_{l=0}^{k-2}\nabla^{(l)}Rm*\nabla^{(k-2-l)}Rm.\\
\end{array}
$$
\end{mylemm}
\begin{proof} We will prove  the result by induction. For $k=1$, the results follow from Lemmas  \ref{order0}, \ref{order1} and  \ref{Weylorder1}. 
We prove the result by induction. We consider  the terms $\nabla^{(k+1)}_{even} A$ in three cases.\\
a) $\nabla^{(k+1)}_{even} A=\nabla_0\nabla_\alpha \nabla^{(k-1)}_{old} A$.\\
In such case, by the Ricci identity (\ref{Ricci})
$$
\nabla^{(k+1)}_{even} A=\nabla_\alpha\nabla_0 \nabla^{(k-1)}_{odd} A+Rm*\nabla^{(k-1)} A=\nabla_\alpha \nabla^{(k)}_{even} A+Rm*\nabla^{(k-1)} A.
$$
We have 
$$
\nabla_\alpha\nabla^{(k)}_{even} A=  \hat{\nabla}_\alpha \nabla^{(k)}_{even} A+\Gamma_{\alpha i}^j* \nabla^{(k)}_{odd} A
= \hat{\nabla}_\alpha \nabla^{(k)}_{even} A+H* \nabla^{(k)}_{odd} A.
$$
Here we use the facts $\Gamma_{\alpha 0}^0=0$,  $\Gamma_{\alpha 0}^\beta=-{L_{\alpha}}^{\beta}$ and $\Gamma_{\alpha \beta}^0=L_{\alpha\beta}$. We prove the inequality for general $k$ by induction. In fact,
$$
\nabla^{(k)} A = L(\hat \nabla^{(k)} \hat A)+H*\nabla^{(k-1)} A+ \sum_{l=0}^{k-2}\nabla^{(l)}Rm*P_{k-2-l}(H, \nabla A),
$$ 
so that
$$
\hat{\nabla}_\alpha \nabla^{(k)} A = L(\hat \nabla^{(k+1)} \hat A)+\hat{\nabla}_\alpha (H*\nabla^{(k-1)} A)+ \hat{\nabla}_\alpha( \sum_{l=0}^{k-2}\nabla^{(l)}Rm*P_{k-2-l}(H, \nabla A)).
$$
We now estimate apply Lemma \ref{order0},
$$
\begin{array}{ll}
\hat{\nabla}_\alpha (H*\nabla^{(k-1)} A)&=  \hat{\nabla}_\alpha H *\nabla^{(k-1)} A+ H*   \hat{\nabla}_\alpha \nabla^{(k-1)} A\\
&=A_{\alpha 0} *\nabla^{(k-1)} A+ H*(\nabla^{(k)} A-H*\nabla^{(k-1)} A)\\
&=A_{\alpha 0} *\nabla^{(k-1)} A+ H*\nabla^{(k)} A+H^2*\nabla^{(k-1)} A.
\end{array}
$$
Similarly, we have
$$
\hat{\nabla}_\alpha( \sum_{l=0}^{k-2}\nabla^{(l)}Rm*P_{k-2-l}(H, \nabla A))=  \sum_{l=0}^{k-1}\nabla^{(l)}Rm*P_{k-1-l}(H, \nabla A).
$$
\\
b) $\nabla^{(k+1)}_{even} A=\nabla_0\nabla_0 \nabla^{(k-1)}_{even} A$.\\
We write
$$
\nabla^{(k+1)}_{even} A=\triangle  \nabla^{(k-1)}_{even} A-\nabla_\alpha\nabla_\alpha \nabla^{(k-1)}_{even} A.
$$
We get
$$
\nabla^{(k+1)}_{even} A=\triangle  \nabla^{(k-1)}_{even} A-\nabla_\alpha\nabla_\alpha \nabla^{(k-1)}_{even} A=  \nabla^{(k-1)}_{even}\triangle A-\nabla_\alpha\nabla_\alpha \nabla^{(k-1)}_{even} A+\sum_{l=0}^{k-1}\nabla^{(l)}Rm*\nabla^{(k-1-l)}A.
$$
From the Bach flat equation (\ref{Bachflat1}), we could write
$$
\nabla^{(k+1)}_{even} A=-\nabla_\alpha\nabla_\alpha \nabla^{(k-1)}_{even} A+\sum_{l=0}^{k-1}\nabla^{(l)}Rm*\nabla^{(k-1-l)}A.
$$
It follows from  the assumptions of  the induction.\\
c) $\nabla^{(k+1)}_{even} A=\nabla_\alpha  \nabla^{(k)}_{even} A$.\\
It is clear in this case by the induction as a). By the same argument, we prove the similar result for $\nabla^{(k)}_{odd} A$.\\
The proof for the Weyl tensor is quite similar as the schouten tensor $A$. The parity relations in the above equalities are clear. We omit the details.  In particular, when the boundary metric $h$ is flat, we recall the fact by Lemma \ref{order0}
$$
A_{00}=\frac {H^2}{6},  A_{\alpha 0}=\frac13\hat\nabla_\alpha H.
$$
Thus the desired result follows. In fact, when $k\ge 2$, we have
$$
\begin{array}{lll}
\nabla^{(k)} A&=&\nabla_{\alpha}\nabla^{(k-1)}A+\sum_{l=0}^{k-2}\nabla^{(l)}Rm*\nabla^{(k-2-l)}A\\
&=& \hat \nabla_{\alpha}\nabla^{(k-1)}A+H*\nabla^{(k-1)}A +\sum_{l=0}^{k-2}\nabla^{(l)}Rm*\nabla^{(k-2-l)}A.
\end{array}
$$
By the induction,
$$
\nabla^{(k-1)}_{even} A=H*\nabla^{(k-2)}A +\sum_{l=0}^{k-3}\nabla^{(l)}Rm*\nabla^{(k-3-l)}A.
$$
so that
$$
\nabla_{\alpha}\nabla^{(k-1)}_{even} A= \nabla_{\alpha}(\sum_{l=0}^{k-3}\nabla^{(l)}Rm*\nabla^{(k-3-l)}A)+  \nabla_{\alpha}(H*\nabla^{(k-2)}A).
$$
We calculate
$$
\begin{array}{lll}
\nabla_{\alpha}(H*\nabla^{(k-2)}A)&=&\hat\nabla_{\alpha}(H*\nabla^{(k-2)}A)+ H^2*\nabla^{(k-2)}A\\
&=&\hat\nabla_{\alpha}H* \nabla^{(k-2)}A+ H*\hat\nabla_{\alpha}  \nabla^{(k-2)}A+ H^2*\nabla^{(k-2)}A\\
&=& H*\nabla_{\alpha}  \nabla^{(k-2)}A+ H^2*\nabla^{(k-2)}A+\hat\nabla_{\alpha}H* \nabla^{(k-2)}A\\
&=& H* \nabla^{(k-1)}A+ A*  \nabla^{(k-2)}A.
\end{array}
$$
Similarly, we have the same result for the Weyl tensor. Here we omit the details.
Therefore, we have established the lemma.\\
 \end{proof}

\section{Some integral formulas}\label{B}

In the rest of this section, we will derive two general integral formulas for the $L^2$ Weyl and Traceless Ricci  curvature $E$ for metric $g \in \mathcal{ B} \mathcal{M} (\bar X)$. These two formulas have been derived in the early paper of \cite{CGY-IHES}, formulas (3.7) in Proposition 3.2 and (3.8) in Proposition 3.30), when $X$ is a compact closed 4-manifold without boundary, here we need to compute the boundary terms.

\begin{mylemm}
For $g\in \mathcal{ B} \mathcal{M} (\bar X)$, we have 
\label{LemmaRiccibis}
\beq
\label{GradientRicci1}
\begin{array}{llll}
  \ds\int |\nabla E|^2
& =   \ds 2\int W_{ikjl}E^{kl}E^{ij}  - 2\int tr(E^3) + \frac{1}{12} \int |\nabla R|^2 - \int \frac{1}{3} R |E|^2 \\ &-\ds \{ 
4\oint \langle \hat{E}, S\rangle+\frac{1}{3}H\hat{\triangle}\hat{R} + \frac{H}{3}|\hat{E}|^2+ \frac{H}{9}|\hat{R}-\frac{2H^2}{3}|^2+\frac{14H}{27}|\hat{\nabla}H|^2 \} \\
&\ds-\oint R(-\frac49 \hat{\triangle}H-\frac1{18}H\hat{R}-\frac1{12}HR +\frac1{27}H^3).
\end{array}
\eeq 
\end{mylemm}

\vskip .2in
\vskip .2in


\begin{proof}
Recall the Bach flat condition (\ref{Bachflat1}). Thus the metric  $g \in \mathcal{ B} \mathcal{M} (\bar X)$, we have 
\vskip .1in
\vskip .in
\beq
\label{Bachflat1bisbis}
\triangle A _{ij}-\frac{1}{6}\nabla_i\nabla_j R+2W_{ikjl}A^{kl}+ |A|^2g_{ij}- 4(A^2)_{ij}=0, 
\eeq
where $ A_{ij} = \frac{1}{2} (Ric_{ij} - \frac{1}6 R g_{ij}) $, is the Schouten tensor. If we rewrite $Ric_{ij}=E_{ij} + \frac{1}{4} R g_{ij} $, where  $E$ is the traceless Ricci tensor. Taking $E$ as the test function and using the integration by parts , we infer
\beq
\label{eq9.1}
\begin{array}{llll}
 &\ds 2\int W_{ikjl}E^{kl}E^{ij}-2\int tr(E^3)  - \int \frac{1}{3} R |E|^2  \\
=&\ds-\int\langle \triangle E,E\rangle+\frac{1}{3}\int\langle \nabla\nabla R,E\rangle\\
=&\ds \int |\nabla E|^2-\frac{1}{12}\int |\nabla R|^2+\oint \langle  E_{,0},E\rangle- \frac{1}{3}\oint \langle  \nabla R,E(0,\cdot)\rangle.
\end{array}
\eeq
We then decompose 
$$
\langle  E_{,0},E\rangle= E_{\alpha\beta, 0} E^{\alpha\beta}+ 2 E_{\alpha0, 0} E^{\alpha0}+ E_{00, 0} E^{00}:=I+II+III.
$$
Applying Lemmas \ref{order0} and \ref{order1} and by the fact $S$ is trace free, we have (to correct)
\begin{align*}
I & =4\langle \hat{Ric},S\rangle+ \frac43\langle\hat{Ric},\hat{\nabla}^2H\rangle- \frac13\hat{R}\hat{\triangle}H+\frac{4}{3}H|\hat{Ric}|^2-\frac{2}{27}H^2\hat{\triangle}H- \frac13 H\hat{R}^2-\frac{4}{27}H^3\hat{R} +\frac{4}{81}H^5\\
&-\frac{R}{12}(\frac{2}{3} \hat{\triangle}H-\frac{4}{3} H\hat{R}+\frac{8}{9} H^3+RH-\frac{1}{4} \nabla_0 R) -\frac{\nabla_0 R}{12}(\frac{\hat{R}}{2} -\frac{H^2}{3}).\\
II & =\frac{40}{27}H|\hat{\nabla}H|^2-\frac{2}{3}\langle \hat{\nabla}H, \hat{\nabla}\hat{R}\rangle +\frac{4}{9}\nabla_\alpha R \hat{\nabla}_\alpha H.\\
III & = -2(\frac{1}{3}H^2-\frac{1}{2}\hat{R})(\frac{1}{3}\hat{\triangle}H+ \frac{1}{3}\hat{R}H-\frac{2}{9}H^3)\\
&-\frac{R}{4}(\frac{2}{3} \hat{\triangle}H+\frac{2}{3} H\hat{R}-\frac{4}{9} H^3-\frac{1}{4} \nabla_0 R) -\frac{\nabla_0 R}{4}(\frac{\hat{R}}{2} -\frac{H^2}{3}).
\end{align*}
On the other hand
\begin{align*}
-\frac{1}{3}\langle R, E(0,\cdot)\rangle=-\frac{2}{9}\nabla_\alpha R \hat{\nabla}_\alpha H+ \frac{\nabla_0 R}{3}(\frac{\hat{R}}{2} -\frac{H^2}{3}-\frac{R}{4})
\end{align*}
These estimates yield by Bianchi's identity $2 \hat{\nabla}^\alpha \hat{Ric}_{\alpha\beta}= \hat{\nabla}_\beta\hat{R}$ and the fact $ |\hat{Ric}|^2=|\hat{E}|^2+\frac{1}{3}\hat{R}^2$.
\begin{align*}
\oint (\langle  E_{,0},E\rangle-\frac{1}{3}\langle R, E(0,\cdot)\rangle)=&
4\oint \langle \hat{E}, S\rangle+\frac{1}{3}H\hat{\triangle}\hat{R} + \frac{H}{3}|\hat{E}|^2+ \frac{H}{9}|\hat{R}-\frac{2H^2}{3}|^2+\frac{14H}{27}|\hat{\nabla}H|^2\\
&+\oint R(-\frac49 \hat{\triangle}H-\frac1{18}H\hat{R}-\frac1{12}HR +\frac1{27}H^3)
\end{align*}
since
$$
2\oint\langle\hat{Ric},\hat{\nabla}\hat{\nabla}H\rangle=- \oint \langle  \hat{\nabla}H, \hat{\nabla}\hat{R}\rangle= \oint \hat{R}\hat{\triangle}H= \oint H\hat{\triangle}\hat{R}.
$$
Combining this, we get (\ref{GradientRicci1}).

\end{proof}

\begin{mylemm}
\label{LemmaWeylbis}
For $g\in \mathcal{ B} \mathcal{M} (\bar X)$, we have 
\beq
\label{gradientweyl}
\int |\nabla W|^2 = + \int 3{W_{ms}}^{ij} {W_{ij}}^{kl} {W^{ms}}_{kl} - \frac{1}{2} R |W|^2 + 2W_{ikjl}  E^{ij} E^{kl}  -8\oint \langle \hat{E}, S\rangle.
\eeq
\end{mylemm}

\begin{myrema}
In \cite[formula (3.8)]{CGY-IHES}, for the corresponding formula on Bach flat manifold without boundary, we have a 
sophisticated computation which expresses the first term in the right hand side of \ref{gradientweyl} in terms of the determinant of the self dual part and non-self dual part of  the Weyl tensor treated as operator on 2 forms. For our purpose in this paper we will just use the expression presented here.
\end{myrema} 
 
\begin{proof}

Recall the Bach flat condition can also be written in terms of the Weyl curvature and $E$ term as:
\beq
\label{bachflatweyl}
0=\nabla^k \nabla^l W_{ikjl}+ W_{ikjl} A^{kl},
\eeq
Taking $A^{ij}$ as a test function, we get
\begin{align*}
\int \nabla^l W_{ikjl} \nabla^k A^{ij}=\int W_{ikjl} A^{kl} 
A^{ij}-\oint \nabla^l W_{i0jl}  A^{ij}.
\end{align*}
It is known that (see \cite{CGY-IHES})
\begin{align*}
(\delta W)_{ijl}:=\nabla^m W_{ijml}=(\nabla_i A_{jl}-\nabla_j A_{il}):=C_{lji}
\end{align*}
where $C_{jil}$ is the Cotton tensor. Thus, by the symmetry of Weyl tensor, we get
\begin{align*}
\nabla^l W_{ikjl} \nabla^k A^{ij}=\frac{1}{2} \nabla^l W_{ikjl} (\nabla^k A^{ij} - \nabla^i A^{kj})= \frac{1}{2}  (-(\delta W)_{ikj})(-(\delta W)^{ikj})= \frac{1}{2} |\delta W|^2
\end{align*}
Thanks of Lemmas \ref{order0} and \ref{Weylorder1} and using the facts that $S$ is of trace free and $W$ vanishes on the boundary, we have on the boundary $\p X$
\begin{align*}
\nabla^l W_{i0jl}  A^{ij}= \nabla^0 W_{i0j0}  A^{ij} =\nabla^0 W_{\alpha0\beta0}  A^{\alpha\beta}= 2\nabla^0 W_{\alpha0\beta0}  \hat{E}^{\alpha\beta}=2\langle S, \hat{E}\rangle.
\end{align*}
Since $ \int  W_{ikjl}  A^{ij} A^{kl} = \frac {1}{4}  \int  W_{ikjl} E ^{ij} E^{kl}$, due to trace free property of the Weyl tensor, we have
\beq
\label{eq9.5}
\int |\delta W|^2 = -2\oint \langle \hat{E}, S\rangle + \frac{1}{2} \int  W_{ikjl}  E^{ij} E^{kl}.
\eeq

The derivation of  the rest of the formula in (\ref{gradientweyl})  follows the same line of proof as the derivation of the formula
(3.8) in \cite{CGY-IHES} , we will skip some details here.

\end{proof}
\vskip .5 in

\end{document}